\documentclass[english]{article}
\usepackage{custom_tex}
\usepackage{diagbox}
\begin{document}
\definecolor{ed}{RGB}{225,0,0}
\newcommand{\ed}[1]{\textcolor{ed}{[ED: #1]}}

\title{Distributed linear regression by averaging}

\date{\today}
\author{Edgar Dobriban\footnote{Wharton Statistics Department, University of Pennsylvania. E-mail: \texttt{dobriban@wharton.upenn.edu}.} \, and Yue Sheng\footnote{Graduate Group in Applied Mathematics and Computational Science, Department of Mathematics, University of Pennsylvania. E-mail: \texttt{yuesheng@sas.upenn.edu}.}}

\maketitle

\abstract{ Distributed statistical learning problems arise commonly when dealing with large datasets. In this setup, datasets are partitioned over machines, which compute locally, and communicate short messages. Communication is often the bottleneck. In this paper, we study one-step and iterative \emph{weighted parameter averaging} in statistical \emph{linear models} under \emph{data parallelism}. We do linear regression on each machine, send the results to a central server, and take a weighted average of the parameters. Optionally, we iterate, sending back the weighted average and doing local ridge regressions centered at it. How does this work compared to doing linear regression on the full data?  Here we study the performance loss in \emph{estimation}, \emph{test error}, and \emph{confidence interval length} in high dimensions, where the number of parameters is comparable to the training data size. 

We find the performance loss in one-step weighted averaging, and also give results for iterative averaging. We also find that different problems are affected differently by the distributed framework. Estimation error and confidence interval length increase a lot, while prediction error increases much less. We rely on recent results from random matrix theory, where we develop a new calculus of deterministic equivalents as a tool of broader interest.
}

\section{Introduction}

\begin{figure}
\centering
\includegraphics[scale=0.5]{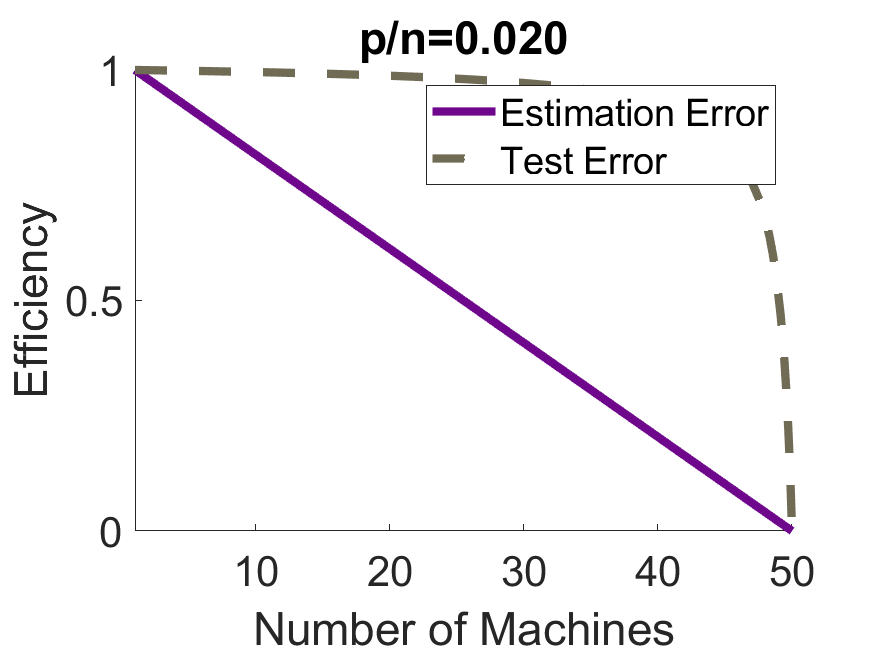}
\caption{How much accuracy do we lose in distributed regression? The plots show the relative efficiency, i.e., the ratio of errors, of the global least squares (OLS) estimator, compared to the distributed estimator averaging the local least squares estimators. This efficiency is at most unity, because the global estimator is more accurate. If the efficiency is close to unity, then averaging is accurate. We show the behavior of estimation and test error, as a function of number of machines. We see that estimation error is \emph{much more affected} than test error. The specific formulas are given in Table \ref{ef-tab}.}
\label{ef-mp}
\end{figure}

Datasets are constantly increasing in size and complexity. This leads to important challenges for practitioners. Statistical inference and machine learning, which used to be computationally convenient on small datasets, now bring an enormous computational burden. 

\emph{Distributed computation} is a universal approach to deal with large datasets. Datasets are partitioned across several machines (or workers). The machines perform computations locally and communicate only small bits of information with each other. They coordinate to compute the desired quantity. This is the standard approach taken at large technology companies, which routinely deal with huge datasets spread over computing units. What are the best ways to divide up and coordinate the work?

The same problem arises when the data is distributed due to privacy, security, or ethical concerns. For instance, medical and healthcare data is typically distributed across hospitals or medical units. The parties agree that they want to aggregate the results. At the same time, they do not want other parties access their data. How can they compute the desired aggregates, without sharing the data?
%Database stored locally, not shared due to privacy, security, ethical reasons.

In both cases, the key question is how to do statistical estimation and machine learning in a distributed setting. And what performance can the best methods achieve? This is a question of broad interest, and it is expected that the area of distributed estimation and computation will grow even more in the future. 

In this paper, we develop precise theoretical answers to fundamental questions in distributed estimation. We study \emph{one-step and iterative parameter averaging} in statistical \emph{linear models} under \emph{data parallelism}. Specifically, suppose in the simplest case that we do linear regression (Ordinary Least Squares, OLS) on each subset of a dataset distributed over $k$ machines, and take an optimal weighted average of the regression coefficients. How do the statistical and predictive properties of this estimator compare to doing OLS on the full data? 

We study the behavior of several learning and inference problems, such as \emph{estimation error}, \emph{test error} (i.e., out-of-sample \emph{prediction error}), and \emph{confidence intervals}. We also consider a high-dimensional (or proportional-limit) setting where the number of parameters is of the same order as the number of total samples (i.e., the size of the training data). We also study an analogous iterative algorithm, where we do local ridge regressions, take averages of the parameters on a central machine, send back the update to the local machines, and then again do local ridge, but where the penalty is centered around the previous mean. Our iterative algorithm falls between several classical methods such as ADMM and DANE, and we discuss connections.

We discover the following key phenomena, some of which are surprising in the context of existing work: 

\begin{enumerate}
\item{\bf Sub-optimality.} One-step averaging is not optimal (even with optimal weights), meaning that it leads to a performance decay. In contrast to some recent work (see the related work section), we find that there is a clear performance loss due to one-step averaging \emph{even if we split the data only into two subsets}. This loss is because the number of parameters is of the same order as the sample size. However, we can quantify this loss precisely. 
\item {\bf Strong problem-dependence.} Different learning and inference problems are affected differently by the distributed framework. Specifically, \emph{estimation error and the length of confidence intervals increases a lot, while prediction error increases less}. The intuition is that prediction is a noisy task, and hence the extra error incurred is relatively smaller. %This phenomenon was apparently not noticed before.

\item{\bf Simple form and universality.} The asymptotic efficiencies for one step distributed learning have simple forms that are often \emph{universal}. Specifically, they do not depend on the covariance matrix of the data, or on the sample sizes on the local machines. For instance, the estimation efficiency  \emph{decreases linearly in the number of machines $k$} (see Figure \ref{ef-mp} and Table \ref{ef-tab}).

\item{\bf Iterative parameter averaging has benefits.} We show that simple iterative parameter averaging mechanisms can reduce the error efficiently. We also exhibit computation-statistics tradeoffs: depending on the hyperparameters, we can converge fast to statistically suboptimal solutions; or vice versa.

\end{enumerate}

While there is already a lot of work in this direction (see Section \ref{relw}) our results are new and complementary. The key elements of novelty of our setting are: (1) The sample size and the dimension are comparable, and we do not assume sparsity. %Hence the problems we study are \emph{essentially high-dimensional}. There is very little work in this direction. 
(2) We have a new mathematical approach, using recent results from asymptotic random matrix theory such as \citep{rubio2011spectral}. Our approach also develops a novel theoretical tool, the \emph{calculus of deterministic equivalents}, and we illustrate how it can be useful in other problems as well. (3) We consider several accuracy metrics (estimation, prediction) in a unified framework of so-called general linear functionals.

%\subsection{Summary of our results}

\begin{table}[]
\renewcommand{\arraystretch}{2}
\centering
\caption{Estimation, Confidence Interval, and test efficiency as a function of number of machines $k$, the sample size $n$, and the dimension $p$. This is how much smaller the error of the global estimator is compared to the distributed estimator. These functions are plotted and described in Figure \ref{ef-mp}.}
\label{ef-tab}
\begin{tabular}{|c|c|}
\hline
Quantity & Relative efficiency ($n,p,k$)\\ \hline
Estimation \& CIs &\(\displaystyle \frac{n-kp}{n-p} \)  \\ \hline
Test error & \(\displaystyle  \frac{1}{1+\frac{p^2(k-1)}{n(n-kp)}}\) \\ \hline
\end{tabular}
\end{table}

The code for our paper is available at \url{http://www.github.com/dobriban/dist}.

\section{Some related work}
\label{relw}
%$x_{k}$
In this section we discuss some related work. 
There is a great deal of work in computer science and optimization on parallel and distributed computation \citep[see e.g.,][]{bertsekas1989parallel,boyd2011distributed,bekkerman2011scaling}. 
In addition, there are several popular examples of distributed data processing frameworks: for instance MapReduce \citep{dean2008mapreduce} and Spark \citep{zaharia2010spark}.

In contrast, there is less work on understanding the statistical properties, and the inherent computation-statistics tradeoffs, in distributed computation environments. This area has attracted increasing attention only in recent years, see for instance \cite{mcdonald2009efficient,zhang2012communication,zhang2013communication, zhang2013divide, duchi2014optimality, zhang2015divide, braverman2016communication, jordan2016communication, rosenblatt2016optimality, smith2016cocoa, fan2017distributed, lin2017distributed, lee2017communication, battey2018distributed, zhu2018distributed}, and the references therein. See \cite{huo2018aggregated} for a review. We can only discuss the most closely related papers due to space limitations. 

\cite{zinkevich2009slow} study the parallelization of SGD for learning, by reducing it to the study of delayed SGD; giving positive results for low latency "multicore" settings. They give an insightful discussion of the impact of various computational platforms, such as shared memory architectures, clusters, and grid computing. \cite{mcdonald2009efficient} propose averaging methods for special conditional maximum entropy models, showing variance reduction properties.  \cite{zinkevich2010parallelized} expand on this, proposing "parallel SGD" to average the SGD iterates computed on random subsets of the data. Their proof is based on the contraction properties of SGD.

\cite{zhang2013communication} bound the leading order term for MSE of averaged estimation in empirical risk minimization. Their bounds do not explicitly take dimension into account. However, their empirical data example clearly has large dimension $p$, considering a logistic regression with sample size $n = 2.4\cdot 10^8$, and $p=740,000$, so that $n/p \approx 340$. In their experiments, they distribute the data over up to 128 machines. So, our regime, where $k$ is of the same order as $n/p$, matches well their simulation setup. In addition, their concern is on regularized estimators, where they propose to estimate and reduce bias by subsampling.

\cite{liu2014distributed} study distributed estimation in statistical exponential families, connecting the efficiency loss from the global setting to the deviation from full exponential families. They also propose nonlinear KL-divergence-based combination methods, which can be more efficient than linear averaging.

  \cite{zhang2015divide} study divide and conquer kernel ridge regression, showing that the partition-based estimator achieves the statistical minimax rate over all estimators. Due to their generality, their results are more involved, and also their dimension is fixed. \cite{lin2017distributed} improve those results.
\cite{duchi2014optimality} derive minimax bounds on distributed estimation where the number of bits communicated is controlled.

\cite{rosenblatt2016optimality} consider the distributed learning problem in three different settings. The first two settings are fixed dimensional. The third setting is high-dimensional M-estimation, where they study the first order behavior of estimators using prior results from \cite{donoho2013high,el2013robust}. This is possibly the most closely related work to ours in the literature. They use the following representation, derived in the previous works mentioned above: a high-dimensional $M$-estimator can be written as $\hbeta = \beta + r(\gamma) \Sigma^{-1/2}\zeta (1+o_P(1))$, where $\zeta \sim \N(0,I_p/p)$, $\gamma$ is the limit of $p/n$, and $ r(\gamma)$ is a constant depending on the loss function, whose expression can be found in \cite{donoho2013high,el2013robust}.

They derive a relative efficiency formula in this setting, which for OLS takes the form
$$\frac{\E\|\hbeta_{dist}-\beta\|^2}{\E\|\hbeta-\beta\|^2} 
= 1 + \gamma (1-1/k)+O(\gamma^2).$$
In contrast, our result for this case (Theorem \ref{re_ell}) is equal to 
$$\frac{1-\gamma}{1-k\gamma}
= 1 + \gamma \frac{k-1}{1-k\gamma}.$$ 
Thus, our result is much more precise, and in fact exact, while of course being limited to the special case of linear regression.

In a heterogeneous data setting, \cite{zhao2016partially} fit partially linear models, and estimate the common part by averaging.
For model selection problems in GLM, \cite{chen2014split} propose weighted majority voting methods. \cite{lee2017communication} study sparse linear regression, showing that averaging debiased lasso estimators can achieve the optimal estimation rate if the number of machines is not too large. \cite{battey2018distributed} study a similar problem, also including hypothesis testing under more general sparse models. \cite{shi2018massive,banerjee2019divide} show that in problems with non-standard rates, averaging can lead to improved pointwise inference, while decreasing performance in a uniform sense. \cite{volgushev2019distributed} (among other contributions) provide conditions under which averaging quantile regression estimators have an optimal rate. \cite{banerjee2018removing} propose improvements based on communicating smoothed data, and fitting estimators after. \cite{szabo2018adaptive} study estimation methods under communication constraints in nonparametric random design regression model, deriving both minimax lower bounds and optimal methods.

See Section \ref{multi} for more discussion of multi-round methods.

\section{One-step weighted averaging: General linear functionals}
\label{gen_fw}

We consider the standard linear model
\beqs
Y = X\beta+\ep.
\eeqs
Here we have an outcome variable $y$ along with some $p$ covariates $x=(x^{1},\ldots,x^{p})^\top$, and want to understand their relationship. We observe $n$ such data points, arranging their outcomes into the $n\times 1$ vector $Y$, and their covariates into the $n\times p$ matrix $X$. We assume that $Y$ depends linearly on $X$, via some unknown $p\times 1$ parameter vector $\beta$.

We assume there are more samples than training data points, i.e., $n>p$, while $p$ can also be large. In that case, a gold standard is the usual least squares estimator (ordinary least squares or OLS)
\beqs
\hbeta = (X^\top X)^{-1}X^\top Y.
\eeqs
We also assume that the coordinates of the noise $\ep$ are uncorrelated and have variance $\sigma^2$.

Suppose now that the samples are distributed across $k$ machines (these can be real machines, but they can also be---say---sites or hospitals in medical applications, or mobile devices in federated learning). The $i$-th machine has the $n_i\times p$ matrix $X_i$, containing $n_i$ samples, and also the $n_i \times 1$ vector $Y_i$ of the corresponding outcomes for those samples. Thus, the $i$-th worker has access to only a subset of training $n_i$ data points out of the total of $n$ training data points. For instance, if the data points denote $n$ users, then they may be partitioned into $k$ sets based on country of residence, and we may have $n_1$ samples from the United States on one server, $n_2$ samples from Canada on another server, etc. The broad question is: How can we estimate the unknown regression parameter $\beta$ if we need to do most of the computations locally?

Let us write the partitioned data as 
$$X = 
\begin{bmatrix}
    X_1 \\
    \hdots \\
    X_k
\end{bmatrix}
,\,\,
Y = 
\begin{bmatrix}
    Y_1 \\
    \hdots \\
    Y_k
\end{bmatrix}.
$$
We also assume that each \emph{local} OLS estimator $\hat \beta_i = (X_i^\top X_i)^{-1}X_i^\top Y_i$ is well defined, which requires that the number of local training data points $n_i$ must be at least $p$ on each machine (so $n_i\ge p$). We first consider combining the local OLS estimators at a parameter server via one-step weighted averaging. Since they are uncorrelated and unbiased for $\beta$, we consider unbiased weighted estimators 
$$\hbeta_{dist}(w) = \sum_{i=1}^k w_i \hbeta_i$$
with $\sum_{i=1}^k w_i=1$.

We introduce a "general linear functional" framework to study learning tasks such as estimation and prediction in a unified way. In the general framework, we predict \emph{linear functionals} of $\beta$ of the form
$$L_A = A \beta+Z.$$
Here $A$ is a fixed $d\times p$ matrix, and $Z$ is a zero-mean Gaussian noise vector of dimension $d$, with covariance matrix $\Cov{Z} = h\sigma^2 I_d$, for some scalar parameter $h\ge 0$. We denote the covariance matrix between $\ep$ and $Z$ by $N$, so that $\Cov{\ep,Z} = N$. If $h=0$, we say that there is no noise. In that case, we necessarily have $N=0$.

We predict the linear functional $L_A$ via plug-in based on some estimator $\hbeta_0$ (typically OLS or distributed OLS)
$$\hat L_A(\hbeta_0)  = A \hbeta_0.$$
We measure the quality of estimation by the mean squared error 
\begin{align*}
M(\hbeta_0) = \E\| L_A  - \hat L_A(\hbeta_0)\|^2.
\end{align*}

We compute the \emph{relative efficiency} of OLS $\hbeta$ compared to a weighted distributed estimator $\hbeta_{dist} = \hbeta_{dist}(w)$:
$$E(A,d;X_1,\ldots,X_k) 
:= \frac{M(\hbeta)}{M(\hbeta_{dist})}.
$$

The relative efficiency is a fundamental quantity, giving the loss of accuracy due to distributed estimation. %It is of great interest to understand the behavior of this quantity. 

\subsection{Examples}
\label{examples}

We now show how several learning and inference problems fall into the general framework. See Table \ref{frame} for a concise summary. %In addition to parameter estimation, we will discuss out-of-sample prediction (test error), in-sample prediction (training error), and confidence intervals. 

\begin{table}[]
\renewcommand{\arraystretch}{1.5}
\centering
\caption{A general framework for finite-sample efficiency calculations. The rows show the various statistical problems studied in our work, namely estimation, confidence interval formation, in-sample prediction, out-of-sample prediction and regression function estimation. The elements of the row show how these tasks fall in the framework of linear functional prediction described in the main body. }
\label{frame}
\begin{tabular}{|l|l|l|l|l|l|}
\hline
Statistical learning problem  & $L_A$ & $\hat L_A$  & $A$   & $h$ & $N$  \\ \hline
Estimation          & $\beta$  & $\hbeta$  & $I_p$      & 0   &0  \\ \hline
Regression function estimation & $X\beta$ & $X\hbeta$   & $X$   & 0   &0  \\ \hline
Confidence interval  & $\beta_j$ & $\hbeta_j$ & $E_{j}^\top$   & 0     &0 \\ \hline
Test error & $x_t^\top \beta+ \ep_t$ & $x_t^\top \hbeta$ & $x_t^\top$ & 1  &0 \\ \hline
Training error  & $X\beta+ \ep$ & $X\hbeta$ & $X$ & $1$   & $\sigma^2 I_n$ \\ \hline
\end{tabular}
\end{table}

\bitem 
\item  {\bf Parameter estimation}. In parameter estimation, we want to estimate the regression coefficient vector $\beta$ using $\hbeta$. This is an example of the general framework by taking $A=I_p$, and without noise (so that $h=0$).

\item {\bf Regression function estimation}. We can use $X\hbeta$ to estimate the regression function $\E(Y|X)=X\beta$. In this case, the transform matrix is $A=X$, the linear functional is $L_A=X\beta$, the predictor is $\hat{L}_A=X\hbeta$, and there is no noise. %We consider the mean squared error:
%$$
%\E\|X\beta-X\hbeta\|^2=\E\|X(\beta-\hbeta)\|^2.
%$$

\item  {\bf Out-of-sample prediction (Test error)}.
For out-of-sample prediction, or test error, we consider a test data point $(x_t,y_t)$, generated from the same model $y_t=x_t^\top \beta+\ep_t$, where $x_t,\ep_t$ are independent of $X,\ep$, and only $x_t$ is observable. We want to use $x_t^\top \hbeta$ to predict $y_t$. 

This corresponds to predicting the linear functional $L_{x_t} = x_t^\top \beta+\ep_t$, so that $A = x_t^\top$, and the noise is $ Z=\ep_t$, which is uncorrelated with the noise $\ep$ in the original problem. 

\item {\bf In-sample prediction (Training error)}. 
For in-sample prediction, or training error, we consider predicting the response vector $Y$, using the model fit $X\hat{\beta}$. Therefore, the functional $L_A$ is
$L_A = Y = X\beta + \ep.$ This agrees with regression function estimation, except for the noise $Z=\ep$, which is identical to the original noise. Hence, the noise scale is $h = 1$, and $N  = \Cov{\ep,Z} = \sigma^2 I_n$.

\item {\bf Confidence intervals}.
To construct confidence intervals for individual coordinates, we consider the normal model $Y \sim \N(X\beta, \sigma^2 I_n)$. Assuming $\sigma^2$ is known, a confidence interval with coverage $1-\alpha$ for a given coordinate $\beta_j$ is
$$\hbeta_j \pm \sigma z_{\alpha/2} V_j^{1/2},$$
where $z_\alpha = \Phi^{-1}(\alpha)$ is the inverse normal CDF, and $V_j$ is the $j$-th diagonal entry of $(X^\top X)^{-1}$.

Therefore, we can measure the difficulty of the problem by $V_j$. The larger $V_j$ is, the longer the confidence interval. This also measures the difficulty of estimating the coordinate $L_A = \beta_j$. This can be fit in our general framework by choosing $A = E_j^\top$, the $1\times p$ vector of zeros, with only a one in the $j$-th coordinate. This problem is noiseless. In this sense, the problem of confidence intervals is the same as the estimation accuracy for individual coordinates of $\beta$.

If $\sigma$ is not known, then we we first need to estimate it in a distributed way. This is an interesting problem in itself, but beyond the scope of our current work.

\eitem

\subsection{Finite sample results}
\label{gen_fw_fs} 
We  now show how to calculate the efficiency explicitly in the general framework. We start with the simpler case where $h=0$. We then have for the OLS estimator 
$$M(\hbeta) = \sigma^2\cdot \tr \left[ (X^\top X)^{-1} A^\top A\right].$$
For the distributed estimator with weights $w_i$ summing to one, given by $\hbeta_{dist}(w) = \sum_i w_i \hbeta_i$, we have 
$$M(\hbeta_{dist}) 
= \sigma^2\cdot \left(\sum_{i=1}^k w_i^2 \cdot \tr \left[ (X_i^\top X_i)^{-1} A^\top A\right]\right ).$$ 
Using a simple Cauchy-Schwarz inequality (see Section \ref{pf:re_ols} for the argument for parameter estimation), we find that the optimal efficiency, for the optimal weights, is
\begin{align}
\label{opt_eff}
E(A; X_1,\ldots,X_k)
%=\frac{ \tr \left[ (X^\top X)^{-1} A^\top A\right]}
%{\frac{1}{\sum_{i=1}^k \frac{1}{\tr \left[ (X_i^\top X_i)^{-1} A^\top A\right]}}}
=\tr \left[ (X^\top X)^{-1} A^\top A\right]\cdot\sum_{i=1}^k \frac{1}{\tr \left[ (X_i^\top X_i)^{-1} A^\top A\right]}
.
\end{align}

This shows that the key to understanding the efficiency are the traces $\tr \left[ (X_i^\top X_i)^{-1} A^\top A\right].$ 
Proving that the efficiency is at most unity turns out to require the concavity of the matrix functional $1/\tr(X^{-1}A^\top A)$. This is a consequence of classical results in convex analysis, see for instance \cite{davis1957all,lewis1996convex}. For completeness, we give a short self-contained proof in Section \ref{pf:ccv_re2} of the Appendix. 

\begin{proposition}[Concavity for general efficiency, \cite{davis1957all,lewis1996convex}]
\label{ccv_re2}
The function $f(X)=1/\tr(X^{-1}A^\top A)$ is a concave function defined on positive definite matrices. As a consequence, the general relative efficiency for distributed estimation is at most unity for any matrices $X_i$:
$$E(A; X_1,\ldots,X_k)\le 1.$$
\end{proposition}

 For the more general case when $h\neq 0$, we can also find the OLS MSE as 
\begin{align*}
M(\hbeta)  &= \sigma^2\cdot \left[\tr\left((X^\top X)^{-1} A^\top A\right) 
- 2 \tr\left( A (X^\top X)^{-1} X^\top N\right)
+ hd\right].
\end{align*}
For the distributed estimator, we can find, denoting $N_i:=\Cov{ \ep_i, Z}/\sigma^2$,  
$$M(\hbeta_{dist}) 
= \sigma^2\cdot \left(\sum_{i=1}^k w_i^2 \cdot \tr \left[ (X_i^\top X_i)^{-1} A^\top A\right]
- 2 w_i \cdot \tr\left( A(X_i^\top X_i)^{-1} X_i^\top N_i\right)\right ) + \sigma^2 h d.$$ 
Let $a_i = \tr \left[ (X_i^\top X_i)^{-1} A^\top A\right]$, and $b_i $ $= $ $\tr\left( A(X_i^\top X_i)^{-1} X_i^\top N_i\right)$.  The optimal weights can be found from a quadratic optimization problem:  
$$w_i = \frac{\lambda^*+b_i}{a_i}, \,\, 
\lambda^*:= \frac{1-\sum_{i=1}^k\frac{b_i}{a_i}}{\sum_{i=1}^k\frac{1}{a_i}}.$$

The resulting formula for the optimal weights, and for the global optimum, can be calculated explicitly. The details can be found in the supplement (Section \ref{gen_w}).

\section{Calculus of deterministic equivalents}
\label{calc_det_equiv}

\subsection{A calculus of deterministic equivalents in RMT}
\label{calc_det_equiv_intro}

We saw that the relative efficiency depends on the trace functionals $\tr[(X^\top X)^{-1}$ $A^\top A]$, for specific matrices $A$. To find their limits, we will use the technique of \emph{deterministic equivalents} from random matrix theory. This is a method to find the almost sure limits of random quantities. See for example \cite{hachem2007deterministic,couillet2011deterministic} and the related work section below.

For instance, the well known Marchenko-Pastur (MP) law for the eigenvalues of random matrices \citep{marchenko1967distribution,bai2009spectral} states that the eigenvalue distribution of certain random matrices is asymptotically deterministic. More generally, one of the best ways to understand the MP law is that \emph{resolvents are asymptotically deterministic}. Indeed, let $\hSigma = n^{-1} X^\top X$, where $X = Z\Sigma^{1/2}$ and $Z$ is a random matrix with iid entries of zero mean and unit variance. Then the MP law means that for any $z$ with positive imaginary part, we have the equivalence
$$(\hSigma-z I)^{-1} \asymp (x_p\Sigma-z I)^{-1},$$
for a certain scalar $x_p=x(\Sigma,n,p,z)$ (that will be specified later). At this stage we can think of the equivalence entry-wise, but we will make this precise next. The above formulation has appeared in some early works by VI Serdobolskii, see e.g., \cite{serdobolskii1983minimum}, and Theorem 1 on page 15 of \cite{serdobolskii2007multiparametric} for a very clear statement. 

The consequence is that simple linear functionals of the random matrix $(\hSigma-z I)^{-1}$ have a deterministic equivalent based on $(x_p\Sigma-z I)^{-1}$.  In particular, we can approximate the needed trace functionals by simpler deterministic quantities. For this we will take a principled approach and define some appropriate notions for a \emph{calculus of deterministic equivalents}, which allows us to do calculations in a simple and effective way.

First, we make more precise the notion of equivalence. We say that the (deterministic or random) not necessarily symmetric matrix sequences $A_n, B_n$ of growing dimensions are \emph{equivalent}, and write 
$$A_n \asymp B_n$$ if 
$$\lim_{n\to\infty}\left|\tr\left[C_n(A_n-B_n)\right]\right|=0$$ 
almost surely, for any sequence $C_n$ of not necessarily symmetric matrices with bounded trace norm, i.e., such that 
$$\lim\sup\|C_n\|_{tr}<\infty.$$ 

We call such a sequence $C_n$ a \emph{standard sequence}.
Recall here that the trace norm (or nuclear norm) is defined by $\|M\|_{tr}=\tr((M^\top M)^{1/2}) = \sum_i \sigma_i$, where $\sigma_i$ are the singular values of $M$. %Since the matrices considered here are real symmetric, we also have that $\|M\|_{tr} = \sum_i |\lambda_i|$, where $\rho_i$ are the eigenvalues of $M$.

\subsection{General MP theorem}
\label{gen_mp}
%\item 
To find the limits of the efficiencies, the most important deterministic equivalent will be the following result, essentially a consequence of the generalized Marchenko-Pastur theorem of \cite{rubio2011spectral} (see Section \ref{pf:gen_det} for the argument). %For a non-negative definite matrix $M$, $M^{1/2}$ denotes its non-negative definite square root. 
We study the more general setting of elliptical data. In this model the data samples may have different scalings, having the form $x_i = g_i^{1/2}\Sigma^{1/2} z_i$, for some vector $z_i$ with iid entries, and for datapoint-specific \emph{scale parameters} $g_i$. Arranging the data as the rows of the matrix $X$, that takes the form $$X = \Gamma^{1/2} Z \Sigma^{1/2},$$ where $Z$ and $\Gamma$ are as before: $Z$ has iid standardized entries, while $\Sigma$ is the covariance matrix of the features. Now $\Gamma$ is the diagonal \emph{scaling matrix} containing the scales $g_i$ of the samples. This model has a long history in multivariate statistics \citep[e.g.,][]{mardia1979multivariate}.

\begin{theorem}[Deterministic equivalent in elliptical models, consequence of \cite{rubio2011spectral}] Let the $n \times p$ data matrix $X$ follow the elliptical model
$$X = \Gamma^{1/2} Z\Sigma^{1/2},$$ 
where $\Gamma$ is an $n \times n$ diagonal matrix with non-negative entries representing the scales of the $n$ observations, and $\Sigma$ is a $p\times p$ positive definite matrix representing the covariance matrix of the $p$ features. Assume the following:

\begin{compactenum}
\item The entries of $Z$ are iid random variables with mean zero, unit variance, and finite $8+c$-th moment, for some $c>0$.
\item The eigenvalues of $\Sigma$, and the entries of $\Gamma$, are uniformly bounded away from zero and infinity. 
\item We have $n,p\to \infty$, with $\gamma_p = p/n$ bounded away from zero and infinity.
\end{compactenum}
Let $\hSigma = n^{-1} X^\top X$ be the sample covariance matrix. Then $\hSigma$ is equivalent to a scaled version of the population covariance
$$\hSigma^{-1} \asymp \Sigma^{-1} \cdot e_p.$$

Here $e_p = e_p(n,p,\Gamma)>0$ is the unique solution of the fixed-point equation 
\begin{align*}
1 &= \frac{1}{n} \tr\left[e_p \Gamma (I+\gamma_p e_p \Gamma)^{-1}\right].
\end{align*}
\label{gen_det}
\end{theorem}
Thus, the inverse sample covariance matrix has a deterministic equivalent in terms of a scaled version of the inverse population covariance matrix. This result does not require the convergence of the aspect ratio $p/n$, or of the e.s.d. of $\Sigma$, and $\Gamma$, as is sometimes the case in random matrix theory. However, if the empirical spectral distribution of the scales $\Gamma$ tends to $G$, the above equation has the limit
$$
\int\frac{se}{1+\gamma se}dG(s)=1.
$$
The usual MP theorem is a special case of the above result where $\Gamma = I_n$. As a result, we obtain the following corollary: 

\begin{corollary}[Deterministic equivalent in MP models] Let the $n \times p$ data matrix $X$ follow the model
$X = Z\Sigma^{1/2},$
where $\Sigma$ is a $p\times p$ positive definite matrix representing the covariance matrix of the $p$ features. Assume the same conditions on $\Sigma$ from Theorem \ref{gen_det}. Then $\hSigma$ is equivalent to a scaled version of the population covariance
$$\hSigma^{-1} \asymp \frac1{1-\gamma_p}\cdot  \Sigma^{-1}.$$
\label{mp_det}
\end{corollary}
The proof is immediate, by checking that $e_p = 1/(1-\gamma_p)$ in this case.

% An even more special case of the above results is the isotropic case when $\Sigma=I_p$. The above equation for $x = x_p$ reads $1-x  =\gamma_p[1+z/(x-z)]=\gamma_p x/(x-z),$
% and the conclusion is $(\hSigma-z I)^{-1} \asymp I_p/(x- z).$ It is easily checked that this is indeed the correct formula.  

\subsubsection{Related work on deterministic equivalents}

There are several works in random matrix theory on deterministic equivalents. One of the early works is \cite{serdobolskii1983minimum}, see \cite{serdobolskii2007multiparametric} for a modern summary. The name "deterministic equivalents" and technique was more recently introduced and re-popularized by \cite{hachem2007deterministic} for signal-plus-noise matrices. Later \cite{couillet2011deterministic} developed deterministic equivalents for matrix models of the type $\sum_{k=1}^B R_k^{1/2} X_k T_k X_k^\top R_k^{1/2}$, motivated by wireless communications. See the book \cite{couillet2011random} for a summary of related work. See also \cite{muller2016random} for a tutorial. However, many of these results are stated only for some fixed functional of the resolvent, such as the Stieltjes transform. One of our points here is that there is a much more general picture. 

%See \cite{dobriban2015high} for recent applications to statistical learning.

\cite{rubio2011spectral} is one of the few works that explicitly states more general convergence of arbitrary trace functionals of the resolvent. Our results are essentially a consequence of theirs.

However, we think that it is valuable to define a set of rules, a "calculus" for working with deterministic equivalents, and we use those techniques in our paper. Similar ideas for operations on deterministic equivalents have appeared in \cite{peacock2008eigenvalue}, for the specific case of a matrix product. Our approach is more general, and allows many more matrix operations, see below. 

\subsection{Rules of calculus}

The calculus of deterministic equivalents has several properties that simplify calculations. We think these justify the name of \emph{calculus}. Below, we will denote by $A_n,B_n,C_n$ etc, sequences of deterministic or random matrices. See Section \ref{pf:calcrules} in the supplement for the proof. 

\begin{theorem}[Rules of calculus]
\label{calcrules}
The calculus of deterministic equivalents has the following properties. 

\begin{compactenum}
\item {\bf Equivalence.} The $\asymp$ relation is indeed an equivalence relation. 

\item {\bf Sum.} If $A_n\asymp B_n$ and $C_n\asymp D_n$, then $A_n+C_n\asymp B_n+D_n$.
\item {\bf Product.} If $A_n$ is a sequence of matrices with bounded operator norms i.e., $\| A_n\|_{op}<\infty$, and $B_n\asymp C_n$, then $A_nB_n\asymp A_nC_n$.

\item {\bf Trace.} If $A_n\asymp B_n$, then $\tr\{n^{-1}A_n\}- \tr\{n^{-1}B_n\} \to 0$ almost surely.

\item {\bf Stieltjes transforms.} As a consequence, if $(A_n-zI_n)^{-1} \asymp (B_n-zI_n)^{-1} $ for symmetric matrices $A_n,B_n$, then $m_{A_n}(z)-  m_{B_n}(z) \to 0$ almost surely. Here $m_{X_n}(z) = n^{-1} \tr (X_n-zI_n)^{-1}$ is the Stieltjes transform of the empirical spectral distribution of $X_n$.
\end{compactenum}
\end{theorem}  

In addition, the calculus of deterministic equivalents has additional properties, such as continuous mapping theorems, differentiability, etc. We have developed the differentiability in the follow-up work \citep{dobriban2019one}.

We also briefly sketch several applications of the calculus of deterministic equivalents in Section \ref{apps_calc} in the supplement, to studying the risk of ridge regression in high dimensions, including in the distributed setting, gradient flow for least squares, interpolation in high dimensions, heteroskedastic PCA, as well as exponential family PCA. We emphasize that in each case, including for the formulas of asymptotic efficiencies in the current work, there are other proof techniques, but they tend to be more case-by-case. The calculus provides a unified set of methods, and separate results can be seen as applications of the same approach.

\section{Examples}

We now use the calculus of deterministic equivalents to find the limits of the trace functionals in our general framework. We study each problem in turn. For asymptotics, we consider as before elliptical models. The data on the $i$-th machine takes the form $$X_i = \Gamma_i^{1/2} Z_i \Sigma^{1/2},$$ where $\Gamma_i$ contains the \emph{scales} of the $i$-th machine and $Z_i$ is the appropriate submatrix of $X$.

In this model, it turns out that the efficiencies can be expressed in a simple way via the $\eta$-transform \citep{tulino2004random}. The $\eta$-transform of a distribution $G$ is
$$\eta(x) = \E_G \frac{1}{1+x T},$$
 for all $x$ for which this expectation is well-defined. %Of course, the $\eta$-transform is simply equal to the Stieltjes transform up to a change of variables. However, many expressions take a simpler form in this parameterization. In particular, 
 We will see that the efficiencies can be expressed in terms of the functional inverse $f$ of the $\eta$-transform evaluated at the specific value $1-\gamma$: 
 \beq
f(\gamma, G) = \eta_{G}^{-1}(1-\gamma).
\label{inv_eta}
\eeq
%Here the inverse is taken in the sense of function inversion. 

We think of elliptical models where the limiting distribution of the scales $g_1,\ldots,g_n$ is $G$. For some insight on the behavior of $\eta$ and $f$, consider first the case when $G$ is a point mass at unity, $G = \delta_1$. In this case, all scales are equal, so this is just the usual Marchenko-Pastur model. Then, we have $\eta(x) = 1/(1+x)$, while $f(\gamma,G) = \gamma/(1-\gamma)$. See Figure \ref{plotfg} for the plots. The key points to notice are that $\eta$ is a decreasing function of $x$, with $\eta(0)=1$, and $\lim_{x\to\infty}\eta(x) = 0$. Moreover, $f$ is an increasing function on $[0,1]$ with $f(0)=0$, $\lim_{\eta\to1}f(\eta) = +\infty$.  The same qualitative properties hold in general for compactly supported distributions $G$ bounded away from $0$.

\begin{figure}
\begin{subfigure}{.45\textwidth}
  \centering
\includegraphics[scale=0.35]
{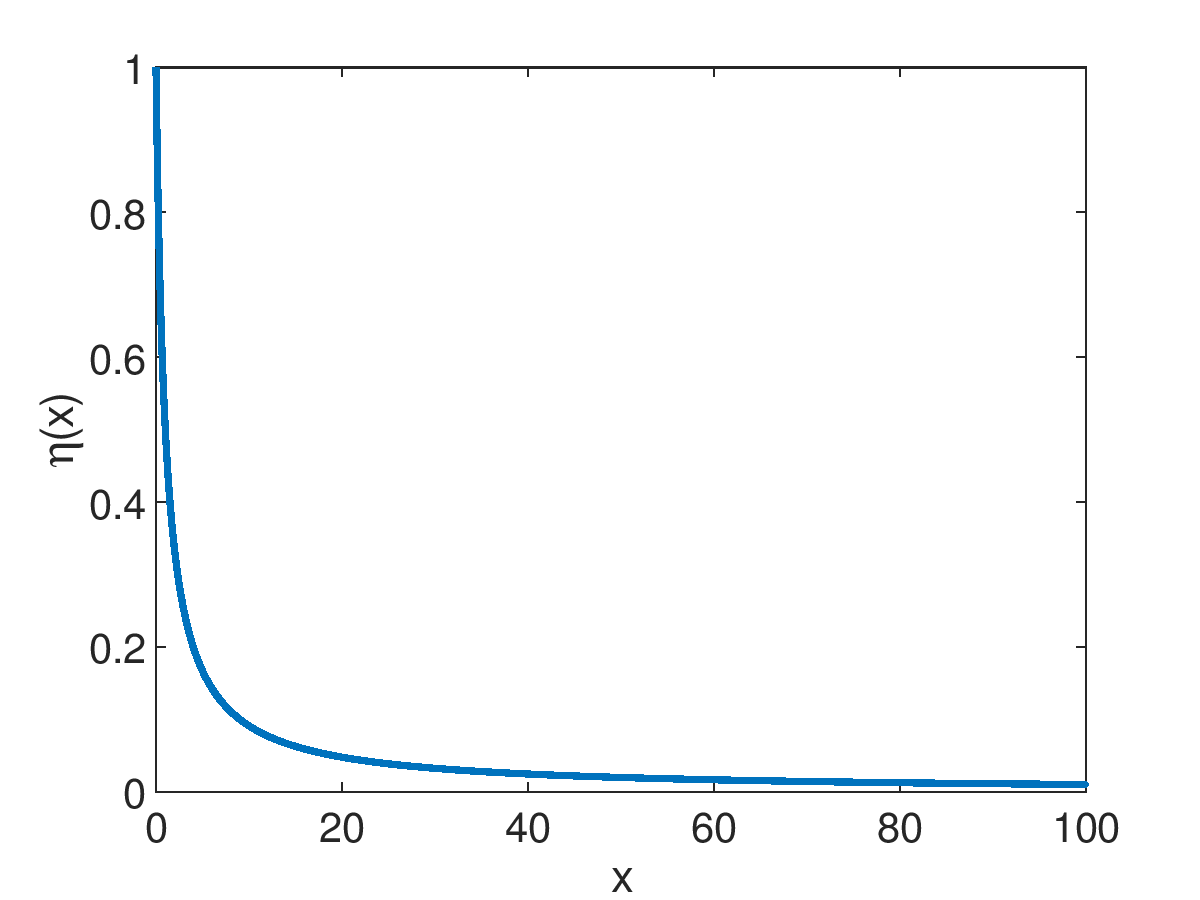}
\end{subfigure}
\begin{subfigure}{.45\textwidth}
  \centering
\includegraphics[scale=0.35]
{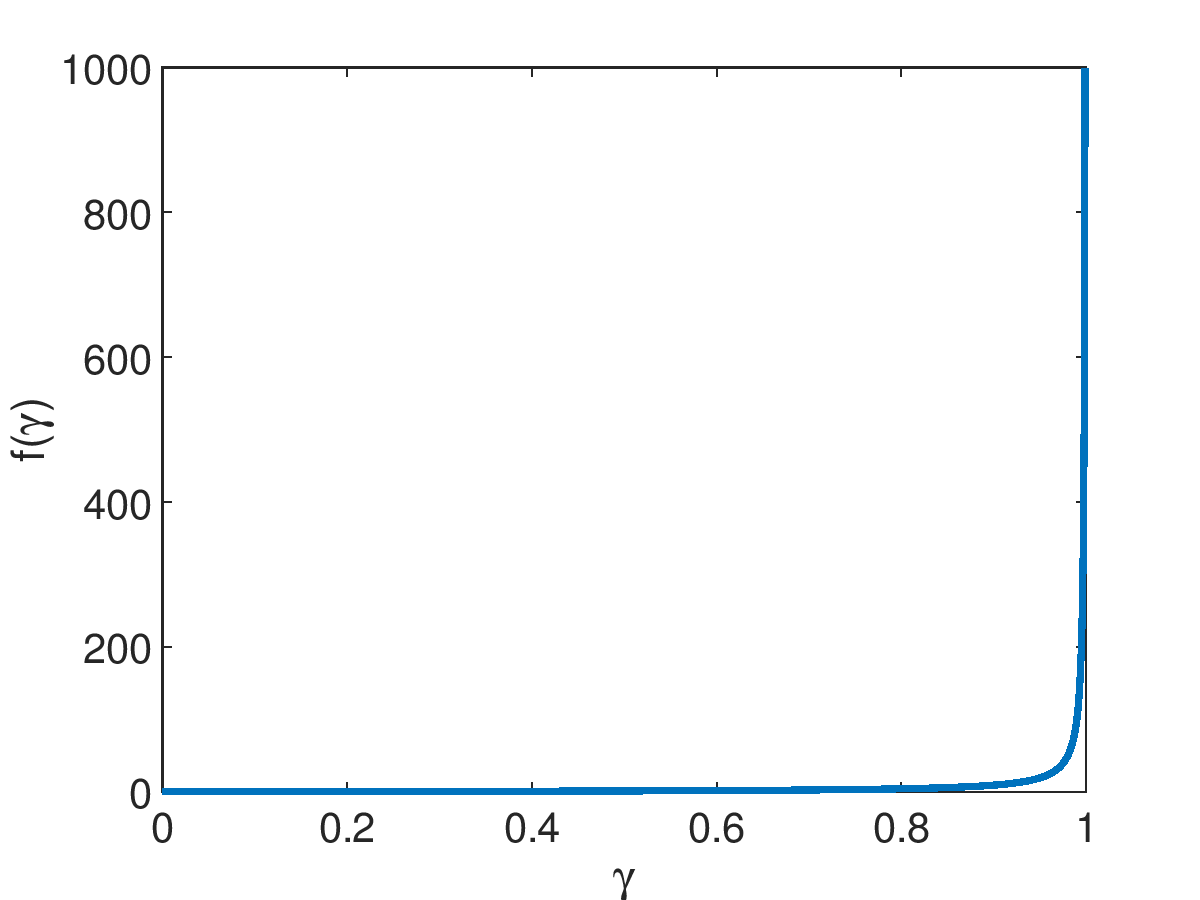}
\end{subfigure}
\caption{Plots of $\eta$ and $f$ for $G$ being the point mass at unity.}
\label{plotfg}
\end{figure}

\subsection{Parameter estimation}

For estimating the parameter, we have $\E\|\beta-\hbeta\|^2$ $=\sigma^2$ $\tr (X^\top X)^{-1}$. We find via \eqref{opt_eff} the estimation efficiency
$$
RE(X_1,\dots,X_k)=
\tr[(X^\top X)^{-1}]\cdot \left[\sum_{i=1}^k \frac{1}{\tr[(X_i^\top X_i)^{-1}]}\right].
$$
Recall that $X^\top X = \sum_{i=1}^k X_i^\top X_i$. 
Recall that the empirical spectral distribution (e.s.d.) of a symmetric matrix $M$ is simply the CDF of its eigenvalues (which are all real-valued). More formally, it is the discrete distribution $F_p$ that places equal mass on all eigenvalues of $M$.

\begin{theorem}[RE for elliptical and MP models]
\label{re_ell}
Under the conditions of Theorem \ref{gen_det}, suppose that, as $n_i\to\infty$ with $p/n_i\to\gamma_i\in(0,1)$, the $e.s.d.$ of $\Gamma$ converges weakly to some $G$, the  $e.s.d.$ of each $\Gamma_i$ converges weakly to some $G_i$, and that the $e.s.d.$ of $\Sigma$ converges weakly to $H$. Suppose that $H$ is compactly supported away from the origin, while $G$ is also compactly supported and does not have a point mass at the origin.
 Then, the RE has almost sure limit
\beqs
ARE =  f(\gamma, G) \cdot 
\sum_{i=1}^k  \frac{1} {f(\gamma_i, G_i)}.
\eeqs
For Marchenko-Pastur models, the RE has the form $(1/\gamma-k)/(1/\gamma-1).$
\end{theorem}

See Section \ref{estim_ellip_supp} in the supplement for the proof. For MP models, for any finite sample size $n$, dimension $p$, and number of machines $k$, we can approximate the ARE as
\beqs
%\label{are}
ARE \approx \frac{n-kp}{n-p}.
\eeqs 
This efficiency for MP models depends on a simple linear way on $k$.  We find this to be a surprisingly simple formula, which can also be easily computed in practice. Moreover, the formula has several more intriguing properties: 

\benum
\item The ARE \emph{decreases linearly} with the number of machines $k$. This holds as long as $ARE\ge 0$. At the threshold case $ARE=0$, there is a phase transition. The reason is that there is a singularity, and the OLS estimator is undefined for at least one machine.

However, we should be cautious about interpreting the linear decrease. For the root mean squared error (RMSE), the efficiency is the square root of the ARE above, and thus does not have a linear decrease.

\item The ARE has two important \emph{universality} properties. 
\benum 
\item 
First, it \emph{does not depend} on how the samples are distributed across the different machines, i.e., it is independent of the specific sample sizes $n_i$. 
\item Second, it \emph{does not depend} on the covariance matrix $\Sigma$ of the samples. This is in contrast to the estimation error of OLS, which does in fact depend on the covariance structure. Therefore, we think that the cancellation of $\Sigma$ in the ARE is noteworthy. 
\eenum 
\eenum

The ARE is also very accurate in simulations. See Figure \ref{f1} for an example. Here we report the results of a simulation where we generate an $n \times p$ random matrix $X$ such that the rows are distributed independently as $x_i \sim \N(0,\Sigma)$. We take $\Sigma$ to be diagonal with entries chosen uniformly at random between 1 and 2. We choose $n>p$, and for each value of $k$ such that $k<n/p$, we split the data into $k$ groups of a random size $n_i$. To ensure that each group has a size $n_i\ge p$, we first let $n_i^0=p$, and then distribute the remaining samples uniformly at random. We then show the theoretical results compared to the theoretical ARE. We observe that the two agree closely.  

\begin{figure}
\begin{subfigure}{.45\textwidth}
  \centering
\includegraphics[scale=0.48]{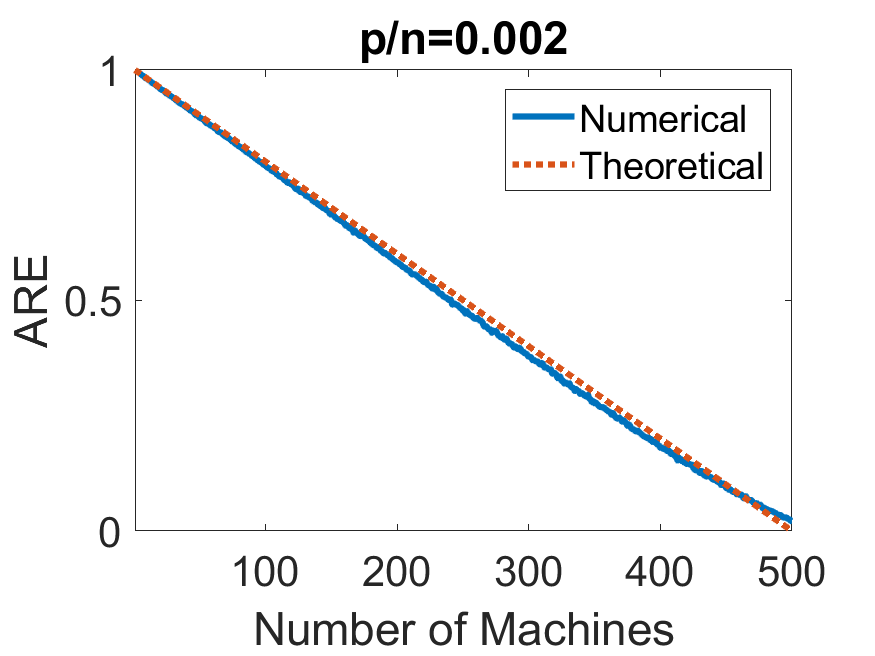}
\end{subfigure}
\begin{subfigure}{.45\textwidth}
  \centering
\includegraphics[scale=0.48]{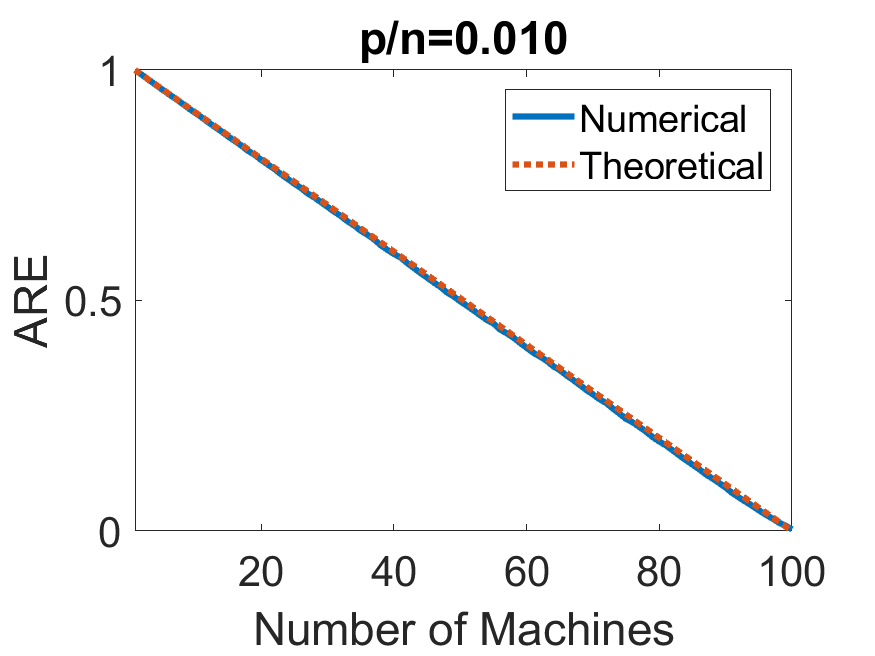}
\end{subfigure}
\caption{Comparison of empirical and theoretical ARE for standard sample covariance matrices. Left: $n=10,000$, $p=20$. Right: $n=10,000$, $p=100$. }
\label{f1}
\end{figure}

\subsection{Regression function estimation}

For estimating the regression function, we have $\E\|X(\beta-\hbeta)\|^2=\sigma^2p$. We then find via equation \eqref{opt_eff} the prediction efficiency
$$
FE(X_1,\dots,X_k)=\sum_{i=1}^k\frac{p}{\tr((X_i^\top X_i)^{-1}X^\top X))}.
$$
For asymptotics, we consider as before elliptical models.

\begin{theorem}[FE for elliptical and MP models]
\label{fe_ell}
Under the conditions of Theorems \ref{gen_det} and \ref{re_ell}, the FE has the almost sure limit
\begin{align*}
FE(X_1,\ldots,X_k)
\to_{a.s.}
\sum_{i=1}^k\frac{1}{1+\left(\frac{1}{\gamma}\E_GT-\frac{1}{\gamma_i}\E_{G_i}T\right)f(\gamma_i, G_i)}
.
\end{align*}
Under Marchenko-Pastur models, the conditions of Corollary \ref{mp_det}, the FE has the almost sure limit
$\frac{\gamma}{1-\gamma}\sum_{i=1}^k\frac{1-\gamma_i}{\gamma_i}.$
\end{theorem}

See Section \ref{pf:fe_ell} for the proof. This efficiency is more complex than that for estimation error; specifically it generally depends on the individual $\gamma_i$ and not just $\gamma$.

\subsection{In-sample prediction (Training error)}

For in-sample prediction, we start with the well known formula
$$\E||X(\beta-\hat{\beta})+\ep||^2=\sigma^2[n-\tr((X^\top X)^{-1}X^\top X]=\sigma^2(n-p).$$
As we saw, to fit in-sample prediction in the general framework, we need to take the transform matrix $A=X$, the noise $Z=\ep$, and the covariance matrices $N_i=\Cov{\ep_i,Z}/\sigma^2=\Cov{\ep_i,\ep}/\sigma^2$. Then, in the formula for optimal weights we need to take $a_i=\tr[(X_i^\top X_i)^{-1}X^\top X]$ and $b_i=\tr(X(X_i^\top X_i)^{-1}X_i^\top N_i)=\tr[(X_i^\top X_i)^{-1}X_i^\top N_iX]=\tr[(X_i^\top X_i)^{-1}X_i^\top X_i]=p$. Therefore, the optimal error for distributed regression is achieved by the weights 
$$
w_i=\frac{\lambda-b_i}{a_i}=\frac{\lambda-p}{a_i},~~\lambda=\frac{1-\sum_{i=1}^k\frac{b_i}{a_i}}{\sum_{i=1}^k\frac{1}{a_i}}=\frac{1}{\sum_{i=1}^k\frac{1}{a_i}}-p.
$$
Plugging these into $M(\hat{\beta}_{dist})$ given in the general framework, we find
$$
M(\hat{\beta}_{dist})=\sigma^2\left(n-2p+\frac{1}{\sum_{i=1}^k\frac{1}{a_i}}\right),~~a_i=\tr((X_i^\top X_i)^{-1}X^\top X).
%=p+\sum_{j\neq i}\tr((X_i^\top X_i)^{-1}X_j^\top X_j).
$$
Thus, the optimal in-sample prediction efficiency is 
$$IE(X_1,\dots, X_k)=\frac{n-p}{n-2p+\frac{1}{\sum_{i=1}^k\frac{1}{\tr((X_i^\top X_i)^{-1}X^\top X)}}}.$$

For asymptotics in elliptical models, we find: 
\begin{theorem}[IE for elliptical and MP models]
\label{IE_mp_thm}
Under the conditions of Theorems \ref{gen_det} and \ref{re_ell}, the IE has the almost sure limit
\begin{align*}
IE(X_1,\ldots,X_k)
\to_{a.s.}
\frac{1-\gamma}{1-2\gamma+\frac{1}{\sum_{i=1}^k\psi(\gamma_i, G_i)}}
,
\end{align*}
where $\psi$ is the following functional of the distributions $G_i$ and $G$, depending on the inverse of the $\eta$-transform $f$ defined in equation \eqref{inv_eta}:
$$
\psi(\gamma_i, G_i)=\frac{1}{\gamma+(\E_GT-\frac{\gamma}{\gamma_i}\E_{G_i}T)f(\gamma_i,G_i)}.
$$
Under the conditions of Corollary \ref{mp_det}, the IE has the almost sure limit
$$
IE(X_1,\ldots,X_k)
\to_{a.s.}
\frac{1-\gamma}{1-2\gamma+\frac{\gamma(1-\gamma)}{1-k\gamma}} 
= \frac{1}{1+\frac{(k-1)\gamma^2}{(1-k\gamma)(1-\gamma)}}.
$$
\end{theorem}
See Section \ref{pf:IE_mp_thm} for the proof. This efficiency does not depend on a simple linear way on $k$, but rather via a ratio of two linear functions of $k$. However, it can be checked that many of the properties (e.g., monotonicity) for ARE still hold here.

\subsection{Out-of-sample prediction (Test error)}

In out-of-sample prediction, we consider a test datapoint $(x_t,y_t)$, generated from the same model $y_t=x_t^\top \beta+\ep_t$, where $x_t,\ep_t$ are independent of $X,\ep$, and only $x_t$ is observable. We want to use $x_t^\top \hbeta$ to predict $y_t$. We compare the prediction error of two estimators: 

$$OE(x_t;X_1,\ldots,X_k) 
:= \frac{\EE{(y_t-x_t^\top \hbeta)^2}}{\EE{(y_t-x_t^\top \hbeta_{dist})^2}}.
$$

In our general framework, we saw that this corresponds to predicting the linear functional $x_t^\top\beta+\ep_t$. Based on equation \eqref{opt_eff}, 
the optimal out-of-sample prediction efficiency is
\begin{align*}
OE(x_t; X_1,\ldots,X_k)
=\frac{1+ x_t^\top (X^\top X)^{-1}x_t}
{1 + \frac{1}{\sum_{i=1}^k \frac{1}{x_t^\top (X_i^\top X_i)^{-1} x_t}}}
.
\end{align*}

For asymptotics in elliptical models, we find the following result. Since the samples have the form $x_i = g_i^{1/2} \Sigma^{1/2} z_i$, the test sample depends on a scale parameter $g_t$. 

\begin{theorem}[OE for elliptical and MP models]
\label{OE_mp_thm}
Under the conditions of Theorems \ref{gen_det} and \ref{re_ell}, the OE has the almost sure limit, conditional on $g_t$
\begin{align*}
OE(x_t; X_1,\ldots,X_k)
\to_{a.s.}
\frac{1+ g_t\cdot  f(\gamma, G)}
{1 + \frac{g_t}{\sum_{i=1}^k \frac{1}{f(\gamma_i, G_i)}}}
.
\end{align*}
For Marchenko-Pastur models under the conditions of Corollary \ref{mp_det}, the OE has the almost sure limit
$$
\frac{\frac{1}{1-\gamma}}
{1 + \frac{\gamma}{1- k\gamma}}
= 
\frac{1}
{1 + \frac{(k-1)\gamma^2}{1- k\gamma}}.
$$
\end{theorem}
See Section \ref{pf:OE_mp_thm} for the proof.  If the scale parameter $g_t$ is random, then the OE typically does not have an almost sure limit, and converges in distribution to a random variable instead. We mention that Theorem \ref{OE_mp_thm} holds under even weaker conditions, if we are only given the $4+c$-th moment of $z_1$ instead of $8+c$-th one. The argument is slightly different, and is presented in the location referenced above.

One can check that that $OE \ge RE$. Thus, out-of-sample prediction incurs a smaller efficiency loss than estimation. The intuition is that the out-of-sample prediction always involves a fixed loss due to the \emph{irreducible noise} in the test sample, which "amortizes" the problem. Moreover, 
$$OE\geq IE\geq RE.$$ 
The intuition here is that IE incurs a smaller fixed loss than OE, because the noise in the training set is effectively reduced, as it is already partly fit by our estimation process. So the graph of IE will be in between the other two  criteria. See Figure \ref{aeoeie}. We also see that the IE is typically very close to OE.

In addition, the increase of the \emph{reducible} part of the error is the same as for estimation error. The prediction error has two components: the irreducible noise, and the reducible error. The reducible error has the same behavior as for estimation, and thus on figure \ref{aeoeie} it would have the same plot as the curve for estimation.

\begin{figure}
\begin{subfigure}{.45\textwidth}
  \centering
\includegraphics[scale=0.48]{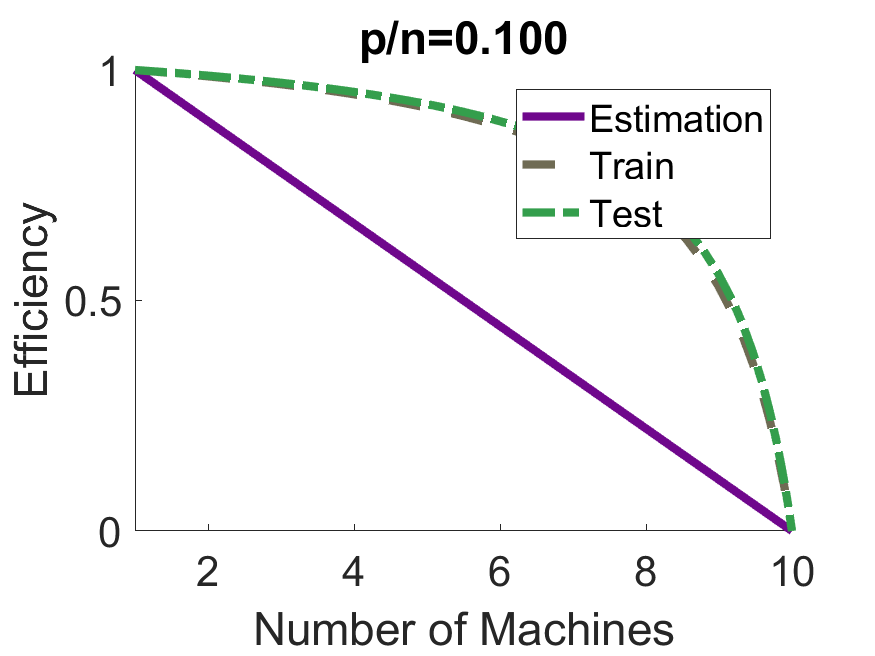}
\end{subfigure}
\begin{subfigure}{.45\textwidth}
  \centering
\includegraphics[scale=0.48]{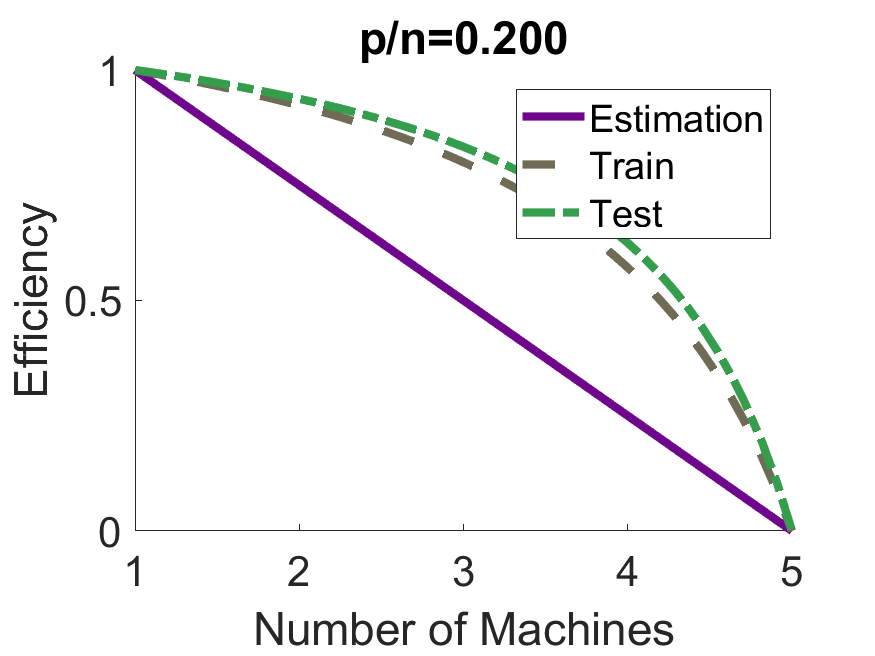}
\end{subfigure}
\caption{Relative efficiency for Marchenko-Pastur model.}
\label{aeoeie}
\end{figure}

\subsection{Confidence intervals}

To form confidence intervals, we consider the normal model $Y \sim \N(X\beta, \sigma^2 I_n)$. Recall that in this model the OLS estimator has distribution $\hbeta \sim \N(\beta, \sigma^2 (X^\top X)^{-1}).$ Assuming $\sigma^2$ is known, an exact level $1-\alpha$ confidence interval for a given coordinate $\beta_j$ can be formed as 
$$\hbeta_j \pm \sigma z_{\alpha/2} V_j^{1/2},$$
where $z_\alpha = \Phi^{-1}(\alpha)$ is the inverse normal CDF, and $V_j$ is the $j$-th diagonal entry of $(X^\top X)^{-1}$. We follow the same program as before, comparing the length of the confidence intervals formed based on our two estimators. However, for technical reasons it is more convenient to work with squared length. 

Thus we consider the criterion
$$CE(j;X_1,\ldots,X_k) 
:= \frac{V_j}{V_{j,dist}}.
$$
Here $V_{j,dist}$ is the variance of the $j$-th entry of an optimally weighted distributed estimator. As we saw in our framework, this is equivalent to estimating the $j$-th coordinate of $\beta$. Hence the optimal confidence interval efficiency is
\begin{align}
CE(j;X_1,\ldots,X_k) 
=[(X^\top X)^{-1}]_{jj}
\cdot 
\sum_{i=1}^k \frac{1}{[(X_i^\top X_i)^{-1}]_{jj}}
.
\label{ce_fs}
\end{align}

For asymptotics, we find: 
\begin{theorem}[CE for elliptical and MP models]
\label{ce_ell}
Under the conditions of Theorems \ref{gen_det} and \ref{re_ell}, the CE has the same limit as the ARE from Theorem \ref{re_ell}. Therefore, for Marchenko-Pastur models, the CE also has the form before, $CE(j)= (1/\gamma-k)/(1/\gamma-1).$
\end{theorem}

See Section \ref{pf:ce_ell} for the proof.

\subsection{Understanding and comparing the efficiencies}
\label{underst_com}
We give two perspectives for understanding and comparing the efficiencies. The key qualitative insight is that estimation and CIs are much more affected than prediction. %We make this precise below.

{\bf Criticality of $k$.} We ask: What is the largest number of machines such that the asymptotic efficiency is at least 1/2? Let us call this the \emph{critical} number of machines. It is easy to check that for estimation and CIs, $k_R=(\gamma+1)/(2\gamma)$. For training error, $k_{Tr} = (\gamma^2-\gamma+1)/\gamma$, while for test error, $k_{Te} = (\gamma^2+1)/(\gamma^2+\gamma)$.

We also have the following asymptotics as $\gamma\to 0$: 
$$k_R\asymp 1/(2\gamma),$$ 
while 
$$k_{Tr} \asymp k_{Te} \asymp 1/\gamma.$$

So the number of machines that can be used is nearly maximal (i.e., $n/p$) for training and test error, while it is about \emph{half that} for estimation error and CIs. This shows quantitatively that estimation and CIs are much more affected  by distributed averaging than prediction.

{\bf Edge efficiency.} The maximum number of machines that we can use is approximately $k^* = 1/\gamma-1$, for small $\gamma$. Let us define the \emph{edge efficiency} $e^*$ as the relative efficiency achieved at this edge case. For estimation and CIs, we have $e^*_R=\gamma/(1-\gamma)$. For training error, $e^*_{Tr} = (1-\gamma)/(2-3\gamma)$, and for test error, $e^*_{Te} = 1/[2(1-\gamma)]$.

We also have the following asymptotic values as $\gamma\to 0$: 
$$e^*_R\asymp \gamma,$$ 
while 
$$e^*_{Tr} \asymp \frac{1}{2}+\frac{\gamma}{4},\textnormal{ and }e^*_{Te} \asymp \frac{1}{2}+\frac{\gamma}{2}.$$
This shows that for $n\gg p$ the edge efficiency is vanishing for estimation and CIs, while it is approximately 1/2 for training and test error. Thus, even for the maximal number of machines, prediction error is not greatly increased.

\section{Insights for Parameter Estimation}
\label{estim}
There are additional insights for the special case of parameter estimation. First, it is of interest to understand the performance of one-step weighted averaging with suboptimal weights $w_i$. How much do we lose compared to the optimal performance if we do not use the right weights? In practice, it may seem reasonable to take a simple average of all estimators. We have performed that analysis in the supplement (Section \ref{subopt}), and we found that the loss can be viewed in terms of an inequality between the arithmetic and harmonic means.

There are several more remarkable properties. We have studied the monotonicity properties and interpretation of the relative efficiency, see the supplement (Section \ref{prop-int}). 
We have also given a multi-response regression characterization that heuristically gives an upper bound on the "\emph{degrees of freedom}" for distributed regression (Section \ref{dof}). 

For elliptical data, the graph of ARE is a curve below the straight line from before. The interpretation is that for elliptical distributions, there is a larger efficiency loss in one-step weighted averaging. Intuitively, the problem becomes "more non-orthogonal" due to the additional variability from sample to sample. 

It is natural to ask which elliptical distributions are difficult for distributed estimation. For what scale distributions $G$ does the distributed setting have a strong effect on the learning accuracy? Intuitively, if some of the scales are much larger than others, then they "dominate" the problem, and may effectively reduce the sample size. We show that this intuition is correct, and we find a sequence of scale distributions $G_\tau$ such that distributed estimation is "arbitrarily bad", so that the ARE decreases very rapidly, and approaches zero even for two machines (see Section \ref{worst_case} in the supplement). 

\section{Multi-shot methods} 
\label{multi}

While our focus has been on methods with one round of communication, in practice it is more common to use iterative methods with several rounds of communication. These usually improve statistical accuracy. A great deal of research has been done on multi-shot distributed algorithms. Due to limited space, here we will only list and analyze some of them.  Our least squares objective can be written as a sum of least squares objectives for each machine as
$$
f(\beta)=\frac{1}{k}\sum_{i=1}^kf_i(\beta)=\frac{1}{k}\sum_{i=1}^k\|X_i\beta-Y_i\|_2^2.
$$
Here each machine has access only to local data $(X_i, Y_i)$. %For simplicity, in this section we assume the samples are uniformly distributed across machines, i.e. $n_1=n_2=\cdots=n_k=n/k$. 
With this formulation, there are a large number of standard optimization methods to minimize this objective: distributed gradient descent, alternating directions method of multipliers, and several others we discuss below. We will focus on parameter server architectures, where each machine communicates independently with a central server.%, though other architectures are also of interest.

{\bf Distributed Gradient Descent.} A simple multi-round approach to distributed learning is synchronous distributed gradient descent (DGD), as discussed e.g., in \cite{chu2007map}. This maintains iterates $\hbeta^{t}$, started with some standard value, such as $\hbeta^{0}=0$. At each iteration $t$ each local machine calculates the gradient $\nabla f_i(\hbeta^{t})$ at the current iterate $\hbeta^{t}$, and then sends the local gradient to the server to obtain the overall gradient 
$$\nabla f(\hbeta^{t}) = \frac{1}{k}\sum_{i=1}^k\nabla f_i(\hbeta^{t}).$$ 
Then the center server sends the updated parameter $\hbeta^{t+1}=\hbeta^{t}-\alpha\nabla f(\hbeta^{t})$ back to the local machines, where $\alpha$ is the learning rate (LR). This synchronous implementation is \emph{identical} to centralized gradient descent. Thus for smooth and strongly convex objectives and suitably small $\alpha$, $\mathcal{O}(L/\lambda\log(1/\ep))$  communication rounds are sufficient to attain an $\ep$-suboptimal solution in terms of objective value, where $L,\lambda$ are the smoothness and strong convexity parameters \citep[e.g.,][]{boyd2004convex}. 

\begin{enumerate}
\item Many works study the optimization properties of GD/synchronous DGD, in terms of convergence rate to the optimal objective or parameter value. From a statistical point of view, the GD iterates start with large bias and small variance, and gradually reduce bias, while slightly increasing the variance. This has motivated work on the risk properties of GD, emphasizing early stopping \citep[][e.g.,]{yao2007early,alnur}. Recently, \cite{alnur} gave a more refined analysis of the estimation risk of GD for OLS, showing that its risk at an optimal stopping time is at most 1.22 times the risk of optimally tuned ridge regression. 

\item Compared to GD, one-shot weighted averaging has several advantages: it is simpler to implement, as it requires no iterations. It requires fewer tuning parameters, and those can be set optimally in an easy way, unlike the LR $\alpha$. The weights are proportional to $1/\tr[(X_i^TX_i)^{-1}]$, which can be computed locally. We point out that GD is sensitive to the learning rate: this has to be bounded (by $2/\lambda_{\max}(X^\top X)$ for OLS) to converge, and the convergence can be faster for large LR, hence in practice sophisticated LR schedules are used. This can make DGD complicated to use. In addition, in practice DGD is susceptible to stragglers, i.e., machines that take too long to compute their answers. To mitigate this problems, asynchronous DGD algorithms \cite[e.g.,][]{tsitsiklis1986distributed,nedic2009distributed}, and other sophisticated coding ideas \citep{tandon2017gradient} have been proposed. However those lead to additional complexity and hyperparameters to tune (e.g., for async algorithms: how much to wait, how to aggregate non-straggler gradients). 

\item One may also use other gradient based methods, such as accelerated or quasi-Newton methods, e.g., L-BFGS \citep{agarwal2014reliable}.

\end{enumerate}

{\bf Alternating Direction Method of Multipliers (ADMM).} Another approach is the alternating direction method of multipliers \cite[see][for an exposition]{boyd2011distributed} and its variants. In ADMM, we alternate between solving local problems, global averaging, and computing local dual variables. For us, at time step $t$ of ADMM, each local machine calculates a local estimator 
$$
\hbeta_i^{t+1}=(X_i^\top X_i+\rho I)^{-1}[X_i^\top Y_i+\rho(\hbeta^{t}-u_i^{t})],
$$
(where $\rho$ is a hyperparameter) and sends it to the parameter server to get an average
$$
\hbeta^{t+1}=\frac{1}{k}\sum_{i=1}^k\hbeta_i^{t+1}.
$$
Finally, the server sends $\hbeta^{t+1}$ back to the local machines to update the dual variables
$$
u_i^{t+1}=u_i^{t}+\hbeta_i^{t+1}-\hbeta^{t+1}.
$$
These three steps can be written as a linear recursion $z^{t+1}=Az^{t}+b$ for a state variable $z^{t}$ including $\hbeta^{t}, \hbeta_i^{t}$ and $u_i^{t}$. If all singular values of $A$ are less than one, then the iteration converges to a fixed point solving $z=Az+b$, so that $z=(I-A)^{-1}b$. However, it seems hard to prove convergence in our asymptotic setting.

{\bf Distributed Approximate Newton-type Method (DANE).} \cite{DANE14} proposed an approximate Newton-like method (DANE), which uses that the sub-problems are similar. For our problem, DANE aggregates the local gradients on the parameter server at each step $t$, and sends this quantity, i.e., $X^\top (X\hbeta^t-Y)/(2k)$ to all machines. Then each machine computes a local estimator by a gradient step in the direction of a regularized local Hessian $X_i^\top X_i+\rho I$,
$$
\hbeta_i^{t+1}=\hbeta^{t}+\frac{\eta}{k}\cdot(X_i^\top X_i+\rho I)^{-1} X^\top(Y-X\hbeta^{t}),
$$
where $\rho$ is the regularizer and $\eta$ is the learning rate. The machines send it to the server to get the aggregated estimator
$$
\hbeta^{t+1}=\frac{1}{k}\sum_{i=1}^k\hbeta_i^{t+1}.
$$
For a noiseless model where $Y = X\beta$, we can summarize the update rule as 
$$\hbeta^{t+1}-\beta=\left(I-\frac{\eta}{k^2}\cdot\sum_{i=1}^k\left(X_i^\top X_i+\rho I\right)^{-1}X^\top X\right)\left(\hbeta^t-\beta\right),$$
so we have the error bound
$$\|\hbeta^t-\beta\|_2\leq\left\|I-\frac{\eta}{k^2}\cdot\sum_{i=1}^k\left(X_i^\top X_i+\rho I\right)^{-1}X^\top X\right\|_2^t\cdot\|\hbeta^0-\beta\|_2.$$
In \cite{DANE14}, the authors showed that given a suitable learning rate $\eta$ and regularizer $\rho$, when $X_i^\top X_i$ is close to $X^\top X/k$, $\hbeta^t\to\beta$ as $t\to\infty$. 

For a noisy linear model $Y=X\beta+\ep$, the limit of $\hbeta^t$ is exactly the OLS estimator of the whole data set, and we have the following recursion formula
$$\hbeta^{t+1}-(X^\top X)^{-1}X^\top Y=\left(I-\frac{\eta}{k^2}\cdot\sum_{i=1}^k\left(X_i^\top X_i+\rho I\right)^{-1}X^\top X\right)\left(\hbeta^t-(X^\top X)^{-1}X^\top Y\right),$$
and the convergence guarantee is the same as for the noiseless case.

{\bf Iterative Averaging Method.} Here we describe an iterative averaging method for distributed linear regression. This method turns out to be connected to DANE, and it has the advantage that it can be analyzed more conveniently. We define a sequence of \emph{local estimates} $\hbeta^{t}_i$ and \emph{global estimates} $ \hbeta^{t}$ with initialization $\hbeta^0 = 0$. At the $t$-th step we update the local estimate by the following weighted average of the local ridge regression estimator and the current global estimate $\hbeta^t$
$$\hbeta^{t+1}_i = 
\left (X_i^\top X_i+n_i\rho_i I\right)^{-1}
 \left (X_i^\top Y_i+n_i\rho_i\hbeta^t\right).$$
%Let us first consider a noiseless model $Y_i = X_i \beta$, we want to take a weighted average of $\beta$ and $\hbeta^t$. A natural way to do this is to use a regularized estimate of $\beta$, amounting to a weighting in each spatial direction with a weight less than unity, and put the complement of the shrinkage weight on $\hbeta^t$. 
Then we average the local estimates
$$ \hbeta^{t+1} = \frac{1}{k} \sum_{i} \hbeta^{t+1}_i.$$
To understand this, let us first consider a noiseless model where $Y_i = X_i \beta$. In that case, we can also write this update as a weighted average,
$$\hbeta^{t+1}_i= (I-W_{i}) \beta+ W_i  \hbeta^t,$$
where 
$$W_i = n_i\rho_i \cdot
\left (X_i^\top X_i+n_i\rho_i I\right)^{-1}$$
is the weight matrix of the global estimate. Propagating the iterative update to the global machine, we find a linear update rule: 
$$ \hbeta^{t+1}
= \frac{1}{k}\sum_i W_{i}   \hbeta^t+ \left (I- \frac{1}{k}\sum_i W_i \right )\beta
=  W  \hbeta^t + (I -  W) \beta, 
$$
where $ W = \frac{1}{k}\sum_i W_{i}$. Hence, the error is updated as
$$\hbeta^{t+1}-\beta=W\cdot [\hbeta^t-\beta]=\left(I-\frac{1}{k}\sum_{i=1}^k\left(X_i^\top X_i+n_i\rho_iI\right)^{-1}X_i^\top X_i\right)\left(\hbeta^t-\beta\right).
$$
This recursion relation is very similar to the one for DANE; we just need to replace $X^\top X/k$ by $X_i^\top X_i$ (and in practice usually $\eta=1$ is used). The only difference is that DANE has a step where we need to collect the local gradients to get the global gradient, and then broadcast it to all local machines. Our iterative averaging method has lower communication cost. 

In terms of convergence, $\hbeta^{t+1}$ will converge geometrically to $\beta$ for all $\beta$, if and only if the largest eigenvalue of $W$ is strictly less than $1$. It is not hard to see that this holds if at least one $X_i^\top X_i$ has positive eigenvalues by using the fact $\lambda_{\max}(A+B)\leq\lambda_{\max}(A)+\lambda_{\max}(B)$. When the samples are uniformly distributed, we should have $X^\top X/k\approx X_i^\top X_i$, which means the convergence rates of DANE and iterative averaging should be very close. Hence, in terms of the total cost (communication and computation), our iterative averaging should compare favorably to DANE.

To summarize the noiseless case, we can formulate the following result:

\begin{theorem}[Convergence of iterative averaging, noiseless case]
Consider the iterative averaging method described above. In the noiseless case when $Y_i = X_i\beta$, we have the following:
%\begin{enumerate}
If at least one $X_i^\top X_i$ has positive eigenvalues, then the iterates converge to the true coefficients geometrically, $\hbeta^t\to\beta$, and:
$$\|\hbeta^t-\beta\|_2
\le 
\lambda_{\max}\left(\frac{1}{k}\sum_{i=1}^k n_i\rho_i\cdot
\left (X_i^\top X_i+n_i\rho_i I\right)^{-1}\right)^t
\cdot
\|\beta\|_2
.$$
%\item Suppose $\beta \in \R^p$ is a random vector with i.i.d. entries of zero mean, $1/p$ variance, and $\sup_i \E (\sqrt{p} \beta_i)^{4+\eta}\le C$ for some $\eta > 0$ and $C< \infty$. Suppose moreover that the data matrices $X_i$ are random as in Theorem \ref{are_mp_thm} with $\Sigma=I$. Let $M_t$ be the MSE of the algorithm at iteration $t$. Then we have asymptotically
%$$M_t -p^{-1}\tr[W^{2t}]\to_{a.s.}0,$$
%and the matrix $k\cdot W$ has a limiting spectral distribution that is the free additive convolution of the $k$ distributions with R-transform $R_i$, $i=1,\ldots,k$ solving
%$$\frac{1}{R_{i}(z)+1/z}
%+\frac{\rho_i}{(R_{i}(z)+1/z)^2}m_{MP}\left(\rho_i\left(\frac{1}{R_{i}(z)+1/z}-1\right);\gamma_i\right)
%=z.$$
%\end{enumerate}
\label{noiseless}
\end{theorem}

Consider now the noisy case when $Y_i = X_i\beta+\ep_i$ with the same assumptions as in the rest of the paper. We have
\begin{align*}
\hbeta^{t+1}_i 
&= W_i  \hbeta^t + (I-W_{i})\beta 
+ \left (X_i^\top X_i+n_i\rho_i I\right)^{-1}X_i^\top \ep_i\\
&=_d W_i  \hbeta^t + (I-W_{i})\beta 
+ \left (X_i^\top X_i+n_i\rho_i I\right)^{-1}X_i^\top X_i \cdot Z_i\\
&= W_i  \hbeta^t + (I-W_{i})(\beta + Z_i),
\end{align*}
where $Z_i \sim \N(0,\sigma^2 [X_i^\top X_i]^{-1})$. As before, defining $Z$ appropriately
\begin{align*}
\hbeta^{t+1}
 & =  W  \hbeta^t + (I -  W) \beta
+ \frac{1}{k}\sum_{i=1}^k(I-W_{i})Z_i\\ 
 & =  W  \hbeta^t + (I -  W) \beta
+ Z,
\end{align*}
so $ \hbeta^{t+1} - \beta =  W\cdot [  \hbeta^t -  \beta] +Z$.

With noise, $\hbeta^t$ does not converge to OLS, but to the following quantity:
$$\hbeta_{*}=\left(\sum_{i=1}^k\left(X_i^\top X_i+n_i\rho_iI\right)^{-1}X_i^\top X_i\right)^{-1}\cdot\sum_{i=1}^k\left(X_i^\top X_i+n_i\rho_iI\right)^{-1}X_i^\top Y_i.$$
We can check that $\hbeta_{*}$ is an unbiased estimator for $\beta$ and
$\hbeta^{t+1}-\hbeta_{*}=W\cdot [\hbeta^t-\hbeta_{*}].
$

Under the conditions of Theorem \ref{noiseless}, we have $\hbeta^t\to\hbeta_{*}$, and the MSE for $\hbeta_{*}$ is
\begin{align*}
\E\|\hbeta_{*}-\beta\|^2
&= \E\|(I-W)^{-1}Z\|^2 = \E\|\frac{1}{k}\sum_{i=1}^k(I-W)^{-1}(I-W_{i})Z_i\|^2\\
&=\frac{\sigma^2}{k^2}\sum_{i=1}^k
\tr[(I-W)^{-2}(X_i^\top X_i+n_i\rho_i I)^{-2}X_i^\top X_i]\\
&=\sigma^2\sum_{i=1}^k
\tr\left[\left(\sum_{i=1}^k\left(X_i^\top X_i+n_i\rho_iI\right)^{-1}X_i^\top X_i\right)^{-2}\left(X_i^\top X_i+n_i\rho_i I\right)^{-2}X_i^\top X_i\right].
\end{align*}

How large is this MSE, and how does it depend on $\rho_i$?% For simplicity, we assume all sample sizes $n_i$ are equal to $n/k$, and let all $\rho_i$ equal to a common value $\rho$. 
We have the following results.

\begin{theorem}[Properties of Iterative averaging, noisy case]
Consider the iterative averaging method described above. In the noisy case when $Y_i = X_i\beta+\ep$, we have the following:
\begin{enumerate}
\item If at least one $X_i^\top X_i$ has strictly positive eigenvalues, then the iterates converge to the following limiting unbiased estimator
$$\hbeta_{*}=\left(\sum_{i=1}^k\left(X_i^\top X_i+n_i\rho_iI\right)^{-1}X_i^\top X_i\right)^{-1}\cdot\sum_{i=1}^k\left(X_i^\top X_i+n_i\rho_iI\right)^{-1}X_i^\top Y_i,$$
and the convergence is geometric
$$\|\hbeta^t-\hbeta_{*}\|_2
\le 
\lambda_{\max}\left(\frac{1}{k}\sum_{i=1}^k n_i\rho_i\cdot
\left (X_i^\top X_i+n_i\rho_i I\right)^{-1}\right)^t
\cdot
\|\hbeta_{*}\|_2
.$$
\item The mean squared error of $\hbeta_{*}$ has the following form
$$\E\|\hbeta_{*}-\beta\|^2=\sigma^2\sum_{i=1}^k
\tr\left[\left(\sum_{i=1}^k\left(X_i^\top X_i+n_i\rho_iI\right)^{-1}X_i^\top X_i\right)^{-2}\left(X_i^\top X_i+n_i\rho_i I\right)^{-2}X_i^\top X_i\right].$$
\item Suppose the samples are evenly distributed, i.e., $n_1=n_2=\cdots=n_k=n/k$ and the regularizers are all the same $\rho_1=\rho_2=\cdots=\rho_k=\rho$. The MSE is a differentiable function $\psi(\rho)$ of the regularizer $\rho\in[0, +\infty)$, with derivative
\begin{align*}
\psi'(\rho)=\frac{2k}{n}\tr[&\Delta^{-1}\sum_{i=1}^k\left(\hSigma_i+\rho I\right)^{-2}\hSigma_i\cdot\Delta^{-2}\sum_{i=1}^k\left(\hSigma_i+\rho I\right)^{-2}\hSigma_i\\
&-\Delta^{-2}\sum_{i=1}^k\left(\hSigma_i+\rho I\right)^{-3}\hSigma_i],
\end{align*}
where $\hSigma_i=X_i^\top X_i/n_i$ and
$\Delta:=\sum_{i=1}^k\left(\hSigma_i+\rho I\right)^{-1}\hSigma_i.$
\item $\psi(\rho)$ is a non-increasing function on $[0, +\infty)$ and $\psi'(0)<0$. So for any $\rho>0, \psi(\rho)<\psi(0)$, i.e. the MSE of the iterative averaging estimator with positive regularizer is smaller than the MSE of the one-step averaging estimator.
\item When $\rho=0$, $\hbeta_{*}$ reduces to the one-step averaging estimator $1/k\cdot\sum_{i=1}^k(X_i^\top X_i)^{-1}X_i^\top Y_i$ with MSE
$$\psi(0)=\sigma^2/k^2\cdot\sum_{i=1}^k\tr(X_i^\top X_i)^{-1}.$$
When $\rho\to+\infty$, $\hbeta_{*}$ converges to the OLS estimator $(X^\top X)^{-1}X^\top Y$ with MSE
$$\lim_{\rho\to+\infty}\psi(\rho)=\sigma^2\tr(X^\top X)^{-1}.$$
\end{enumerate}
\label{noisy}
\end{theorem}

See Section \ref{pf:noisy} of the supplement for the proof. The argument for monotonicity relies on Schur complements, and is quite nontrivial. From Theorem \ref{noisy}, it appears we should choose the regularizer $\rho$ as large as possible, since the limiting estimator $\hbeta_{*}$ will converge to the OLS estimator as $\rho\to\infty$. This is true for statistical accuracy. However, there is a computational tradeoff, since the convergence rate of $\hbeta^t$ to $\hbeta_{*}$ is slower for large $\rho$.

Moreover, one may argue that $\hbeta_{*}$ reduces to the naive averaging estimator but not the optimally weighted averaging estimator when $\rho=0$. However, we have shown in the supplement (Section \ref{subopt}) that for evenly distributed samples, the MSE of the naive averaging estimator and the optimally weighted averaging estimator is asymptotically the same. Thus, there exists a regularizer such that the iterative averaging estimator has smaller MSE than the one-step weighted averaging estimator.

{\bf Other approaches.} There are many other approaches to distributed learning.  \emph{Dual averaging for decentralized optimization} over a network \citep{duchi2011dual}  builds on Nesterov's dual averaging method \citep{nesterov2009primal}. It chooses the iterates to minimize an averaged first-order approximation to the function, regularized with a proximal function. The \emph{communication-efficient surrogate likelihood} approximates the objective by an expression of the form $\tilde f(\beta) = f_1(\beta) - \beta^\top(\nabla f_1(\bar\beta)-\nabla f(\bar\beta))$, where $\bar \beta$ is a preliminary estimator \citep{jordan2016communication,wang2017efficient}.  \cite{chen2018quantile} propose a related method for quantile regression. Both are related to DANE \citep{DANE14}. 

\cite{chen2018first} study divide and conquer SGD (DC-SGD), running SGD on each machine and averaging the results. They also propose a distributed first-order Newton-type estimator starting with a preliminary estimator $\bar\beta$, of the form $\bar\beta-\Sigma^{-1}(k^{-1}\sum_i \nabla f_i(\bar\beta))$, where $\Sigma$ is the population Hessian. They show how to numerically estimate this efficiently, and also develop a more accurate multi-round version.

\subsection{Numerical comparisons} We report simulations to compare the convergence rate and statistical accuracy of the one-shot weighted method with some popular multi-shot methods described above (Figure \ref{comparison}). Here we work with a linear model $Y=X\beta+\ep$, where $X, \beta$ and $\ep$ all follow standard normal distributions. We take $n = 10000, p = 100$, and $k = 20$. We plot the relative efficiencies of different methods against the number of iterations. 

We can see that the one-shot weighted method is good in some cases. The multi-shot methods usually need several iterations to achieve better statistical accuracy. When the communication cost is large, one-shot methods are attractive. Also, we can clearly see the computation vs accuracy tradeoff for the iterative averaging method from the plots. When the regularizer is small, the convergence is fast, but in the end the accuracy is not as good as the other multi-shot methods. On the other hand, if the regularizer is large, we have a better accuracy with slower convergence. Moreover, the widely-used multi-shot methods can require a lot of work for parameter tuning, and sometimes it is very difficult to find the optimal  parameters. See Figure \ref{stepsize} for an example. In contrast, weighted averaging requires less tuning, making it a more attractive method.

\begin{figure}
\begin{subfigure}{.45\textwidth}
  \centering
\includegraphics[scale=0.18]
{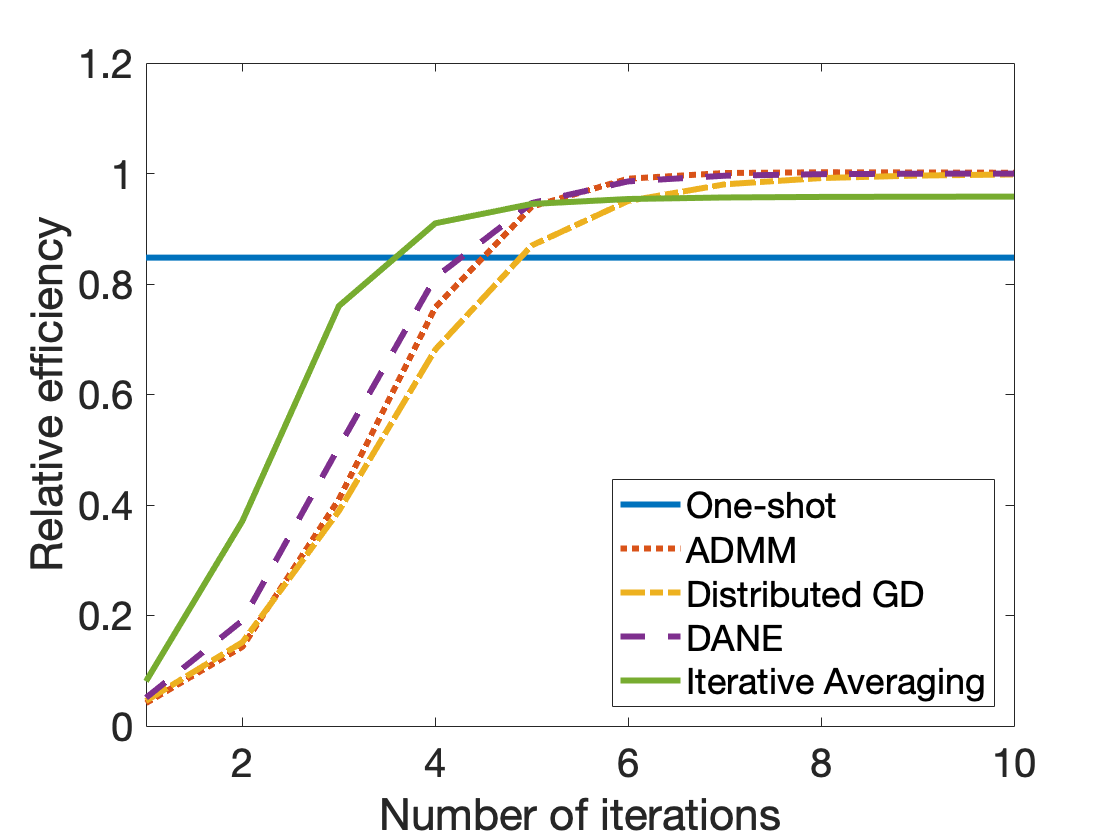}
\end{subfigure}
\begin{subfigure}{.45\textwidth}
  \centering
\includegraphics[scale=0.18]
{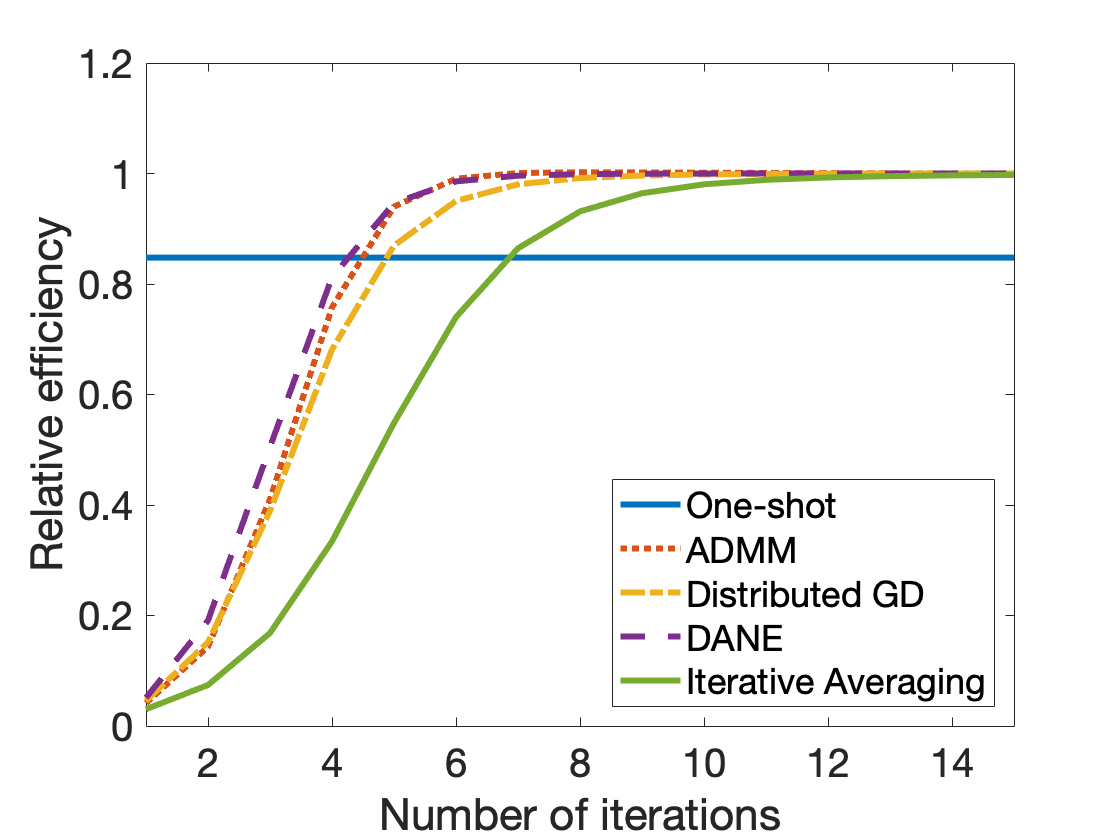}
\end{subfigure}
\caption{Comparison of the one-shot weighted method and several widely used multi-shot methods,}
\label{comparison}
\end{figure}

\begin{figure}
\begin{subfigure}{.45\textwidth}
  \centering
\includegraphics[scale=0.18]
{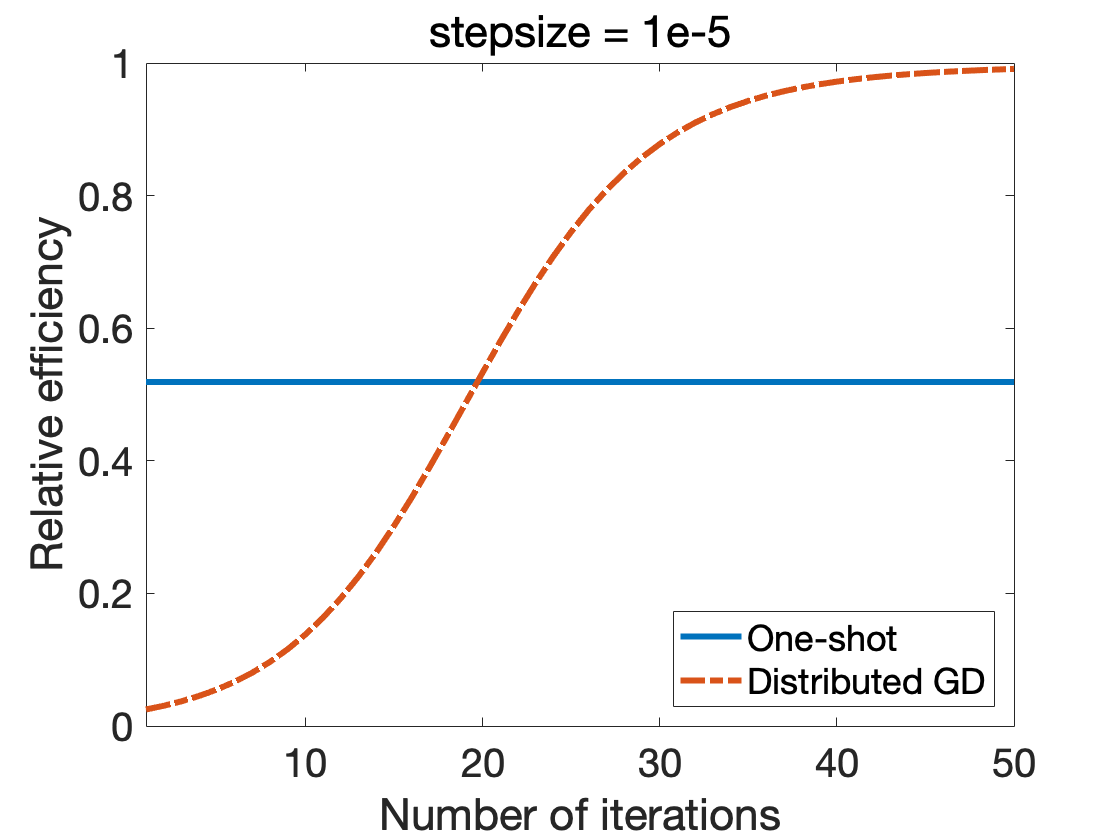}
\end{subfigure}
\begin{subfigure}{.45\textwidth}
  \centering
\includegraphics[scale=0.18]
{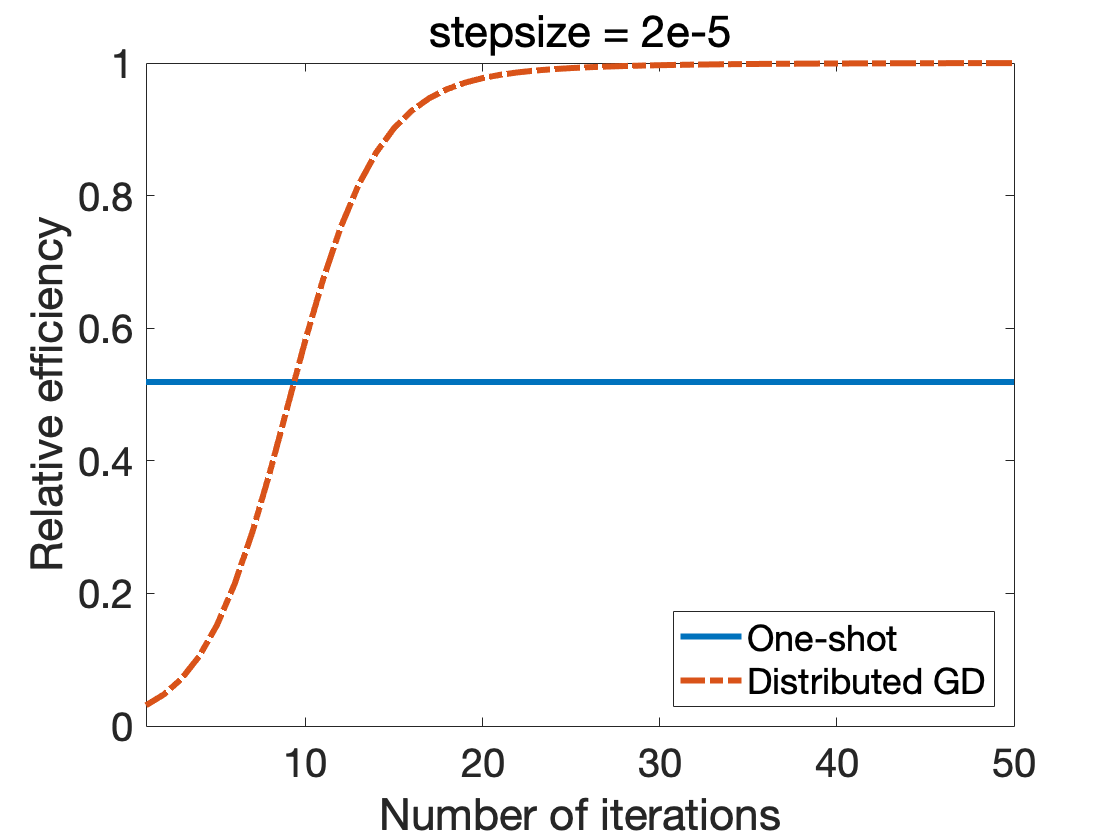}
\end{subfigure}
\begin{subfigure}{.45\textwidth}
  \centering
\includegraphics[scale=0.18]
{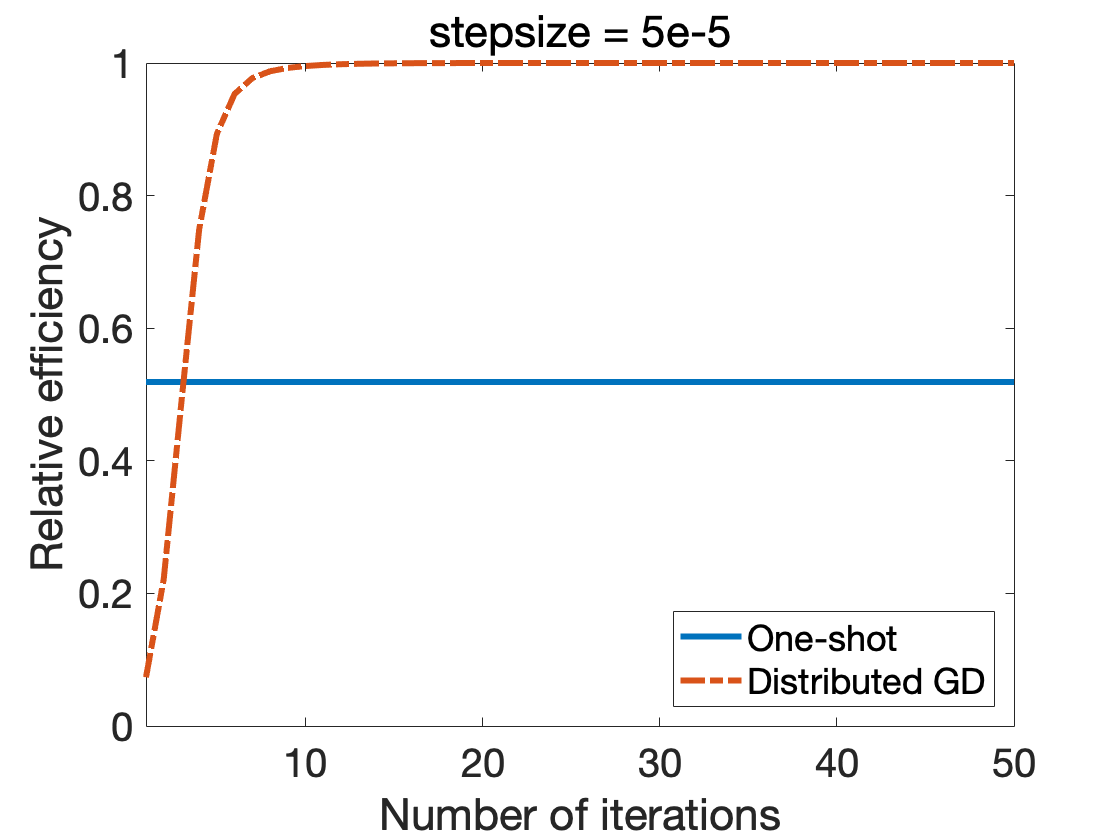}
\end{subfigure}
\begin{subfigure}{.45\textwidth}
  \centering
\includegraphics[scale=0.18]
{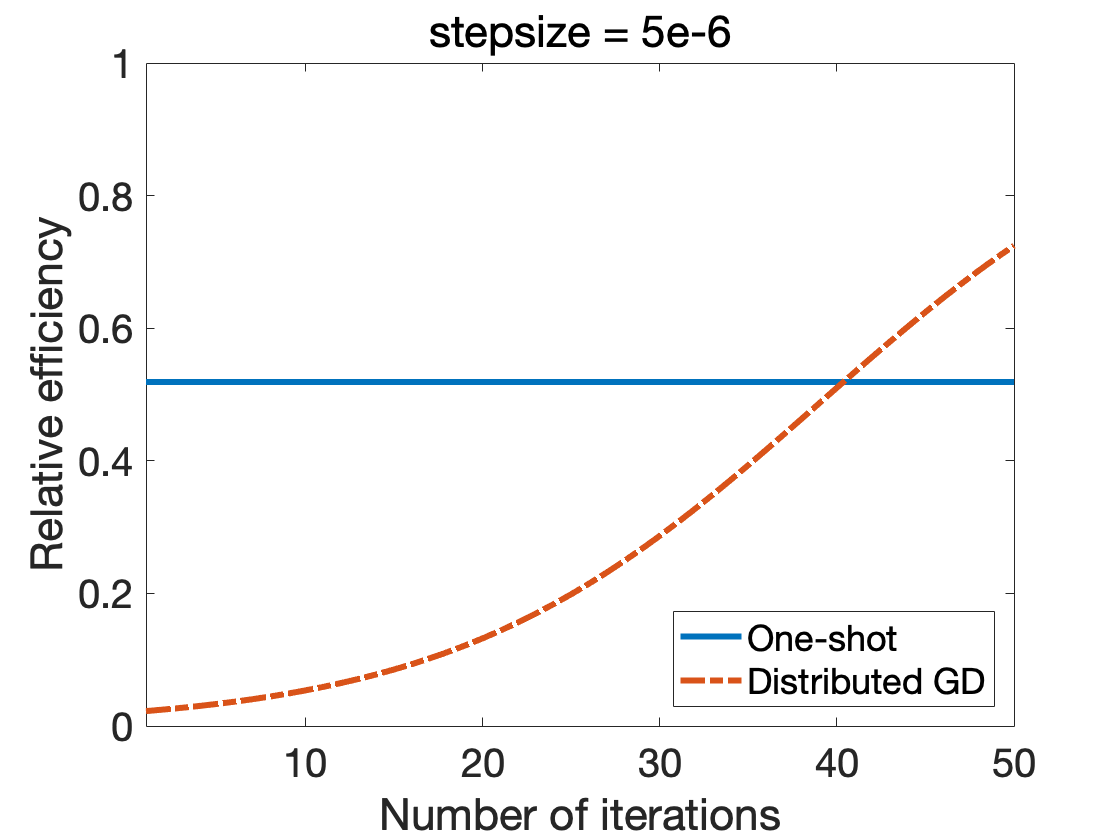}
\end{subfigure}
\caption{Comparison of the one-shot weighted method and distributed GD with different stepsizes. The point of this figure is that the behavior of GD depends strongly on the stepsize. In particular, the number of iterations needed to reach the performance of one-shot regression can vary a lot.}
\label{stepsize}
\end{figure}

We have performed several more numerical simulations to verify our theory, in addition to the results shown in the paper. Due to space limitations,  these are presented in the supplement. In Section \ref{emp}, we present an empirical data example to assess the accuracy of our theoretical results for one-shot averaging. We find that they can be quite accurate.

\section*{Acknowledgments}

We thank Jason D. Lee, Philip Gressman, Andreas Haeberleen, Boaz Nadler, Balasubramanian Narashiman and Ziwei Zhu for helpful discussions. We are grateful to Sifan Liu for providing an initial script for processing the empirical data. We thank John Duchi for pointing out references from convex optimization showing the concavity of the relative efficiencies (proposition \ref{ccv_re2}). We thank Du Bing for pointing out a typo in the definition of $N_i$ in a previous version of the manuscript.
This work was partially supported by NSF BIGDATA grant IIS 1837992. 

%\section{Proofs} The proofs are presented in the Appendix.

{\small
\setlength{\bibsep}{0.0pt plus 0.0ex}
\bibliographystyle{plainnat-abbrev}
\bibliography{references}
}

\appendix

\numberwithin{equation}{section}

%The supplement is organized as follows: Section \ref{sec:comput} describes the efficient computation of the risk formulas, Section \ref{proofs_ridge_supp} has the proofs for ridge regression, and Section \ref{proofs_rda_supp} has the proofs for regularized discriminant analysis. At several locations we refer to equation numbers from the main text.

\section{Form of optimal weights}
\label{pf:re_ols}

%\begin{proof}
Here we describe the proof of the form of optimal weights for the special case of parameter estimation.
Each local estimator is unbiased, and has MSE $M_i =  \sigma^2 \tr [(X_i^\top X_i)^{-1}]$. If we restrict to $\sum_{i=1}^k w_i=1$, then the weighted estimator is unbiased and its MSE equals
$$MSE(w) = \sum_i w_i^2 \cdot MSE(\hbeta_i) = \sum_i w_i^2 \cdot M_i.$$

Clearly, to minimize this subject to $\sum_i w_i =1$, by the Cauchy-Schwarz inequality we should take $w_i^*  = M_i^{-1}/(\sum_j M_j^{-1})$, and the minimum is $1/(\sum_j M_j^{-1})$. This finishes the proof. 
%\end{proof}

% \section{Proof of Proposition \ref{ccv_re}}
% \label{pf:ccv_re}

% %\begin{proof}

% %\end{proof}

\section{Proof of Proposition \ref{ccv_re2}}
\label{pf:ccv_re2}

 Notice that, it is  sufficient to show that, for any given $A$, the function $f(X)=1/\tr(X^{-1}A^\top A)$ is concave on positive definite matrices. Similar to the proof of proposition 2.2., we can define $g(t)=f(X+tV)$. The constraints on $X, V, X+tV$ are the same as before. Now, we have

\begin{align*}
g(t)&=\frac{1}{\tr((X+tV)^{-1}A^\top A)}=\frac{1}{\tr((I+tX^{-1/2}VX^{-1/2})^{-1}X^{-1/2}A^\top AX^{-1/2})}\\
&=\frac{1}{\tr((I+t\Lambda)^{-1}Q^\top X^{-1/2}A^\top AX^{-1/2}Q)}\\
&=\bigg(\sum_{i=1}^n\frac{(Q^\top X^{-1/2}A^\top AX^{-1/2}Q)_{ii}}{1+t\lambda_i}\bigg)^{-1}.
\end{align*}
Since $Q^\top X^{-1/2}A^\top AX^{-1/2}Q$ is always nonnegative definite, we can get the desired result by an explicit calculation. We show below the steps for $A=I$ for simplicity, but the same steps extend to general $A$. 

Let us define $g(t)=f(X+tV)$, where $X\succ 0$ is a positive definite matrix and $V$ is any symmetric matrix such that $X+tV\succ 0$ is still positive definite. Then $f(X)$ is concave iff $g(t)$ is concave on its domain for any $X$ and $V$.\\
Now we have
\begin{align*}
g(t)=\frac{1}{\tr[(X+tV)^{-1}]} &=\frac{1}{\tr[X^{-1}(I+tX^{-1/2}VX^{-1/2})^{-1}]}\\
& = \frac{1}{\tr[X^{-1}Q(I+t\Lambda)^{-1}Q^\top]}\\
& = \frac{1}{\tr[Q^\top X^{-1}Q(I+t\Lambda)^{-1}]}\\
& = \bigg(\sum_{i=1}^n\frac{(Q^\top X^{-1}Q)_{ii}}{1+t\lambda_i}\bigg)^{-1},
\end{align*}
where $\lambda_i$-s are eigenvalues of $X^{-1/2}VX^{-1/2}$. From the assumption, we always have $1+t\lambda_i>0$. Since $Q^\top X^{-1}Q$ is a positive definite matrix, its diagonal elements are all positive. We may use the notation $\alpha_i=(Q^\top X^{-1}Q)_{ii}$. Then, let us compute $g'(t)$ and $g''(t)$. First we have
$$
g'(t)=\bigg(\sum_{i=1}^n\frac{\alpha_i}{1+\lambda_it}\bigg)^{-2}\cdot\bigg(\sum_{i=1}^n\frac{\alpha_i\lambda_i}{(1+\lambda_it)^2}\bigg),$$
Next, we find
\begin{align*}
g''(t)&=2\bigg(\sum_{i=1}^n\frac{\alpha_i}{1+\lambda_it}\bigg)^{-3}\cdot\bigg[\bigg(\sum_{i=1}^n\frac{\alpha_i\lambda_i}{(1+\lambda_it)^2}\bigg)^2-\bigg(\sum_{i=1}^n\frac{\alpha_i}{1+\lambda_it}\bigg)\bigg(\sum_{i=1}^n\frac{\alpha_i\lambda_i^2}{(1+\lambda_it)^3}\bigg)\bigg]\end{align*}
Multiplying by $-2\bigg(\sum_{i=1}^n\frac{\alpha_i}{1+\lambda_it}\bigg)^{3}$, we get the expression
\begin{align*}
&\sum_{1\leq i<j\leq n}\frac{\alpha_i\alpha_j\lambda_j^2}{(1+\lambda_it)(1+\lambda_jt)^3}
+\frac{\alpha_j\alpha_i\lambda_i^2}{(1+\lambda_jt)(1+\lambda_it)^3}
-\frac{2\alpha_i\alpha_j\lambda_i\lambda_j}{(1+\lambda_it)^2(1+\lambda_jt)^2}\\
&=\sum_{1\leq i<j\leq n}\frac{\alpha_i\alpha_j}{(1+\lambda_it)^3(1+\lambda_jt)^3}[\lambda_j^2(1+\lambda_it)^2+\lambda_i^2(1+\lambda_jt)^2
-2\lambda_i\lambda_j(1+\lambda_it)(1+\lambda_jt)]\\
&=\sum_{1\leq i<j\leq n}\frac{\alpha_i\alpha_j(\lambda_i-\lambda_j)^2}{(1+\lambda_it)^3(1+\lambda_jt)^3}\geq 0.
\end{align*}
Hence $g(t)$ concave, and so is $f(X)$. 

We can use the convexity directly to check $RE$ is less than or equal to unity. Indeed, $f$ is affine, in the sense that $f(c X) = c f(X)$ for any $c>0$. The concavity result that we proved implies that, with $A_i = X_i^\top X_i$,
$$\sum_{i=1}^k f(A_i)/k \le f\left (\sum_{i=1}^k A_i/k\right ).$$
By the affine nature of $f$, this result implies that $f$ is sub-additive. This can be checked to be equivalent to $RE \le 1$, finishing the proof.

\section{Computing optimal weights in the general framework, Section \ref{examples}}
\label{gen_w}
Recall that we have 
\begin{align*}
M(\hbeta_0) &= \E\| L_A  - \hat L_A(\hbeta_0)\|^2 
= \E\|A (\beta-\hbeta_0) + Z\|^2
= \tr\left(\Cov{ A\hbeta_0-Z}\right)\\
&= \tr\left(\Cov{ A\hbeta_0}\right) 
- 2\tr\left(\Cov{ A\hbeta_0, Z}\right)
+ \tr\left(\Cov{Z, Z}\right)\\
&= \tr\left(\Cov{\hbeta_0} A^\top A\right) 
- 2 \tr\left( A\Cov{\hbeta_0,Z}\right)
+ hd\sigma^2
\end{align*}
For OLS, we can calculate, recalling  $N  = \Cov{\ep,Z}$,  $\Cov{\hbeta, Z} = (X^\top X)^{-1} X^\top N.$
Hence
\begin{align*}
M(\hbeta)  &= \sigma^2\cdot \left[\tr\left[(X^\top X)^{-1} A^\top A\right] 
- 2 \tr\left[ A (X^\top X)^{-1} X^\top N\right]
+ hd\right].
\end{align*}
For the distributed estimator $\hbeta_{dist}(w) = \sum_i w_i \hbeta_i$, we have 
\begin{align*}
\Cov{\hbeta_{dist}, Z} 
&= \Cov{\sum_i w_i (X_i^\top X_i)^{-1} X_i^\top \ep_i, Z} \\
&= \sum_i w_i (X_i^\top X_i)^{-1} X_i^\top \Cov{ \ep_i, Z}
= \sum_i w_i (X_i^\top X_i)^{-1} X_i^\top N_i
\end{align*}
Above, we denoted $N_i:=\Cov{ \ep_i, Z}$.  Therefore, 
$$M(\hbeta_{dist}) 
= \sigma^2\cdot \left(\sum_{i=1}^k w_i^2 \cdot \tr \left[ (X_i^\top X_i)^{-1} A^\top A\right]
- 2 w_i \cdot \tr\left[ A(X_i^\top X_i)^{-1} X_i^\top N_i\right]\right ) + \sigma^2hd.$$ 
To find the optimal weights, we consider more generally the quadratic optimization problem 
$$\min_{w \in \R^k} \sum_{i=1}^k \frac{a_i}{2} w_i^2 - b_i w_i $$
subject to $\sum_{i=1}^k w_i = 1.$ We assume that $a_i >0$. In that case, the problem is convex, and we can use a simple Lagrangian reformulation to solve it.  Note that we do not impose the constraint $w_i \ge 0$, because in principle one could allow negative weights, and because it is usually satisfied without imposing the constraint. 

Denoting by $\Psi(w)$ the objective, we consider the problem of minimizing the Lagrangian $\Psi_\lambda(w) = \Psi(w) -\lambda (\sum_i w_i -1)$. It is easy to check that the condition $\frac{\partial\Psi_\lambda}{\partial w_i} = 0$ reduces to
$$w_i = \frac{\lambda+b_i}{a_i}.$$
In order for the constraint $\sum_{i=1}^k w_i = 1$ to be satisfied, we need that 
$$\lambda  = \lambda^*:= \frac{1-\sum_{i=1}^k\frac{b_i}{a_i}}{\sum_{i=1}^k\frac{1}{a_i}}.$$
Plugging back this value of $\lambda$, we obtain the optimal value or the weights $w_i^*$.  To apply this result to our problem, we choose $a_i = \tr \left[ (X_i^\top X_i)^{-1} A^\top A\right]$, and $b_i $ $= $ $\tr\left( A(X_i^\top X_i)^{-1} X_i^\top N_i\right)$. This finishes the proof.

\section{Proof of Theorem \ref{gen_det}}
\label{pf:gen_det}

%\begin{proof}(Of Theorem \ref{gen_det}.)
We want to show 
$$\hSigma^{-1} \asymp \Sigma^{-1} \cdot e_p.$$
As mentioned, the proof of this result relies on the generalized Marchenko-Pastur theorem of \cite{rubio2011spectral}. From that result, we have under the stated assumptions
$$(\hSigma-z I)^{-1} \asymp (x_p\Sigma-z I)^{-1},$$
where $x_p=x_p(z), e_p=e_p(z)$ are the unique solutions of the system 

\begin{align*}
e_p &= \frac{1}{p} \tr\left[\Sigma (x_p\Sigma- z I )^{-1}\right],\, 
x_p = \frac{1}{n} \tr\left[\Gamma (I+\gamma_p e_p \Gamma)^{-1}\right].
\end{align*}

From section 2.2 of \cite{paul2009no}, when the e.s.d of $\Sigma$ converges to $H$ and the e.s.d of $\Gamma$ converges to $G$, $x_p$ and $e_p$ will converge to $x$ and $e$ respectively, where $x=x(z)$ and $e=e(z)$ are the unique solutions of the system
\begin{align*}
e &= \int\frac{t}{tx-z}dH(t),\,
x = \int\frac{t}{1+\gamma te}dG(t).
\end{align*}

Recall that, in Section \ref{pf:are_ell}, we have the system of equations
$$
\delta=\int\frac{\gamma t}{-z(1+\tilde{\delta}t)}d\mu_H(t),\,\,
\tilde{\delta}=\int\frac{t}{-z(1+\delta t)}d\mu_G(t).
$$
Then, it's easy to check that $\delta=\gamma e$ and $x=-z\tilde{\delta}$. We will use these relations later.

Now, we want to show that we can take $z=0$, i.e.
$$\hSigma^{-1} \asymp (x_p(0)\Sigma)^{-1}=\Sigma^{-1}\cdot e_p(0).$$ So for a given sequence of matrices $C_p$ with bounded trace norm we need to bound

\begin{align*}
\Delta_p :&= \tr[C_p (\hSigma^{-1} - (x_p(0)\Sigma)^{-1})]\\ 
&=\tr[C_p (\hSigma^{-1} -  (\hSigma-zI)^{-1})] + 
\tr[C_p ((\hSigma-zI)^{-1}- (x_p\Sigma-zI)^{-1}))]\\
 &+ \tr[C_p ((x_p\Sigma-zI)^{-1} - (x_p\Sigma)^{-1}) ]+\tr[C_p((x_p\Sigma)^{-1}-(x_p(0)\Sigma)^{-1})] \\
& = \Delta_{p}^1+\Delta_{p}^2+\Delta_{p}^3+\Delta_p^4.
\end{align*}

We can bound the four error terms in turn:
\benum 
\item Bounding $\Delta_{p}^1$:

We have 
$$D_1(z) = (\hSigma-z I)^{-1} - \hSigma^{-1} = z (\hSigma-z I)^{-1} \hSigma^{-1}.$$
Hence, the operator norm of $D(z)$ can be bounded as
$$\|D_1(z)\|_{op} \le \frac{2|z|}{\lambda_{\min}(\hSigma)^2}$$
for sufficiently small $z$.

Recall that $X=\Gamma^{1/2}Z\Sigma^{1/2}$, where $\Gamma$ is a diagonal matrix with positive entries and $\Sigma$ is a symmetric positive definite matrix. Let us consider the least singular value of $X$. By assumption, the entries of $\Gamma$ and the eigenvalues of $\Sigma$ are uniformly bounded below by some constant K. So we can bound $\sigma_{\min}(X)$ as follows:
$$
\sigma_{\min}(X)=\sigma_{\min}(\Gamma^{1/2}Z\Sigma^{1/2})\geq\sigma_{\min}(\Gamma^{1/2})\sigma_{\min}(Z)\sigma_{\min}(\Sigma^{1/2})\geq K\cdot\sigma_{\min}(Z).
$$
By using the bound above, we have
\begin{align*}
\lambda_{\min}(\hat{\Sigma})
&=\lambda_{\min}(\frac{X^\top X}{n})
=\frac{(\sigma_{\min}(X))^2}{n}\geq \frac{K^2\cdot(\sigma_{\min}(Z))^2}{n}\\
&=K^2\cdot\lambda_{\min}(\frac{Z^\top Z}{n})\to_{a.s.}K^2(1-\sqrt\gamma)^2,
\end{align*}
where the final step comes from the well-known Bai-Yin law \citep{bai2009spectral}.

Thus, we conclude that
\begin{align*}
\lim_{p\rightarrow+\infty}|\Delta_{p}^1| &=\lim_{p\rightarrow+\infty}|\tr[C_p (\hSigma^{-1} -  (\hSigma-zI)^{-1})]|\\ 
&\le \lim_{p\rightarrow+\infty}\|C_p\|_{tr} \cdot\|D_1(z)\|_{op}  \le\lim_{p\rightarrow+\infty} \|C_p\|_{tr}\cdot\frac{2|z|}{(K^2\cdot\lambda_{\min}(Z^\top Z/n))^2}\le C' |z|.
\end{align*}
This holds almost surely, for some fixed constant $C'>0$.\\

\item Bounding $\Delta_{p}^2$:

This just follows Theorem 1 of \cite{rubio2011spectral}:

$$|\Delta_p^2|=\tr[C_p ((\hSigma-zI)^{-1}- (x_p\Sigma-zI)^{-1}))] \to_{a.s.} 0.$$

\item Bounding $\Delta_{p}^3$:

By a similar logic, we can obtain a bound on the operator norm of 

$$D_2(z) = (x_p\Sigma-z I)^{-1} - (x_p\Sigma)^{-1}$$
for sufficiently small $z$, of the form 
$$\|D_2(z)\|_{op} \le \frac{2|z|}{ |x_p(z)|^2\cdot \lambda_{\min}(\Sigma)^2}$$
for sufficiently small $z$. Again, we have assumed that the smallest eigenvalues of $\Sigma$ are always bounded away from zero, so that $\lambda_{\min}(\Sigma)>c>0$ for some fixed constant $c>0$. Since $x_p(z)\rightarrow x(z)=-z\tilde{\delta}(z)$ as $p\rightarrow+\infty$, and we know that $-z\tilde{\delta}(z)$ is analytic in a neighborhood of the origin with $x(0)=\lim_{z\to0}[-z\tilde{\delta}(z)]=\tilde{\rho}(\{0\})>0$. We can argue that $|x(z)|$ is bounded below in a neighborhood of the origin.

So we conclude that
\begin{align*}
\lim_{p\rightarrow+\infty}|\Delta_{p}^3| &=\lim_{p\rightarrow+\infty}|\tr[C_p ((x_p\Sigma-z I)^{-1} - (x_p\Sigma)^{-1})]|\\ 
&\le \lim_{p\rightarrow+\infty}\|C_p\|_{tr} \cdot\|D_2(z)\|_{op}\\
&\le \limsup\|C_p\|_{tr} \cdot\frac{2|z|}{ |x(z)|^2\cdot \lambda_{\min}(\Sigma)^2} \le C'' |z|.
\end{align*}
This holds almost surely, for some fixed constant $C''>0$.\\

\item Bounding $\Delta_{p}^4$:

\begin{align*}
\lim_{p\rightarrow+\infty}|\Delta_{p}^4| &=\lim_{p\rightarrow+\infty}|\tr[C_p((x_p\Sigma)^{-1}-(x_p(0)\Sigma)^{-1})]\\ 
&\le \lim_{p\rightarrow+\infty}\|C_p\|_{tr} \cdot\frac{|x_p(z)^{-1}-x_p(0)^{-1}|}{\lambda_{\min}(\Sigma)}\\
&\le  \limsup\|C_p\|_{tr} \cdot\frac{|x(z)^{-1}-x(0)^{-1}|}{\lambda_{\min}(\Sigma)}\le C''' |z|.
\end{align*}
This holds almost surely for some fixed constant $C'''>0$, since $x(z)$ is analytic near the origin with $x(0)>0$.
\eenum

Combining these, we have 
$$\lim_{p\rightarrow+\infty}|\Delta_p|=\lim_{p\rightarrow+\infty}|\Delta_p^1+\Delta_p^2+\Delta_p^3+\Delta_p^4|\le(C'+C''+C''')|z|.$$
Since $|z|$ can be arbitrarily small, we conclude that, almost surely
$$\lim_{p\rightarrow+\infty}|\Delta_p|=\lim_{p\rightarrow+\infty}\tr[C_p(\hSigma^{-1} - (x_p(0)\Sigma)^{-1})]=0.$$
This finishes the proof. 
%\end{proof}

\section{Proof of Theorem \ref{calcrules}}
\label{pf:calcrules}

%\begin{proof}
Recall that we defined $A_n \asymp B_n$ if $\lim_{n\to\infty}\left|\tr\left[E_n(A_n-B_n)\right]\right|=0$ a.s., for any standard sequence $E_n$ (of symmetric deterministic matrices with bounded trace norm). Below, $E_n$ will always denote such a sequence. 

\benum
\item 
The three required properties are that the $\asymp$ relation is reflexive, symmetric and transitive. The reflexivity and symmetry are obvious. To verify transitivity, we suppose $A_n \asymp B_n$ and $B_n \asymp C_n$. Then, for any standard sequence $E_n$, by the triangle inequality, 
 
$$\left|\tr\left[E_n(A_n-C_n)\right]\right| \le 
\left|\tr\left[E_n(A_n-B_n)\right]\right| + \left|\tr\left[E_n(B_n-C_n)\right]\right|.$$ 

Since the two sequences on the right hand side converge to zero almost surely, the conclusion follows. 

\item Let $D_n^1 = A_n-B_n$ and $D_n^2 = C_n-D_n$. Then we can bound by the triangle inequality

$$\left|\tr\left[E_n(D_n^1+D_n^2)\right]\right| \le 
\left|\tr\left[E_nD_n^1\right]\right| + \left|\tr\left[E_nD_n^2\right]\right|.$$ 

As before, the two sequences on the right hand side converge to zero almost surely, so the conclusion follows.

\item We need to show that $A_nB_n\asymp A_nC_n$. Let $E_n$ be any standard sequence. For this it is enough to show that $A_nE_n$ is still a standard sequence. However, this is clear, because
$$\lim\sup\|A_nE_n\|_{tr} 
\le \lim\sup\|A_n\|_{op}\|E_n\|_{tr}
\le \lim\sup\|A_n\|_{op}\lim\sup\|E_n\|_{tr} <\infty.$$ 

\item We know that $\lim_{n\to\infty}\left|\tr\left[E_n(A_n-B_n)\right]\right|=0$ a.s., for any standard sequence $E_n$. Consider $E_n = n^{-1} I_n$. Then $\|E_n\|_{tr} = 1$, so that $E_n$ is a standard sequence. Therefore,  $\lim_{n\to\infty}\left|\tr\left[A_n-B_n\right]\right|=0$ a.s., as desired.

\item This is a direct consequence of the trace property. 
\eenum
%\end{proof}

\section{Applications of the calculus}
\label{apps_calc}
In this section we briefly sketch several applications of the calculus of deterministic equivalents. We emphasize that in each case, there are other proof techniques, but they tend to be more case-by-case. The calculus provides a unified set of methods using which separate results can be seen as applications of the same approach.

{\bf Risk of ridge regression.} We can get a simpler derivation of certain previously found formulas for the risk of ridge regression \citep{dobriban2018high}. Considering Theorem 2.1 in that work, the finite-sample predictive risk of ridge regression involves $\tr(\Sigma(\hSigma+\lambda I)^{-1})$. Using the calculus of deterministic equivalents, we can compute this quantity using the following steps, starting with the main equivalence:
\begin{align*}
(\hSigma+\lambda I)^{-1} &\asymp (x_p\Sigma+\lambda I)^{-1}\\
\Sigma(\hSigma+\lambda I)^{-1} &\asymp \Sigma(x_p\Sigma+\lambda I)^{-1}\\
p^{-1}\tr[\Sigma(\hSigma+\lambda I)^{-1}] &- p^{-1}\tr[\Sigma(x_p\Sigma+\lambda I)^{-1}]\to_{a.s.}0.
\end{align*}
It remains to find the limit of the right hand side. This can be done by using the fixed point equation defining $x_p$. From the proof of Theorem \ref{gen_det}, we have that $x_p$ and the associated scalar $e_p$ are the unique solutions of the system 
\begin{align*}
e_p &= \frac{1}{p} \tr\left[\Sigma (x_p\Sigma+ \lambda I )^{-1}\right],\, 
x_p = (1+\gamma_p e_p)^{-1}.
\end{align*}
This gives a characterization of $x_p$ that cannot be simplified, and hence solves the problem to the extent that random matrix theory can, recovering the known results in a simpler way \citep{dobriban2018high}.

{\bf Fine-grained structure of ridge regression.} In a recent work \citep{dobriban2019ridge}, we have studied ridge regression on a deeper level, including presenting an equivalent for ridge as a sum of a covariance matrix dependent transform of the parameter vector, and another transform of the noise. In that work, we have relied significantly on the calculus of deterministic equivalents.

{\bf Distributed ridge regression.} In a follow-up work to the present one, we study one-shot distributed ridge regression \citep{dobriban2019one}. In that work, we rely heavily on the calculus of deterministic equivalents, and we think that this is a good example of results that would be hard or complicated to obtain otherwise. We refer the reader to \cite{dobriban2019one} for details.

{\bf Gradient flow for least squares.} To study gradient flow for least squares regression, \cite{alnur} also require the limiting behavior of $\tr(\Sigma(\hSigma+\lambda I)^{-1})$ (see the proof of their Theorem 6). Our calculus can thus be used as an alternative way to derive their results.

{\bf Interpolation.} The recent work of \cite{hastie2019surprises} has studied high-dimensional interpolation using techniques from random matrix theory. Some of their arguments can be phrased and simplified in the language of the calculus of deterministic equivalents. Consider for instance their Lemma 2, whose proof in version 3 of their arXiv preprint relies on the results of \cite{rubio2011spectral}. This proof can be simplified if expressed in the natural way in the calculus:  In equation 30, they want to find the limit
$$\lim_{z\to 0^+}z\beta^\top (\hSigma+zI)^{-1}\beta.$$
Using the calculus of deterministic equivalents, we know that for any fixed $z$, and any fixed $\beta$-sequence with bounded norm
\begin{align*}
z(\hSigma+z I)^{-1} &\asymp z(x_p\Sigma+z I)^{-1}\\
z\beta\beta^\top(\hSigma+z I)^{-1} &\asymp 
z\beta\beta^\top(x_p\Sigma+z I)^{-1}\\
z\beta^\top(\hSigma+z I)^{-1}\beta &- z\beta^\top(x_p\Sigma+z I)^{-1}\beta\to_{a.s.}0.
\end{align*}
These results are not uniform in $z$, but this could be proved with a bit more effort. Now, in their work $\Sigma=I_p$, hence
\begin{align*}
z\beta^\top(x_p\Sigma+z I)^{-1}\beta&=z/(x_p+z)\|\beta\|^2.
\end{align*}
Moreover, 
\begin{align*}
e_p &= 1/(x_p+z),\, 
x_p = 1/(1+\gamma_p e_p).
\end{align*}
After some elementary calculation, we find that the limit as $z\to 0^+$ when $\|\beta\|^2=r^2$ is $(1-1/\gamma)r^2$, which agrees with \cite{hastie2019surprises}.

{\bf Heteroskedastic PCA.} The calculus of deterministic equivalents can be used to simplify certain arguments used to study heteroskedastic PCA \citep{hong2018asymptotic,hong2018optimally}. Specifically, for Lemma 5 in \cite{hong2018optimally}, the key problem is to find the limit of $\zeta \tr[W R W]$, where $R = (\zeta^2 I - \tilde E^H \tilde E)^{-1}$ is the resolvent of a Marchenko-Pastur type matrix $\tilde E^H \tilde E$. The calculus of deterministic equivalents leads, using notation that has to be changed mutatis mutandis, and $\Sigma:=\E \tilde E^H \tilde E$, to
\begin{align*}
\zeta (\zeta^2 I -  \tilde E^H \tilde E)^{-1} &
\asymp \zeta(\zeta^2 I - x_p \Sigma)^{-1}\\
\zeta W (\zeta^2 I -  \tilde E^H \tilde E)^{-1}W  &
\asymp \zeta W (\zeta^2 I - x_p \Sigma)^{-1}W \\
n^{-1} \zeta  \tr W (\zeta^2 I -  \tilde E^H \tilde E)^{-1}W &  -
n^{-1} \zeta  \tr  W (\zeta^2 I - x_p \Sigma)^{-1}W \to_{a.s.}0.
\end{align*}
Then, using the specific expression of $\Sigma$, which is diagonal in this case, it is not hard to recover the statement of Lemma 5 from \cite{hong2018optimally}.

{\bf ePCA theory.} Another application of the calculus of deterministic equivalents is to develop a rigorous analysis for spiked covariance models in exponential families, which were  proposed in \cite{liu2018pca}. This is an ongoing project of one of the authors, and we refer to the forthcoming manuscript for details \citep{dobriban2019pca}.

\section{Elliptical models}
\label{estim_ellip_supp}

We study the setting of elliptical data. In this model the data samples may have different scalings, having the form $x_i = g_i^{1/2}\Sigma^{1/2} z_i$, for some vector $z_i$ with iid entries, and for datapoint-specific \emph{scale parameters} $g_i$. Arranging the data as the rows of the matrix $X$, that takes the form $$X = \Gamma^{1/2} Z \Sigma^{1/2},$$ where $Z$ and $\Gamma$ are as before: $Z$ has iid standardized entries, while $\Sigma$ is the covariance matrix of the features. Now $\Gamma$ is the diagonal \emph{scaling matrix} containing the scales $g_i$ of the samples. This model has a long history in multivariate statistical analysis \citep[e.g.,][]{mardia1979multivariate}. %Beyond the normal distribution, the model has also been analyzed from a random matrix theory point of view, see e.g., \cite{lixin2007spectral,paul2009no,el2009concentration,couillet2014analysis}.

%G - uniform on [1,2]
%eta(x) = \int_0^1 1/(1+xt) dt = log(1+xt)/x|_1^2

In the elliptical model, we find the following expression for the ARE.
\begin{theorem}[ARE for elliptical models]
\label{are_ell}
Consider the above high-dimensional asymptotic limit, where the data matrix is random, and the samples have the form $X = \Gamma^{1/2} Z\Sigma^{1/2}$. Suppose that, as $n_i\to\infty$ with $p/n_i\to\gamma_i\in(0,1)$, the $e.s.d.$ of $\Gamma$ converges weakly to $G$, the  $e.s.d.$ of each $\Gamma_i$ converges weakly to some $G_i$, and that the $e.s.d.$ of $\Sigma$ converges weakly to $H$. Suppose that $H$ is compactly supported away from the origin, $G$ is also compactly supported and does not have point mass at the origin.
 Then, the ARE has the form
\beqs
ARE =  f(\gamma, G) \cdot 
\sum_{i=1}^k  \frac{1} {f(\gamma_i, G_i)}.
\eeqs
\end{theorem}

See the following sections (Section \ref{are_ell}) for the proof.

There are two implicit relations in the above formula. First, $\sum 1/\gamma_i = 1/\gamma$, because $\sum n_i/p = n/p$. Second, $n \cdot G = \sum_{i=1}^k n_i \cdot G_i$, or equivalently $G/\gamma = \sum_{i=1}^k G_i/\gamma_i$, because $\Gamma$ contains all entries of each $\Gamma_i$.

For the special case when all aspect ratios $\gamma_i$ are equal, and all scale distributions $G_i$ are equal to $G$, we can say more about the ARE. We have the following theorem.

\begin{theorem}[Properties of ARE for elliptical models]
\label{are_ell_props}
Consider the behavior of distributed regression in elliptical models under the conditions of Theorem \ref{are_ell}. Suppose that the data sizes $n_i$ on all machines are equal, so that $\gamma_i = \gamma_j=k\gamma$ for all $i,j$. Suppose moreover that the scale distributions $G_i$ on all machines are also equal. Then, the ARE has the following properties 
\begin{enumerate}
\item It can be expressed equivalently as
$$
ARE(k)
=\frac{k\cdot\eta_{G}^{-1}(1-\gamma)}{\eta_{G}^{-1}(1-k\gamma)}
=\frac{k\cdot f(\gamma,G)}{f(k\gamma,G)}
 = \frac{e(\gamma, G)}{e(k\gamma, G)}.
$$
Here $\eta_G$ is the $\eta$-transform of $G$, $f$ is defined above, while $e$ is the unique positive solution of the equation
\begin{align*}
\int\frac{se}{1+\gamma se}dG(s)=1.
\end{align*}
\item
Suppose also that $G$ does not have a point mass at the origin. Then, the ARE is a strictly decreasing smooth convex function for $k\in[1,1/\gamma]$. Here $k$ is viewed as a continuous variable. Moreover $ARE(1)=1$, and 
$$\lim_{k\rightarrow 1/\gamma}ARE(k)=0.$$
%\item 
\end{enumerate}
\end{theorem}

See Section \ref{pf:are_ell_props} for the proof. These theoretical formulas again match simulation results well, see Figure \ref{f2}. On that figure, we use the same simulation setup as for Figure \ref{f1}, and in addition we choose the scale distribution to be uniform on $[0,1]$.

%\ed{Show some plots?}

\begin{figure}
\begin{subfigure}{.45\textwidth}
  \centering
\includegraphics[scale=0.48]
{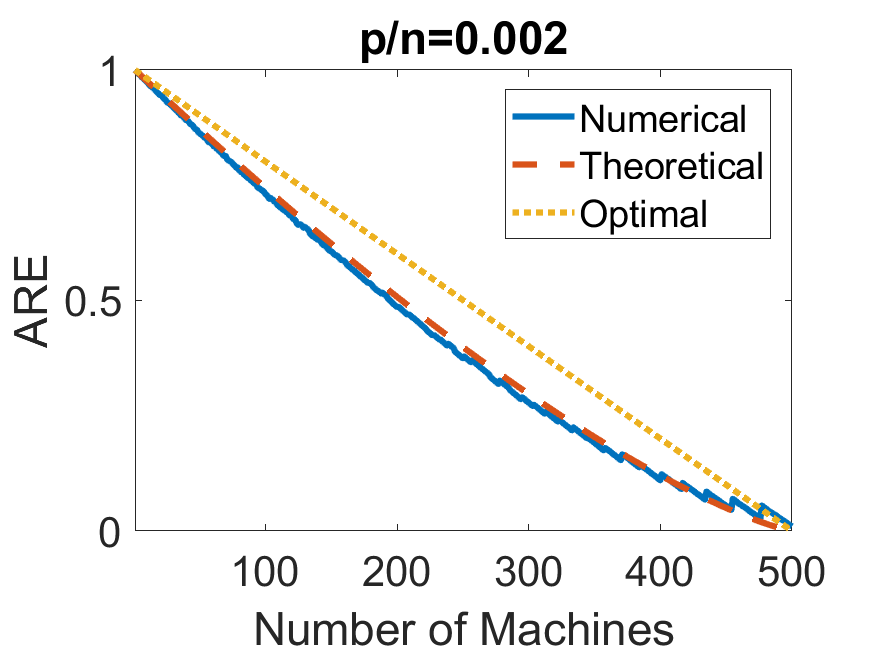}
\end{subfigure}
\begin{subfigure}{.45\textwidth}
  \centering
\includegraphics[scale=0.48]
{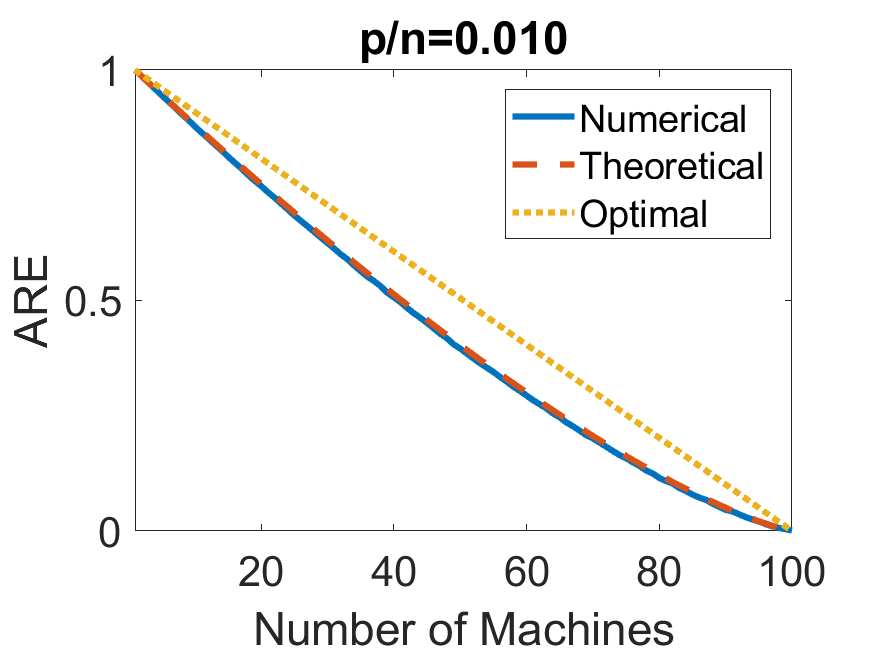}
\end{subfigure}
\caption{Comparison of empirical and theoretical ARE for elliptical distributions. Left: $n=10,000$, $p=20$. Right: $n=10,000$, $p=100$. }
\label{f2}
\end{figure}

The ARE for a constant scale distribution is a straight line in $k$, ARE$(k)=(1-k\gamma)/(1-\gamma)$. For a general scale distribution, the graph of ARE is a curve below that straight line. The interpretation is that for elliptical distributions, there is a larger efficiency loss in one-step averaging. Intuitively, the problem becomes "more non-orthogonal" due to the additional variability from sample to sample. 

\subsection{Worst-case analysis}
\label{worst_case}
It is natural to ask which elliptical distributions are difficult for distributed estimation. That is, for what scale distributions $G$ does the distributed setting have a strong effect on the learning accuracy? Intuitively, if some of the scales are much larger than others, then they "dominate" the problem, and may effectively reduce the sample size. 

Here we show that this intuition is correct. We find a sequence of scale distributions $G_\tau$ such that distributed estimation is "arbitrarily bad", so that the ARE decreases very rapidly, and approaches zero even for two machines. %Therefore, there is no worst case $G$, and the ARE can be arbitrarily small. 

\begin{proposition}[Worst-case example]
\label{ex1}
Consider elliptical models with scale distributions that are a mixture of small variances $\tau$, and larger variances $1/\gamma$, with weights $1-\gamma$ and $\gamma$, i.e., $G_\tau=(1-\gamma)\delta_\tau+\gamma\delta_{1/\gamma}$. Then, as $\tau\to0$, we have ARE$(2)=O(\tau^{1/2})\rightarrow 0$. Therefore, the relative efficiency for any $k\ge 2$ tends to zero.
\end{proposition}

See Section \ref{pf:ex1} for the proof. 

Next, we consider more general scale distributions that are a mixture of small scales $\tau$, and larger scales $\alpha\tau$, with arbitrary weights $1-c$ and $c$: 

\beq
G=(1-c)\delta_\tau+c\delta_{\alpha\tau},
\label{gen_mix}
\eeq 
where $c\in[0,1], \tau>0$, and $\alpha>1$. To gain some intuition about the setting where $\alpha\to\infty$, we notice that only the large scales contribute a non-negligible amount to the sample covariance matrix $X^\top X$. Therefore, the sample size is reduced by a factor equal to the fraction of large scales, which equals $c$. More specifically, we have
$$n^{-1}X^\top X = n^{-1}\sum_{j=1}^n x_j x_j^\top 
= \alpha^2 n^{-1}\sum_{j\le cn} z_j z_j^\top +n^{-1}\sum_{j>cn} z_j z_j^\top
\approx \alpha^2 n^{-1}\sum_{j\le cn} z_j z_j^\top.
$$
The last approximation follows because the sample covariance matrix has $p$ eigenvalues, out of which we expect $\min(cn,p)$ to be large, of the order of $\alpha^2\gg1$. The remaining $\max(p-cn,0)$ are smaller, of unit order. Thus, heuristically, this matrix is well approximated by a scaled sample covariance matrix of $cn$ vectors. Therefore, the sample size is reduced to the number of large scales. If $cn<p$, the matrix is nearly singular, while if $cn>p$, it is well-conditioned. This should provide some intuition for the results to follow. 

\begin{theorem}[More general worst-case example]
\label{ex2}
Consider elliptical models with scale distribution $$G=(1-c)\delta_\tau+c\delta_{\alpha\tau},$$ as in \eqref{gen_mix}, where $c\in[0,1]$, and $\tau>0$. When $\alpha$ tends to infinity, the ARE will depend on $c$ and $\gamma$ as summarized in Table \ref{worsetab}. 
\begin{table}[]
\renewcommand{\arraystretch}{2}
\centering
\caption{ARE as a function of number of machines $k\geq 1$ and the fraction of large scales $c=M\gamma$, as the ratio $\alpha$ tends to infinity. Note that for any fixed $M$, we need $k/M\ge 1/M$ for the result to be well-defined. We also need $k$ to be an integer. In the table below, we mark by a $\diagup$ the cases where this is not satisfied.}
\label{worsetab}
\begin{tabular}{|c|c|c|c|c|c|}
\hline
\diagbox{$\frac{k}{M}$}{ARE($k$)}{$M$}	&$0<M<1$&$M=1$&$1<M<2$&$M=2$&$M>2$\\
\hline
$\frac{k}{M}<1$&$\diagup$&$\diagup$&$1$&$1$&$\frac{c-k\gamma}{c-\gamma}$\\
\hline
$\frac{k}{M}=1$&$\diagup$&$1$&$\diagup$&$O(\alpha^{-1/2})$&$O(\alpha^{-1/2})$\\
\hline
$\frac{k}{M}>1$&$\frac{k(\gamma-c)(1-k\gamma)}{(1-\gamma)(k\gamma-c)}$&$O(\alpha^{-1/2})$&$O(\alpha^{-1})$&$O(\alpha^{-1})$&$O(\alpha^{-1})$\\
\hline
\end{tabular}
\end{table}
\end{theorem}

See Section \ref{pf:ex2} for the proof. 

As an example of special interest, if $c=M\gamma$ for some $M>2$, then
$$
\lim_{\alpha\rightarrow+\infty}ARE(k)
=
\begin{cases}
\frac{c-k\gamma}{c-\gamma},~~k<M,\\
O(\alpha^{-1/2}),~~k=M,\\
O(\alpha^{-1}),~~k>M.
\end{cases}
$$
As before, this result can be understood in terms of reducing the effective sample size to $cn$. When $k<(cn)/p$, i.e., $k<M$, each local problem can be well-conditioned. However, when $k>(cn)/p$, i.e., $k>M$, all local OLS problems are ill-conditioned, so the ARE is small.

\newcommand{\FW}{0.4\textwidth}
\newcommand{\TRA}{0}%{2}
\newcommand{\TRB}{0}
\newcommand{\TRC}{0}%{2}
\newcommand{\TRD}{0}
\newcommand{\PBW}{0.3}

\begin{figure}[h]
\centering
\begin{tabular}{cc}
\includegraphics[width=\FW, trim = \TRA mm \TRB mm \TRC mm \TRD mm, clip = TRUE]{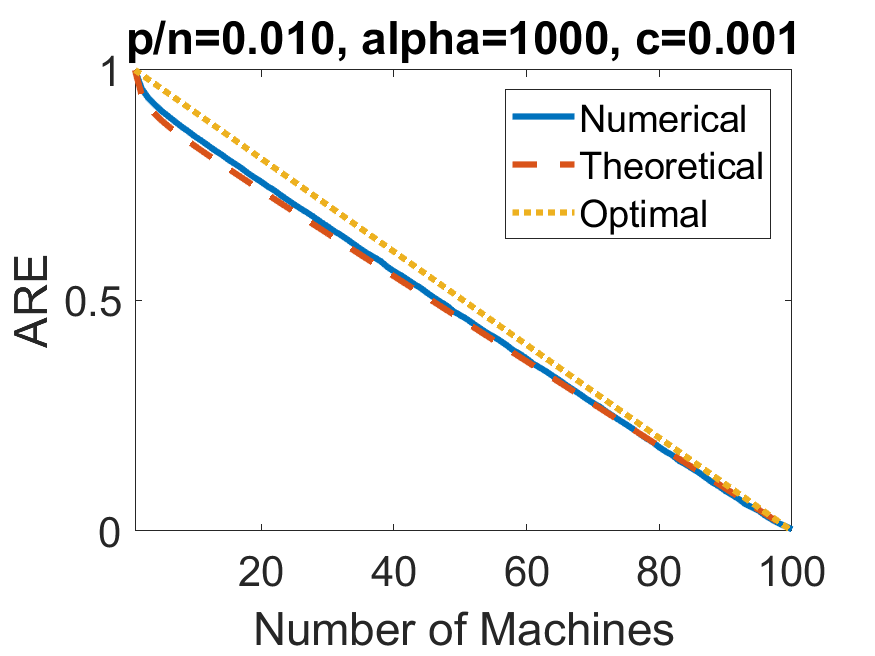} &
\includegraphics[width=\FW, trim = \TRA mm \TRB mm \TRC mm \TRD mm, clip = TRUE]{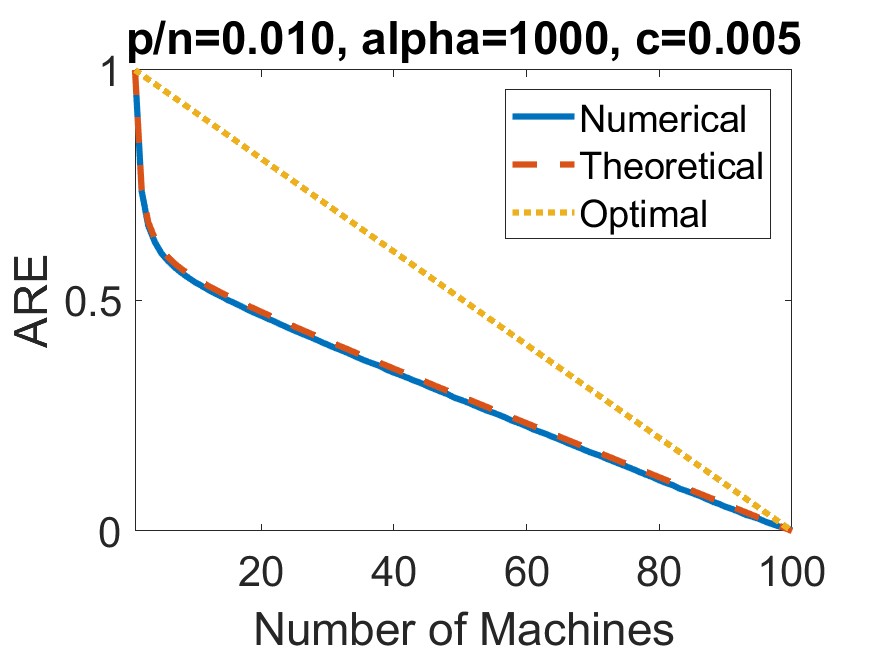} \\
\includegraphics[width=\FW, trim = \TRA mm \TRB mm \TRC mm \TRD mm, clip = TRUE]{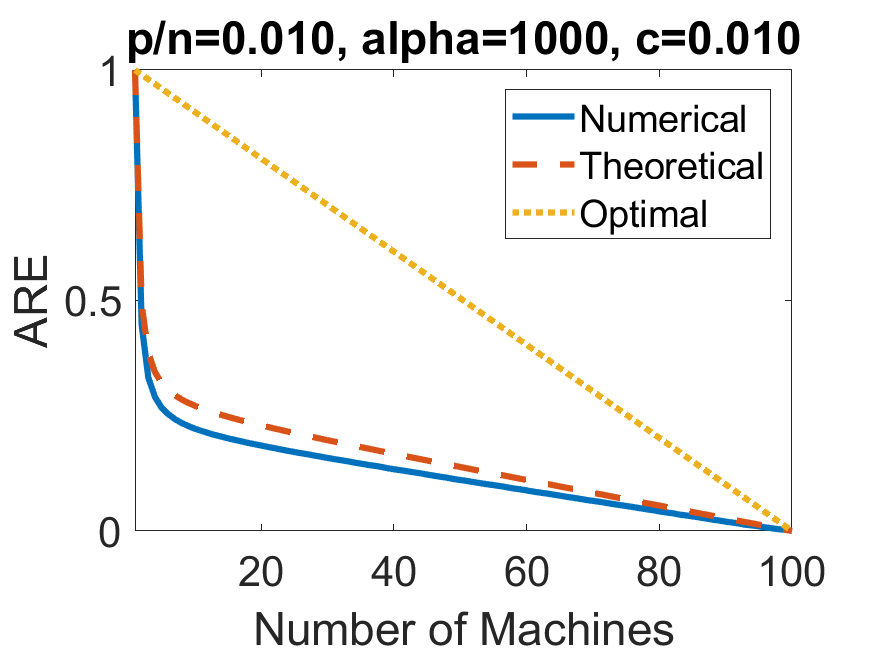} &
\includegraphics[width=\FW, trim = \TRA mm \TRB mm \TRC mm \TRD mm, clip = TRUE]{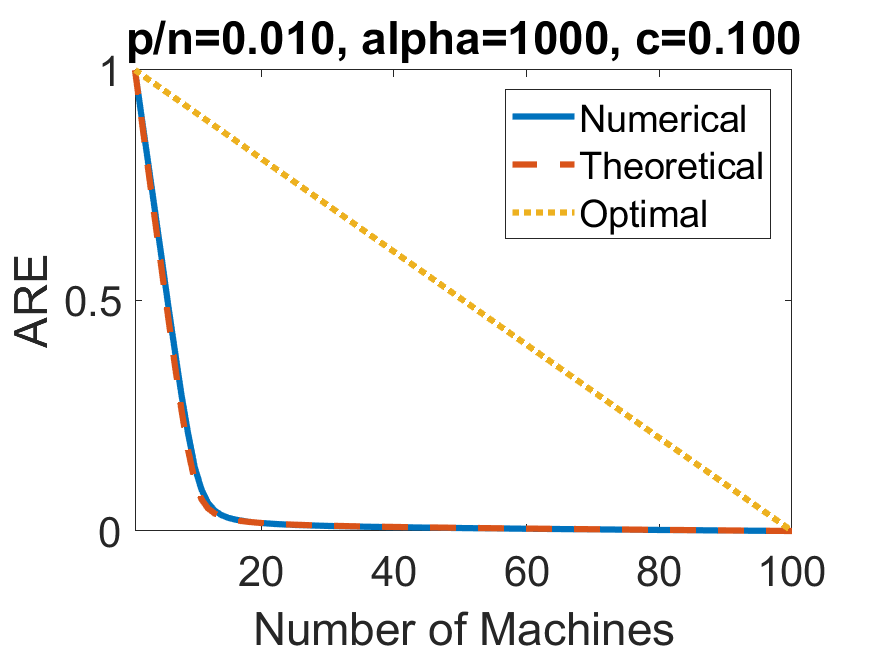}  \\
\includegraphics[width=\FW, trim = \TRA mm \TRB mm \TRC mm \TRD mm, clip = TRUE]{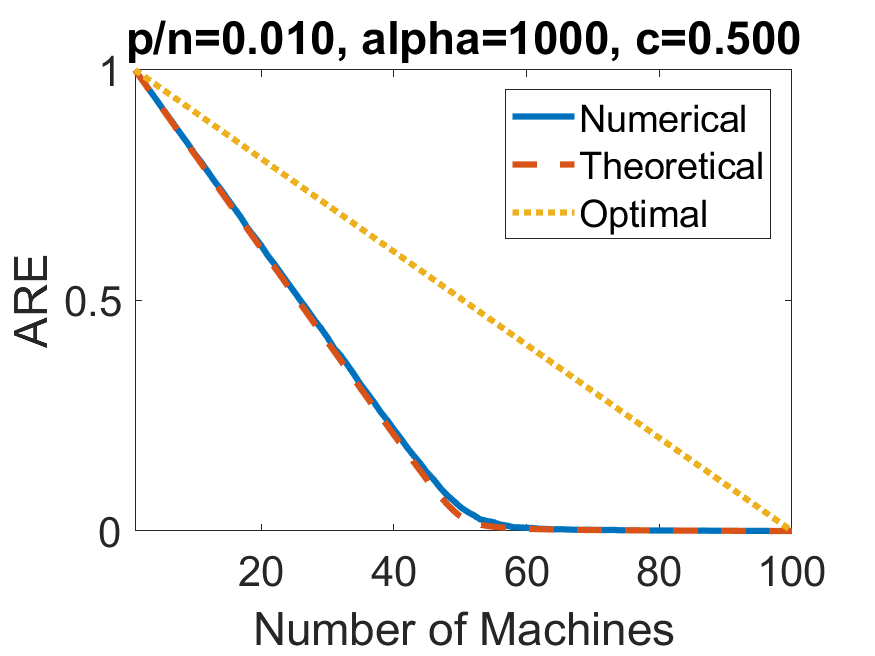} &
\includegraphics[width=\FW, trim = \TRA mm \TRB mm \TRC mm \TRD mm, clip = TRUE]{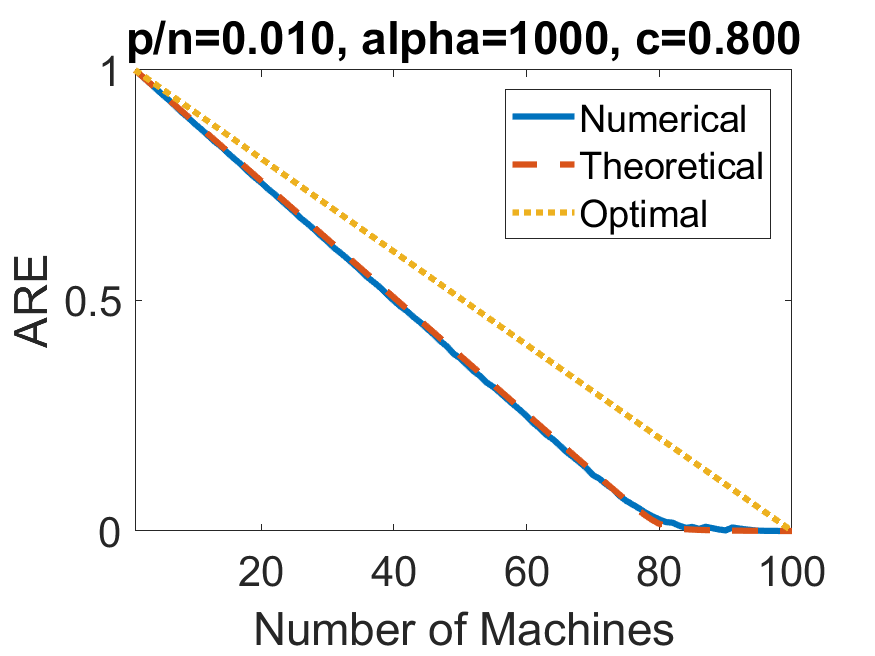}  \\
\end{tabular}
\caption{Comparison of empirical and theoretical ARE for worst case  elliptical distributions. We fix $n=10,000$, $p=100$, so that $\gamma = 0.01$. We also fix $\alpha = 1000$. We vary $c$. }
\label{f3}
\end{figure}

\subsection{Proof of Theorem \ref{are_ell}}
\label{pf:are_ell}

%\begin{proof}

Consider the matrix
\begin{align*}
\hSigma =\frac{1}{n}X^\top X&=\frac{1}{n}\Sigma^{\frac{1}{2}}Z^\top \Gamma Z\Sigma^{\frac{1}{2}}.
\end{align*}

Recall that the $e.s.d.$ of $\Gamma$ converges to $G$, and that the $e.s.d.$ of $\Sigma$ converges to $H$. According to \cite{paul2009no}, with probability 1, the $e.s.d.$ of $\hSigma$ converges to a distribution $F$, whose Stieltjes transform $m(z),z\in\mathbb{C}^+$ is given by 
\begin{align*}
m(z)=\int\frac{1}{t\int\frac{s}{1+\gamma se}dG(s)-z}dH(t),
\end{align*}
where $e=e(z)$ is the unique solution in $\mathbb{C}^+$ of the equation
$$
e=\int\frac{t}{t\int\frac{s}{1+\gamma se}dG(s)-z}dH(t).
$$
Since $\tr[(X^\top X)^{-1}]\rightarrow \gamma m(0)$, we only need to solve for $m(0)$. As we will show below, we can take $z\to 0$, and obtain that
\begin{align*}
m(0)=\frac{\int\frac{1}{t}dH(t)}{\int\frac{s}{1+\gamma se}dG(s)},
\quad e(0)=\frac{1}{\int\frac{s}{1+\gamma se}dG(s)},
\end{align*}
hence $e:=e(0)$ can be checked to be the unique positive solution to the equation
\begin{align*}
\int\frac{se}{1+\gamma se}dG(s)=1.
\end{align*}
To make this rigorous, we need to use some results from \cite{couillet2014analysis}. Let $\mu_F, \mu_G, \mu_H$ be the probability measures corresponding to the distributions $F, G, H$. Our goal is to show that, when $\mu_H$ is compactly supported away from the origin, $\mu_G$ is compactly supported and does not have a point mass at the origin, then $\mu_F$ is also compactly supported away from the origin and the solutions $m(z)$, $e(z)$ to the above equations can be extended to the origin.%\\

First, for any $z\in\mathbb C^+$, \cite{couillet2014analysis} showed that the system of equations
$$
\delta=\int\frac{\gamma t}{-z(1+\tilde{\delta}t)}d\mu_H(t),\,\,
\tilde{\delta}=\int\frac{t}{-z(1+\delta t)}d\mu_G(t)
$$
admits a unique solution $(\delta,\tilde{\delta})\in(\mathbb C^+)^2$. Let $\delta(z)$ and $\tilde{\delta}(z)$ be these solutions. Notice that $\delta(z)$ and $e(z)$ have the following relation: $\delta(z)=\gamma e(z)$. Therefore, we can equivalently study $\delta(z)$ instead of $e(z)$. The function $m(z)$, which is the Stieltjes transform of $\mu_F$, can be expressed as:
$$
m(z)=\int\frac{1}{-z(1+\tilde{\delta}(z)t)}d\mu_H(t),~~z\in\mathbb C^+.
$$
We will use this expression later.\\

A important and useful proposition from \cite{couillet2014analysis} is that the functions $\delta(z),\tilde{\delta}(z)$ admit the representations
$$
\delta(z)=\int_0^\infty\frac{1}{t-z}d\rho(t),~~\tilde{\delta}(z)=\int_0^\infty\frac{1}{t-z}d\tilde{\rho}(t),
$$
where $\rho$ and $\tilde{\rho}$ are two Radon positive measures on $\mathbb R^+$ such that
$$
0<\int_0^\infty\frac{1}{1+t}d\rho(t)<\infty,~~0<\int_0^\infty\frac{1}{1+t}d\tilde{\rho}(t)<\infty.
$$
Thus, $\delta(z)$ and $\tilde{\delta}(z)$ can be analytically extended to $\mathbb C\setminus$supp$(\rho)$ and $\mathbb C\setminus$supp$(\tilde{\rho})$ respectively.\\

For the support of measures $\mu_F,\mu_G,\mu_H,\rho$ and $\tilde{\rho}$, we have the following relations from \cite{couillet2014analysis}:
\benum

\item 
$$
\mu_F(\{0\})=1-\min[1-\mu_H(\{0\}), \frac{1-\mu_G(\{0\})}{\gamma}].
$$
So under our assumption that each of $H,G$ have zero point mass at the origin, and that $\gamma<1$, we have $\mu_F(\{0\})=0$.

\item Let $\mathbb R^\ast=\mathbb R\setminus\{0\}$, then supp$(\rho)\cap\mathbb R^\ast=$ supp$(\tilde{\rho})\cap\mathbb R^\ast=$ supp$(\mu_F)\cap\mathbb R^\ast$.
 
\item Suppose $\inf($supp$(\mu_H)\cap\mathbb R^\ast)>0$, i.e. the support of $\mu_H$ is away from the origin, then $\inf($supp$(\mu_F)\cap\mathbb R^\ast)>0$, the support of $\mu_F$ is also away from the origin.

\item supp$(\mu_F)$ is compact if and only if supp$(\mu_G)$ and supp$(\mu_H)$ are both compact.
\item Under our assumption, $\tilde{\rho}(\{0\})=\lim_{y\downarrow 0}(-iy\tilde{\delta}(iy))>0$. Since $\tilde{\delta}(z)\rightarrow \infty$ as $z\rightarrow 0$ and $\mu_H$ is compactly supported away from the origin, by the dominated convergence theorem (DCT),
$$
\rho(\{0\})=\lim_{y\downarrow 0}(-iy\delta(iy))=\lim_{y\downarrow 0}\int\frac{\gamma t}{1+\tilde{\delta}(iy)t}d\mu_H(t)=0.
$$

\eenum

Given these, the picture is now clear. That is, under our assumption, supp$(\mu_F)=$ supp$(\rho)=K$, supp$(\tilde{\rho})=\{0\}\cup K$, where $K$ is some compact set on $\mathbb R^+$ away from the origin. Thus, $m(z)$ and $\delta(z)$ can be analytically extended to $\mathbb C\setminus K$. And $\tilde{\delta}(z)$ can be extended to a meromorphic function on $\mathbb C\setminus K$, with a simple pole at $z=0$.\\

Let us rewrite the system of equations as
$$
\delta(z)=\int\frac{\gamma t}{-z(1+\tilde{\delta}(z)t)}d\mu_H(t),\,\,
-z\tilde{\delta}(z)=\int\frac{t}{1+\delta(z)t}d\mu_G(t),
$$
where $z\in\mathbb C^+$. Now, using the integral representations of $\delta, \tilde \delta$ given above, 
$$
\delta(0)=\int_0^\infty\frac{1}{t}d\rho(t)>0,~~\lim_{z\rightarrow 0}z\tilde{\delta}(z)=-\tilde{\rho}(\{0\})<0,
$$ it is  easy to check that 
the right-hand sides of both equations above are analytic at least in a small neighborhood $U$ of the origin. By the uniqueness property of analytic functions, the above system of equations will hold for all $z\in U$. This means that we can evaluate the equation at $z=0$. For the equation
$$
m(z)=\int\frac{1}{-z(1+\tilde{\delta}(z)t)}d\mu_H(t),~~z\in\mathbb C^+,
$$
we find by a similar argument that we can also evaluate $m(z)$ at $z=0$. This finishes the proof for the expressions of $m(0), e(0)$ given at the beginning of the proof of the main theorem.\\

Moreover, the Stieltjes transform we are looking for has the form $m(0)=e(0) \cdot  \E_H T^{-1}$. Let us write $e(\gamma, G)$ for $e(0)$ showing the dependence on $\gamma,G$ explicitly. Then, 
$$\tr\left[(X^\top X)^{-1}\right]\rightarrow \gamma e(\gamma, G) \cdot  \E_H T^{-1}.$$

Similarly, since $X_i$ has the same elliptical form $X_i = \Gamma_i^{1/2} Z_i \Sigma^{1/2}$, and by assumption the e.s.d. of $\Gamma_i$ converges to $G_i$, we obtain that

$$\tr\left[(X_i^\top X_i)^{-1}\right]\rightarrow \gamma_ie(\gamma_i, G_i) \cdot  \E_H T^{-1}.$$
Thus, the ARE equals
$$ARE =  \gamma e(\gamma, G) \cdot  \E_H T^{-1} \cdot 
\sum_{i=1}^k  \frac{1} {\gamma_i e(\gamma_i, G_i) \cdot  \E_H T^{-1}}
= 
\gamma e(\gamma, G) \cdot 
\sum_{i=1}^k  \frac{1} {\gamma_i e(\gamma_i, G_i)}.$$
Notice now that $f(\gamma, G) = \gamma e(\gamma, G)$ so the ARE also equals
$f(\gamma, G) \cdot 
\sum_{i=1}^k  \frac{1} {f(\gamma_i, G_i)}.$
This finishes the proof. 
%\end{proof}

\subsection{Proof of Theorem \ref{are_ell_props}}
\label{pf:are_ell_props}

%\begin{proof}
%\benum 
%\item 
Under the assumptions of the theorem, we have:
$$
ARE(k)=\frac{kf(\gamma,G)}{f(k\gamma, G)}
=\frac{k\cdot\eta_{G}^{-1}(1-\gamma)}{\eta_{G}^{-1}(1-k\gamma)}
=\frac{e(\gamma, G)}{e(k\gamma, G)}.
$$
The second form given in the theorem follows directly from the definition of $e$.

%\item 
Next, we assume $G$ does not have a point mass at the origin. From the definition of $\eta$-transform, we have the following observation. For any $G$, $\eta_G(x)$ is a smooth decreasing function on $[0,+\infty)$ with $\eta_G(0)=1$ and $\lim_{x\rightarrow+\infty}\eta_G(x)=0$. So $\eta_G^{-1}(x)$ defined on $(0,1]$ is also smooth and decreasing with $\eta_G^{-1}(1)=0$ and $\lim_{x\rightarrow0+}\eta_G^{-1}(x)=+\infty$. This means that $ARE(k)$ is indeed a well-defined function for $k\in[1,1/\gamma]$.

%Furthermore, let us compute the first derivative and the second derivative of $\phi(k)$:
%\begin{align*}
%\phi'(k)&=\frac{\eta_G^{-1}(1-\gamma)}{[\eta_G^{-1}(1-k\gamma)]^2}\cdot\frac{\eta_G^{-1}(1-k\gamma)\cdot\eta'_G(\eta_G^{-1}(1-k\gamma))+k\gamma}{\eta'_G(\eta_G^{-1}(1-k\gamma))}\\
%&=-\frac{\eta_G^{-1}(1-\gamma)}{[\eta_G^{-1}(1-k\gamma)]^2\cdot\eta'_G(\eta_G^{-1}(1-k\gamma))}\cdot\bigg(\int\frac{t\cdot\eta_G^{-1}(1-k\gamma)}{(1+t\cdot\eta_G^{-1}(1-k\gamma))^2}dG(t)-k\gamma\bigg)\\
%&\leq-\frac{\eta_G^{-1}(1-\gamma)}{[\eta_G^{-1}(1-k\gamma)]^2\cdot\eta'_G(\eta_G^{-1}(1-k\gamma))}\cdot\bigg(\int\frac{t\cdot\eta_G^{-1}(1-k\gamma)}{1+t\cdot\eta_G^{-1}(1-k\gamma)}dG(t)-k\gamma\bigg)\\
%&\leq-\frac{\eta_G^{-1}(1-\gamma)}{[\eta_G^{-1}(1-k\gamma)]^2\cdot\eta'_G(\eta_G^{-1}(1-k\gamma))}\cdot(1-(1-k\gamma)-k\gamma)\\
%&=0
%\end{align*}

Next we show that ARE$(k)$ is a decreasing convex function.  Convexity is equivalent to saying that $1/e(k\gamma, G)$ is decreasing and convex in $k$.  Let $\psi(k)=1/e(k\gamma,G)$. Then $\psi(k)$ is the unique positive solution to the equation
$$
\int\frac{t}{\psi(k)+k\gamma t}dG(t)=1.
$$
We can differentiate with respect to $k$ on both sides to get
$$
\psi'(k)=-\frac{\int\frac{\gamma t^2}{(\psi+k\gamma t)^2}dG(t)}{\int\frac{t}{(\psi+k\gamma t)^2}dG(t)}\leq0.
$$
Similarly, we differentiate it twice to get
$$
\psi''(k)=\frac{\int\frac{2t(\psi'+\gamma t)^2}{(\psi+k\gamma t)^3}dG(t)}{\int\frac{t}{(\psi+k\gamma t)^2}dG(t)}\geq0.
$$
This proves that $\psi$ is decreasing and convex, and finishes the proof. 

%\item 
%Furthermore, notice that ARE$(1)=1$, and $\lim_{k\rightarrow1/\gamma}$ARE$(k)=0$. When we choose $G=\delta_c$ to be any point mass at $c>0$, it is  easy to compute that ARE$(k)=(1-k\gamma)/(1-\gamma)$ would be a straight line, which achieves the optimal case. For general $G$, the graph of ARE$(k)$ would be some curve below the straight line.
%\eenum 
%\end{proof}

\subsection{Proof of Proposition \ref{ex1}}
\label{pf:ex1}

Recall that $G_\tau=(1-\gamma)\delta_\tau+\gamma\delta_{1/\gamma}$. Since ARE$(k)$ is always decreasing, we only need to show that $\lim_{\tau\rightarrow0}$ARE$(2)=0$. Now, 
$$
ARE(2)
=\frac{2\cdot\eta_{G}^{-1}(1-\gamma)}{\eta_{G}^{-1}(1-2\gamma)}.
$$
Recall also that $\eta_G(x) = \E_G 1/(1+xT)$, and so for $G_\tau=(1-\gamma)\delta_\tau+\gamma\delta_{1/\gamma}$, 
$$
\eta_{G_\tau}(x) = \frac{1-\gamma}{1+\tau x}+\frac{\gamma}{1+\frac{x}{\gamma}}.
$$
To find $x_1=\eta_{G_\tau}^{-1}(1-\gamma)$, $x_2=\eta_{G_\tau}^{-1}(1-2\gamma)$, it is sufficient to solve the following quadratic equations:
$$
\frac{1-\gamma}{1+\tau x_1}+\frac{\gamma}{1+\frac{x_1}{\gamma}}=1-\gamma,~~\frac{1-\gamma}{1+\tau x_2}+\frac{\gamma}{1+\frac{x_2}{\gamma}}=1-2\gamma,
$$
We can rewrite this as
$$
(1-\gamma)\tau x_1^2+\tau(1-2\gamma)\gamma x_1-\gamma^2=0,
$$
$$
(1-2\gamma)\tau x_2^2+\gamma(\tau-1-3\gamma\tau)x_2-2\gamma^2=0.
$$
Since we are looking for a positive solution, we find that 
\begin{align*}
x_1&=\frac{-\tau\gamma(1-2\gamma)+\sqrt{\tau^2\gamma^2(1-2\gamma)^2+4\tau\gamma^2(1-\gamma)}}{2(1-\gamma)\tau}\\
&=\frac{4\tau\gamma^2(1-\gamma)}{2(1-\gamma)\tau}\cdot\frac{1}{\tau\gamma(1-2\gamma)+\sqrt{\tau^2\gamma^2(1-2\gamma)^2+4\tau\gamma^2(1-\gamma)}}\\
&=\frac{2\gamma^2}{\tau\gamma(1-2\gamma)+\sqrt{\tau^2\gamma^2(1-2\gamma)^2+4\tau\gamma^2(1-\gamma)}}\\
&=O(\tau^{-1/2}),
\end{align*}
$$
x_2=\frac{\gamma(1+3\gamma\tau-\tau)+\sqrt{\gamma^2(1+3\gamma\tau-\tau)^2+8\tau\gamma^2(1-2\gamma)}}{2(1-2\gamma)\tau}\sim\frac{2\gamma}{2(1-2\gamma)\tau}=O(\tau^{-1}).
$$
The order of magnitude calculations follow as $\tau\to 0$. Specifically, for the first case, one can check that the numerator is of order $\tau^{1/2}$ by using the formula for conjugate square roots. For the second case, we only need to notice that the numerator tends to $2\gamma$ as $\tau\to 0$. So ARE$(2)=2x_1/x_2=O(\tau^{1/2})\rightarrow 0$.%\\

\subsection{Proof of Theorem \ref{ex2}}
\label{pf:ex2}

As before, to find $x_1=\eta_{G}^{-1}(1-\gamma)$, $x_k=\eta_{G}^{-1}(1-k\gamma)$, we solve the quadratic equations:
$$
\frac{1-c}{1+\tau x_1}+\frac{c}{1+\alpha\tau x_1}=1-\gamma,~~\frac{1-c}{1+\tau x_k}+\frac{c}{1+\alpha\tau x_k}=1-k\gamma,
$$
and choose the positive solutions:
$$
x_1=\frac{(\gamma-c)\alpha+c+\gamma-1+\sqrt{((\gamma-c)\alpha+c+\gamma-1)^2+4\gamma(1-\gamma)\alpha}}{2(1-\gamma)\alpha\tau},
$$

$$
x_k=\frac{(k\gamma-c)\alpha+c+k\gamma-1+\sqrt{((k\gamma-c)\alpha+c+k\gamma-1)^2+4k\gamma(1-k\gamma)\alpha}}{2(1-k\gamma)\alpha\tau}.
$$
So
$$
ARE(k)=\frac{kx_1}{x_k}=\frac{k(1-k\gamma)}{1-\gamma}\cdot\frac{(\gamma-c)\alpha+c+\gamma-1+\sqrt{((\gamma-c)\alpha+c+\gamma-1)^2+4\gamma(1-\gamma)\alpha}}{(k\gamma-c)\alpha+c+k\gamma-1+\sqrt{((k\gamma-c)\alpha+c+k\gamma-1)^2+4k\gamma(1-k\gamma)\alpha}}
$$
is independent of $\tau$, which intuitively makes sense. Indeed, the problem is scale invariant. We can rescale our data by any constant and the relative efficiency does not change. 

Observe that, when we take $\alpha$ to tend to infinity, the limit will depend on the choice of $c$ and $\gamma$.  There are five sub-cases:

\begin{compactenum}
\item $0<c<\gamma$.\\
\begin{align*}
ARE(2)&=\frac{2x_1}{x_2}=\frac{2(1-2\gamma)}{1-\gamma}\cdot\frac{(\gamma-c)\alpha+c+\gamma-1+\sqrt{((\gamma-c)\alpha+c+\gamma-1)^2+4\gamma(1-\gamma)\alpha}}{(2\gamma-c)\alpha+c+2\gamma-1+\sqrt{((2\gamma-c)\alpha+c+2\gamma-1)^2+8\gamma(1-2\gamma)\alpha}}\\
&=O(1), ~~\alpha\rightarrow+\infty.
\end{align*}
In this case, since the limit of $2x_1/x_2$ is $O(1)$, we should look at the limit of $kx_1/x_k$ for $k>2$, which turns out to be also $O(1)$:
\begin{align*}
&\lim_{\alpha\rightarrow+\infty}ARE(k)=\lim_{\alpha\rightarrow+\infty}\frac{kx_1}{x_k}\\
&=\lim_{\alpha\rightarrow+\infty}\frac{k(1-k\gamma)}{1-\gamma}\cdot\frac{(\gamma-c)\alpha+c+\gamma-1+\sqrt{((\gamma-c)\alpha+c+\gamma-1)^2+4\gamma(1-\gamma)\alpha}}{(k\gamma-c)\alpha+c+k\gamma-1+\sqrt{((k\gamma-c)\alpha+c+k\gamma-1)^2+4k\gamma(1-k\gamma)\alpha}}\\
&=\frac{k(\gamma-c)(1-k\gamma)}{(1-\gamma)(k\gamma-c)}.
\end{align*}

\item $c=\gamma$.\\

The previous example is a sub-case of this case, where the ratio between the large and small variances equals $\alpha: = 1/(\gamma \tau) \to \infty$.
$$
ARE(2)=\frac{2x_1}{x_2}=\frac{2(1-2\gamma)}{1-\gamma}\cdot\frac{2\gamma-1+\sqrt{(2\gamma-1)^2+4\gamma(1-\gamma)\alpha}}{\gamma\alpha+3\gamma-1+\sqrt{(\gamma\alpha+3\gamma-1)^2+8\gamma(1-2\gamma)\alpha}}=O(\alpha^{-1/2}), ~~\alpha\rightarrow+\infty.
$$

\item $\gamma<c<2\gamma$.\\
\begin{align*}
ARE(2)&=\frac{2x_1}{x_2}=\frac{2(1-2\gamma)}{1-\gamma}\cdot\frac{(\gamma-c)\alpha+c+\gamma-1+\sqrt{((\gamma-c)\alpha+c+\gamma-1)^2+4\gamma(1-\gamma)\alpha}}{(2\gamma-c)\alpha+c+2\gamma-1+\sqrt{((2\gamma-c)\alpha+c+2\gamma-1)^2+8\gamma(1-2\gamma)\alpha}}\\
&=O(\alpha^{-1}), ~~\alpha\rightarrow+\infty.
\end{align*}
%In this case, the effective sample is reduced to $cn>p$ for the global problem. So the global problem is still well-conditioned. However, at least one of the local problems has sample size $n_i$ such that the effective sample size $cn_i<p$, because $c(n_1+n_2) = cn <2p$. Therefore, OLS is ill-defined on at least one of the local problems. Hence, the ARE should be very small, which is indeed the case.

\item $c=2\gamma$.\\
\begin{align*}
ARE(2)&=\frac{2x_1}{x_2}=\frac{2(1-2\gamma)}{1-\gamma}\cdot\frac{-\gamma\alpha+3\gamma-1+\sqrt{(-\gamma\alpha+3\gamma-1)^2+4\gamma(1-\gamma)\alpha}}{4\gamma-1+\sqrt{(4\gamma-1)^2+8\gamma(1-2\gamma)\alpha}}\\
&=O(\alpha^{-1/2}), ~~\alpha\rightarrow+\infty.
\end{align*}

\item $c>2\gamma$.\\

Here, we can easily find $x_1/x_2=O(1)$ as $\alpha\rightarrow+\infty$, but now the exact value of $c$ matters. That is, suppose $c=M\gamma$ for some $M>2$. Then we will find that
$$
\lim_{\alpha\rightarrow+\infty}ARE(k)=\lim_{\alpha\rightarrow+\infty}\frac{kx_1}{x_k}=
\begin{cases}
\frac{c-k\gamma}{c-\gamma},~~k<M,\\
O(\alpha^{-1/2}),~~k=M,\\
O(\alpha^{-1}),~~k>M.
\end{cases}
$$

%As before, this result can be understood in terms of reducing the effective sample size to $cn$. When $k<p/(cn)$, i.e., $k<M$, each local problem can be well-conditioned. However, when $k>p/(cn)$, i.e., $k>M$, at least one of the local OLS problems is ill-conditioned, so the ARE is small.

%\ed{Would it make sense to summarize the results for all M,k?}

\end{compactenum}

% \section{Verifying the special case of the general MP theorem}
% \label{pf:spec}
% When $\Gamma = I_n$, the second equation in the system simplifies to
% \begin{align*}
% x_p &= \frac{1}{1+\gamma_p e_p}.
% \end{align*}
% Plugging this into the first equation, we find
% \begin{align*}
% \frac{1}{x_p} -1 = \frac{1}{n} \tr\left[\Sigma (x_p\Sigma- z I )^{-1}\right].
% \end{align*}

% Or $x_p$ is the unique solution of: 

% $$1 -x_p = \gamma_p \left[1 + \frac{z}{p}\tr\left (x_p\Sigma-z I\right)^{-1}\right].$$

\subsection{Proof of Theorem \ref{fe_ell}}
\label{pf:fe_ell}
In the proof of Theorem \ref{IE_mp_thm}, we derive the limit for 
$$
\frac{\tr((X_i^\top X_i)^{-1}X^\top X)}{p}=\frac{p+\sum_{j\neq i}\tr((X_i^\top X_i)^{-1}X_j^\top X_j)}{p},
$$
which is 
$$
1+(\frac{1}{\gamma}\E_GT-\frac{1}{\gamma_i}\E_{G_i}T)f(\gamma_i, G_i).
$$
Then the desired result follows. It is not hard to check the results in the MP case.

\subsection{Proof of Theorem \ref{IE_mp_thm}}
\label{pf:IE_mp_thm}

We consider the Elliptical type sample covariance matrices first. Recall that we have $X=\Gamma^{1/2}Z\Sigma^{1/2}$, where $Z$ is an $n\times p$ matrix with standardized entries, $\Gamma$ is an $n\times n$ diagonal matrix with positive entries and $\Sigma$ is a $p\times p$ nonnegative-definite matrix. Our goal is to understand the limit of 
$$
\tr[(X_i^\top X_i)^{-1}X^\top X]=\tr[(X_i^\top X_i)^{-1}X_i^\top X_i]+\sum_{j\neq i}\tr[(X_i^\top X_i)^{-1}X_j^\top X_j]=p+\sum_{j\neq i}\tr[(X_i^\top X_i)^{-1}X_j^\top X_j].
$$
If we delete all rows of $X_i$ from $X$ and denote the remaining matrix by $\tilde{X}_i$, then this can be written as
$$p+\tr[\tilde{X}_i(X_i^\top X_i)^{-1}\tilde{X}_i^\top ].$$
Since $X_i=\Gamma_i^{1/2}Z_i\Sigma^{1/2}, \tilde{X}_i=\tilde{\Gamma}_i^{1/2}\tilde{Z}_i\Sigma^{1/2}$, where the $n_i\times p$ matrix $Z_i$ and the $(n-n_i)\times p$ matrix $\tilde{Z}_i$ both have i.i.d. standardized entries, the $(n-n_i)\times(n-n_i)$ diagonal matrix $\tilde{\Gamma}_i$ is the remaining matrix after deleting all the entries of $\Gamma_i$ from $\Gamma$. Then we find that the population covariance $\Sigma$ will cancel out:
\begin{align*}
&\tr[\tilde{X}_i(X_i^\top X_i)^{-1}\tilde{X}_i^\top ]
=\tr[\tilde{\Gamma}_i^{1/2}\tilde{Z}_i\Sigma^{1/2}(\Sigma^{1/2}Z_i^\top\Gamma_iZ_i\Sigma^{1/2})^{-1}\Sigma^{1/2}\tilde{Z}_i^\top\tilde{\Gamma}_i^{1/2} ]\\
&=\tr[\tilde{\Gamma}_i^{1/2}\tilde{Z}_i(Z_i^\top\Gamma_i Z_i)^{-1}\tilde{Z}_i^\top\tilde{\Gamma}_i^{1/2}].
\end{align*}
To evaluate the limit, we will use the following lemma from \cite{rubio2011spectral}.
\begin{lemma}[Concentration of average of quadratic forms, Lemma 4 in \cite{rubio2011spectral}] Let $\mathcal{U}=\{\xi_k\in\mathbb C^M, 1\leq k\leq N\}$ denote a collection of i.i.d. random vectors with i.i.d entries that have mean $0$, variance $1$ and finite $4+\delta$ moment, $\delta>0$. Furthermore, consider a collection of random matrices $\{\mathbf{C}_{(k)}\in\mathbb C^{M\times M}, 1\leq k\leq N\}$ such that, for each $k$, $\mathbf{C}_{(k)}$ may depend on all the elements of $\mathcal{U}$ except for $\xi_k$, and the trace norm of $\mathbf{C}_{(k)}$, $||\mathbf{C}_{(k)}||_{\tr}$ is almost surely uniformly bounded for all $M$. Then, almost surely as $N\to\infty$,
$$\left|\frac{1}{N}\sum_{k=1}^N\left(\xi_k^H\mathbf{C}_{(k)}\xi_k-\tr\mathbf{C}_{(k)}\right)\right|\rightarrow 0.$$
\label{average_quad_form}
\end{lemma}
\noindent
For our purpose, we can take the number of summands to be $N=n-n_i,$ the dimension $M=p,$ and the inner matrices to be $\mathbf{C}_{(k)}=\frac{n_i}{p}(Z_i^\top\Gamma_i Z_i)^{-1}\cdot(\tilde{\Gamma}_i)_{kk}$. Also, we let $\xi_k^\top$ to be the $k$-th row of $\tilde{Z}_i$. By using the well-known result on spectrum separation\citep[see e.g.,][]{bai2009spectral}, almost surely, the smallest eigenvalue of ${n_i}^{-1}Z_i^\top Z_i$ is uniformly bounded below by some constant. Since $\lambda_{\min}({n_i}^{-1}Z_i^\top\Gamma_i Z_i)\geq \lambda_{\min}(\Gamma_i)\cdot\lambda_{\min}({n_i}^{-1}Z_i^\top Z_i)$, $\lambda_{\min}({n_i}^{-1}Z_i^\top\Gamma_i Z_i)$ is also uniformly bounded below almost surely. So under the assumption of Theorem \ref{IE_mp_thm}, we can check that the trace norm of $\frac{n_i}{p}(Z_i^\top\Gamma_i Z_i)^{-1}\cdot(\tilde{\Gamma}_i)_{kk}$ is indeed uniformly bounded. Then by Lemma \ref{average_quad_form}, we will have as $n\to\infty$
$$
\left|\frac{1}{n-n_i}\sum_{k=1}^{n-n_i}\left[\frac{n_i}{p}(\tilde{\Gamma}_i)_{kk}\cdot\xi_k^\top(Z_i^\top\Gamma_i Z_i)^{-1}\xi_k-\tr\left(\frac{n_i}{p}(Z_i^\top\Gamma_i Z_i)^{-1}\cdot(\tilde{\Gamma}_i)_{kk}\right)\right]\right|\to_{a.s.}0.
$$
This implies
$$
\frac{1}{n-n_i}\tr[\tilde{\Gamma}_i^{1/2}\tilde{Z}_i(Z_i^\top\Gamma_i Z_i)^{-1}\tilde{Z}_i^\top\tilde{\Gamma}_i^{1/2}]\to_{a.s.}f(\gamma_i, G_i)\left(\frac{\gamma_i}{\gamma_i-\gamma}\E_GT-\frac{\gamma}{\gamma_i-\gamma}\E_{G_i}T\right),
$$
since $\tr(Z_i^\top\Gamma_i Z_i)^{-1}\to_{a.s.} f(\gamma_i, G_i)$ and 
$$
\frac{\sum_k(\tilde{\Gamma}_i)_{kk}}{n-n_i}=\frac{n}{n-n_i}\cdot\frac{\tr(\Gamma)}{n}-\frac{n_i}{n-n_i}\cdot\frac{\tr(\Gamma_i)}{n_i}\to_{a.s.}\left(\frac{\gamma_i}{\gamma_i-\gamma}\E_GT-\frac{\gamma}{\gamma_i-\gamma}\E_{G_i}T\right).
$$
This holds for all $i$. Thus, for the elliptical model, we have
\begin{align*}
IE(X_1,\ldots,X_k)=\frac{1-\frac{p}{n}}{1-\frac{2p}{n}+\frac{1}{\sum_{i=1}^k\frac{n}{\tr[(X_i^\top X_i)^{-1}X^\top X]}}}
\to_{a.s.}
\frac{1-\gamma}{1-2\gamma+\frac{1}{\sum_{i=1}^k\psi(\gamma_i, G_i)}}
,
\end{align*}
where 
$$
\psi(\gamma_i, G_i)=\frac{1}{\gamma+(\E_GT-\frac{\gamma}{\gamma_i}\E_{G_i}T)f(\gamma_i,G_i)}.
$$
Now, for the MP model, we can simply take $\Gamma$ to be identity matrix, then the above result reduces to 
$$
IE(X_1,\ldots,X_k)=\frac{1-\frac{p}{n}}{1-\frac{2p}{n}+\frac{1}{\sum_{i=1}^k\frac{n}{\tr[(X_i^\top X_i)^{-1}X^\top X]}}}
\to_{a.s.}
\frac{1-\gamma}{1-2\gamma+\frac{\gamma(1-\gamma)}{1-k\gamma}},
$$
which finishes the proof.

\subsection{Proof of Theorem \ref{OE_mp_thm}}
\label{pf:OE_mp_thm}

We first provide the proof for the MP model. 
Since $\Sigma$ is positive definite, we have
$$x_t^\top (X_i^\top X_i)^{-1} x_t
= z_t^\top \Sigma^{1/2} (\Sigma^{1/2}Z_i^\top Z_i\Sigma^{1/2})^{-1} \Sigma^{1/2}z_t
= z_t^\top (Z_i^\top Z_i)^{-1} z_t.$$
This cancellation shows that the test error does not depend on the covariance matrix. 

%\item 
For the null case, we will show below that we have, almost surely
$$ z_t^\top (Z_i^\top Z_i)^{-1} z_t \to  \frac{ \gamma_i}{1-\gamma_i}.$$ 

 Hence, we obtain that 
\begin{align*}
OE(x_t; X_1,\ldots,X_k)
\to_{a.s.}
\frac{1+ \frac{ \gamma}{1-\gamma}}
{1 + \frac{1}{\sum_{i=1}^k \frac{1-\gamma_i}{\gamma_i}}}
= 
\frac{\frac{1}{1-\gamma}}
{1 + \frac{1}{\frac{1}{\gamma}- k}}
.
\end{align*}
% The expression for the OE can be rewritten and partially simplified into 
% $$
% \frac{\frac{1}{1-\gamma}}
% {1 + \frac{\gamma}{1- k\gamma}}
% = 
% \frac{1}
% {1 + \frac{(k-1)\gamma^2}{1- k\gamma}}.
% $$

Under the elliptical model, we have
$x_t^\top (X_i^\top  X_i)^{-1} x_t
= g_t  z_t^\top (Z_i^\top \Gamma_i Z_i)^{-1} z_t.$
While $\Sigma$ still cancels out, the scale parameters do not cancel out anymore. Therefore, we must take them into account when taking the limits. However, similarly to the proof of Theorem \ref{fe_ell}, we find that, almost surely
$$z_t^\top (Z_i^\top \Gamma_i Z_i)^{-1} z_t  \to f(\gamma_i, G_i).$$
Putting these together finishes the proof. \\

To see the reason for the convergence of quadratic forms, we present a slightly different argument. In fact, Theorem \ref{OE_mp_thm} will still hold if we are only given the $4+c$-th moment of $z_1$ instead of $8+c$-th one. This follows by the concentration of quadratic forms $x^\top A x -p^{-1}\tr A \to 0$ for matrices $A$ whose spectral distribution converges, and for vectors $x$ with iid entries. Specifically, we will use the following well-known statement about concentration of quadratic forms. To use this result, we simply choose $x = z_t/\sqrt{p}$, and $A_p = (Z_i^\top \Gamma_i Z_i/p)^{-1}$, and the desired claim follows.

\begin{lemma}[Concentration of quadratic forms, consequence of Lemma B.26 in \citet{bai2009spectral}] Let $x \in \RR^p$ be a random vector with i.i.d. entries and $\EE{x} = 0$, for which $\EE{(\sqrt{p}x_i)^2} = \sigma^2$ and $\sup_i \EE{(\sqrt{p}x_i)^{4+\eta}}$ $ < C$ for some $\eta>0$ and $C <\infty$. Moreover, let $A_p$ be a sequence of random $p \times p$ symmetric matrices independent of $x$, with uniformly bounded eigenvalues. Then the quadratic forms $x^\top A_p x $ concentrate around their means at the following rate

$$P(|x^\top A_p x - p^{-1} \sigma^2 \tr A_p|^{2+\eta/2}>C) \le C p^{-(1+\eta/4)}.$$
\label{quad_form}
\end{lemma}

To prove this, we will use the following Trace Lemma quoted from \cite{bai2009spectral}, see also \cite{dobriban2017optimal} for a similar argument. 

\begin{lemma}[[Trace Lemma, Lemma B.26 of \cite{bai2009spectral}] Let $y$ be a p-dimensional random vector of i.i.d. elements with mean 0. Suppose that $\EE{y_i^2} = 1$, and let $A_p$ be a fixed $p \times p$ matrix. Then  
$$\EE{|y^\top A_p y-\tr A_p|^q} \le C_q \left\{\left(\EE{y_1^4}\tr[A_pA_p^\top]\right)^{q/2}+\EE{y_1^{2q}}\tr[(A_pA_p^\top)^{q/2}]\right\},$$ 
for some constant $C_q$ that only depends on $q$. 
\label{trace_lemma}
\end{lemma}

\begin{proof}
Under the conditions of Lemma \ref{quad_form}, the operator norms $\|A_p\|_2$ are almost surely uniformly bounded by a constant $C$, thus $\tr[(A_pA_p^\top)^{q/2}] \le p C^q$ and $\tr[A_pA_p^\top] \le p C^2$. Consider now a random vector $x$ with the properties assumed in the present lemma. For $y = \sqrt{p}x/\sigma$ and $q = 2+\eta/2$, using that $\EE{y_i^{2q}}\le C$ and the other the conditions in Lemma \ref{quad_form}, Lemma \ref{trace_lemma} thus yields
$$\frac{p^q}{\sigma^{2q}}\EE{|x^\top A_p x-\frac{\sigma^2}{p}\tr A_p|^q} \le C \left\{\left( p C^2\right)^{q/2}+pC^q\right\},$$ 
or equivalently $\EE{|x^\top A_p x-\frac{\sigma^2}{p}\tr A_p|^{2+\eta/2}} \le C p^{-(1+\eta/4)}$. 

By Markov's inequality applied to the $2+\frac{\eta}{2}$-th moment of $\varepsilon_p = x^\top A_p x-\frac{\sigma^2}{p}\tr A_p$, we obtain as required 
$$P(|\ep_p|^{2+\eta/2}>C) \le C p^{-(1+\eta/4)}.$$
\end{proof}

\subsection{Proof of Theorem \ref{ce_ell}}
\label{pf:ce_ell}

From Theorem \ref{gen_det}, it follows that the inverse sample covariance matrix $\hSigma$ is equivalent to a scaled version of the population covariance
$$\hSigma^{-1} \asymp \Sigma^{-1} \cdot e_p, $$
for some scalar sequence $e_p>0$. By taking in Theorem \ref{gen_det} the matrix $C_p = E_j E_j^\top$, the $p\times p$ matrix with a 1 in the $(j,j)$-th entry, and zeros otherwise, we find that almost surely,

$$[\hSigma^{-1}]_{jj} - [\Sigma^{-1}]_{jj} \cdot e_p \to 0, $$

We can apply this to each sub-matrix $X_i$ to find

$$n_i \cdot [(X_i^\top X_i)^{-1}]_{jj} - [\Sigma^{-1}]_{jj} \cdot e_p(i) \to 0.$$

Here $e_p(i)$ is the solution to the fixed point equation 
\begin{align*}
1 &= \frac{1}{n_i} \tr\left[e_p(i) \Gamma_i (I_{n_i}+\gamma_{p,i} \cdot e_p(i) \Gamma_i)^{-1}\right].
\end{align*}

Moreover, $\gamma_{p,i} = p/n_i$ and $\Gamma_i$ is the $n_i \times n_i$ sub-matrix of $\Gamma$ corresponding to the $i$-th machine. It follows that the CE has a deterministic equivalent equal to
\begin{align*}
&\frac{[\Sigma^{-1}]_{jj} \cdot e_p}{n}
\cdot 
\sum_{i=1}^k \frac{n_i}{[\Sigma^{-1}]_{jj} \cdot e_p(i)} = \\
=&\frac{p \cdot e_p}{n}
\cdot 
\sum_{i=1}^k \frac{n_i}{p \cdot e_p(i)} \to
\gamma \cdot e(\gamma,G)
\cdot 
\sum_{i=1}^k \frac{1}{\gamma_i\cdot e(\gamma_i,G_i)}
.
\end{align*}
Here $e(\gamma_i,G_i)$ are the quantities encountered before, discussed after Theorem \ref{gen_det}. The convergence follows from the discussion after Theorem \ref{gen_det}. Also, from the definition of $f(\gamma, G) $ it follows that $f(\gamma, G) = \gamma e(\gamma, G)$, so that we get the desired result. This finishes the proof. 

\section{Special case of parameter estimation}
\label{par-est}
\subsection{Suboptimal weights}
\label{subopt}

If we take all weights $w_i$ to be equal, i.e. $w_i=1/k$, then the MSE is
$$
MSE_{subopt}=\frac{\sigma^2}{k^2}\sum_{i=1}^k\tr(X_i^\top X_i)^{-1}\to \frac{\sigma^2}{k^2}\sum_{i=1}^k \frac{\gamma_i \cdot \E_{H} T^{-1}}{1-\gamma_i}.$$

Thus the ARE of the equally weighted estimator becomes (with the notation AE denoting asymptotic MSE)
$$
ARE_{subopt}=\frac{AE(\hbeta_{dist}(1/k,\ldots,1/k))}{AE_{subopt}}
=\frac{k^2\frac{\gamma}{1-\gamma}}{\sum_{i=1}^k\frac{\gamma_i}{1-\gamma_i}}.$$
Now, $ARE_{subopt}$ can be viewed as a harmonic mean of the numbers 
$$k\frac{\gamma}{1-\gamma}\frac{1-\gamma_i}{\gamma_i},$$
while the optimal ARE is the corresponding arithmetic mean. Therefore, we have  $ARE_{subopt}\le ARE.$

\subsection{Properties and interpretation of the relative efficiency.}
\label{prop-int}

 Let $f(n,p,k)$ be the relative efficiency for estimation, $(n-kp)/(n-p)$. If $kp>n$, that expression is negative, but in that case it is more proper to define the relative efficiency as 0. So, we consider
$$f(n,p,k) = \max\left(\frac{n-kp}{n-p},0\right).$$

This has the following properties. Each of these has a statistical interpretation. 

\begin{compactenum}
\item {\bf Well-definedness}. $f$ is well-defined for all $n,p,k$ such that $n>p$% and $n \ge kp$
\item {\bf Range}. $0\le f\le 1$ for all $n,p,k$. Clearly the efficiency should be between zero and unity.

Also, $f$ is zero for $k \ge n/p$. In this case, some machine has an OLS estimator that is not well-defined. 

Moreover, $f = 1$ when $k=1$ or when $p=0$. When we have one machine, the efficiency is unity by definition. When $p=0$, the problem is not well-defined, as there are no parameters to estimate.

\item {\bf Monotonicity}.

\begin{compactenum}
\item {\bf $f$ is monotone decreasing in $k$.} This property is easy to interpret. The distributed regression problem gets harder as $k$ increases.

\item {\bf $f$ is monotone increasing in $n$.} The linear regression problem should get easier as $n$ grows. However, it turns out that more is true. The distributed problem gets relatively easier compared to the "shared" problem. %When $n\to \infty$, the two prb

\item {\bf $f$ is monotone decreasing in $p$.}  Similarly, a typical linear regression problem should get harder as $p$ grows. However, the relative difficulty of solving the distributed problem also gets larger. 
\end{compactenum}

\item {\bf Limits and singularity}. 

\begin{compactenum}
\item {\bf $n\to \infty$}. When $n \to \infty$ with fixed $k,p$, then $f$ tends to unity. When $n\to \infty$, the distributed estimator becomes asymptotically efficient. 

\item {\bf $p=n$}. The function is singular when $p =n$, because the OLS estimator itself is singular when $p=n$.
\end{compactenum}
\end{compactenum}

Note that these properties are not enough to characterize the relative efficiency. In fact, for any monotone increasing transform such that $g(0) = 0$ and $g(1) = 1$, $g(f(n,p,k))$ has the same properties.

% \section{Minimax optimality}
% \label{mmx}
% Our results show that the distributed regression estimator is minimax rate-optimal as long as the number of machines is not too large. Indeed, it is well known that the minimax estimation error in linear regression is $\sigma^2 \tr[(X^\top X)^{-1}]$ \citep{lehmann1998theory}. Asymptotically, our results show that the estimation error of OLS is less than that of one-step averaging by a factor ARE given in \eqref{are1}. Thus, as long as $ARE>c>0$ for some universal constant $c>0$, we can say that one-step averaging in distributed linear regression is asymptotically minimax rate-optimal.

% Indeed, in finite samples minimax rate optimality of a sequence of estimators $\htheta_n$ with respect to the risk functions $R_n$ is defined as $R_n(\htheta_n) \le C R_n^*$, where $R_n^*$ is the minimax risk with $n$ samples, and where $C$ is any universal constant independent of $n$. 
% %This is equivalent to $R_n^*/ R_n(\htheta_n)\ge c >0$ for $c =1/C>0$. 
% Asymptotically, this is equivalent to $\lim\inf_n R_n^*/ R_n(\htheta_n) >0$. For our problem, it can be checked that this holds precisely if %$\lim\inf(n-kp)/(n-p)>0$, or also if 
% $$\limsup_{n,p \to \infty ,\, k \in \mathbb{N}} \frac{kp}{n} <1.$$
% This gives a precise condition under which one-step averaging is rate-optimal. However, we note that our results are much stronger than that, because we find the \emph{exact limit} of the risk, and not just up to unspecified constants.
% %\ed{Minimax: finite sample first, then rescale and take limit. }
% %\subsection{Suboptimal weights}
 
\subsection{Degrees of freedom interpretation}
\label{dof}

Next we give a multi-response regression characterization that heuristically gives an upper bound on the "\emph{degrees of freedom}" for distributed regression. This will be helpful to understand the asymptotic formulas derived above.

We re-parametrize $Y = X\beta + \ep$, treating the samples on each machine as a different outcome. We write the $n \times k$ multi-response outcome matrix $\underline Y$, the $n \times pk$ feature matrix $\underline X$, and the corresponding noise $\underline \ep$ as
$$\underline Y = 
\begin{bmatrix}
    Y_{1} & 0 & \dots  & 0 \\
    0 & Y_{2} & \dots  & 0 \\
    \vdots & \vdots & \ddots & \vdots \\
    0 & 0 & \dots  & Y_{k}
\end{bmatrix}
,\,\,\underline X = 
\begin{bmatrix}
    X_{1} & 0 & \dots  & 0 \\
    0 & X_{2} & \dots  & 0 \\
    \vdots & \vdots & \ddots & \vdots \\
    0 & 0 & \dots  & X_{k}
\end{bmatrix}
,\,\,\underline \ep = 
\begin{bmatrix}
    \ep_{1} & 0 & \dots  & 0 \\
    0 & \ep_{2} & \dots  & 0 \\
    \vdots & \vdots & \ddots & \vdots \\
    0 & 0 & \dots  & \ep_{k}
\end{bmatrix}.
$$

We also introduce $\underline \beta$, the $pk \times k$ parameter matrix, which shares parameters across the $k$ outcomes: 
$$\underline \beta = 
\begin{bmatrix}
    \beta & 0 & \dots  & 0 \\
    0 & \beta & \dots  & 0 \\
    \vdots & \vdots & \ddots & \vdots \\
    0 & 0 & \dots  & \beta
\end{bmatrix}
=  I_k \otimes \beta
$$

Note that $Y = X\beta + \ep$ is equivalent to $\underline Y = \underline X\underline \beta + \underline \ep$. The OLS estimator of $\underline \beta$ is  $\hat{\underline \beta}  = (\underline X^\top \underline X)^{-1} \underline X^\top \underline Y$. This can be calculated as

$$\hat{\underline \beta} = 
\begin{bmatrix}
    \hbeta_1 & 0 & \dots  & 0 \\
    0 & \hbeta_2 & \dots  & 0 \\
    \vdots & \vdots & \ddots & \vdots \\
    0 & 0 & \dots  & \hbeta_k
\end{bmatrix}
= 
\begin{bmatrix}
    (X_1^\top X_1)^{-1}X_1^\top Y_1 & 0 & \dots  & 0 \\
    0 & (X_2^\top X_2)^{-1}X_2^\top Y_2 & \dots  & 0 \\
    \vdots & \vdots & \ddots & \vdots \\
    0 & 0 & \dots  & (X_k^\top X_k)^{-1}X_k^\top Y_k
\end{bmatrix}
$$
Notice that the estimators of the coefficients of different outcomes are the familiar distributed OLS estimators.
Now, we can find a plug-in estimator of $\beta$, based on $\underline \beta$. Given the form of $\underline \beta$ above, for any vector $w$ such that $\sum_{i=1}^k w_i = 1$, we have that $\beta$ can be expressed in terms of the tensorized parameter $\underline \beta$ as a weighted combination
$$\beta = (1_p^\top\otimes I_k) \underline  \beta w.
$$
Therefore, for any unbiased estimator of $\underline \beta$, the corresponding weighted combination estimators given below are unbiased for $\beta$: 
$$\hbeta(w) = (1_p^\top\otimes I_k) \hat{\underline \beta} w.
$$
In our case, given the zeros in the estimator, this simply reduces to the weighted sum $\hbeta(w) = \sum_{i=1}^k w_i \hbeta_i$.

This explains how our problem can be understood in the framework of multi-response regression. Also, the number of parameters in that problem is $kp$, so the "degrees of freedom" is $n - kp$. 
Indeed, the residual effective degrees of freedom $\hat{r} =y-Hy$ is usually defined as $\tr (I-H)$. Let $H_i$ be the hat matrix on the $i$-th machine, so that $H_i = X_i (X_i^\top X_i)^{-1} X_i^\top$. Then it is easy to see that $\tr (I-H_i) = n_i - p$, for all $i$. Since $H_{dist}$ is simply the block diagonal matrix with $H_i$ as blocks, we see that $\tr (I-H_{dist}) = n-pk$, as required. 

This provides a simple explanation for why the "effective number of parameters" in a one-step distributed linear regression problem is upper bounded by $kp$. Equivalently, the residual "degrees of freedom" of a one-step distributed estimation problem is lower bounded by $n - kp$. Note that with more rounds of communication, we could drive the degrees of freedom up, and hence this is just a heuristic bound. %, and also for why the relative efficiency is approximately $(n-kp)/(n-p)$. 

\section{Proof of Theorem \ref{noisy}}
\label{pf:noisy}

First, when $\rho = 0$, the MSE reduces to $\sigma^2/k^2\cdot\sum_{i=1}^k\tr(X_i^\top X_i)^{-1}$ which is exactly the MSE of the one-step naive average
estimator $1/k\cdot\sum_{i=1}^k(X_i^\top X_i)^{-1}X_i^\top Y_i$. Second, when $\rho\to\infty$, 
\begin{align*}
\hbeta_{*}&=\left(\sum_{i=1}^k\left(X_i^\top X_i+n_i\rho I\right)^{-1}X_i^\top X_i\right)^{-1}\cdot\sum_{i=1}^k\left(X_i^\top X_i+n_i\rho I\right)^{-1}X_i^\top Y_i\\
&=\left(\sum_{i=1}^k\left(\frac{X_i^\top X_i}{\rho}+n_i I\right)^{-1}X_i^\top X_i\right)^{-1}\cdot\sum_{i=1}^k\left(\frac{X_i^\top X_i}{\rho}+n_i I\right)^{-1}X_i^\top Y_i\\
&\to\left(\sum_{i=1}^k\frac{X_i^\top X_i}{n_i}\right)^{-1}\cdot\sum_{i=1}^k\frac{X_i^\top Y_i}{n_i}=(X^\top X)^{-1}X^\top Y,
\end{align*}
which is the OLS estimator for the whole data set and the MSE also converges to the corresponding MSE $\sigma^2\cdot\tr(X^\top X)^{-1}$. Actually, we can define the MSE as a function of $\rho$  in the following way
\begin{align*}
\psi(\rho)&=\sum_{i=1}^k
\tr\left[\left(\sum_{i=1}^k\left(X_i^\top X_i+n_i\rho I\right)^{-1}X_i^\top X_i\right)^{-2}\left(X_i^\top X_i+n_i\rho I\right)^{-2}X_i^\top X_i\right]\\
&=\sum_{i=1}^k
\tr\left[\left(\sum_{i=1}^k\left(\hSigma_i+\rho I\right)^{-1}\hSigma_i\right)^{-2}\left(\hSigma_i+\rho I\right)^{-2}\cdot\frac{\hSigma_i}{n_i}\right],
\end{align*}
where $\hSigma_i=X_i^\top X_i/n_i$.
For any fixed $\rho\in[0, \infty)$, we can consider a small perturbation $\ep$ at $\rho$, i.e. we can consider the difference between $\psi(\rho+\ep)$ and $\psi(\rho)$. By using the formula
$$\left(I+\ep A\right)^{-1}=\sum_{n=0}^\infty (-1)^nA^n\ep^n=I-\ep A+o(\ep),$$
and we define the following quantity
$$\Delta:=\sum_{i=1}^k\left(\hSigma_i+\rho I\right)^{-1}\hSigma_i.$$
We have the following expansions:
$$\left(\sum_{i=1}^k\left(\hSigma_i+\rho I+\ep I\right)^{-1}\hSigma_i\right)^{-1}=\Delta^{-1}+\ep\Delta^{-1}\left(\sum_{i=1}^k\left(\hSigma_i+\rho I\right)^{-2}\hSigma_i\right)\Delta^{-1}+o(\ep),$$
$$\left(\hSigma_i+\rho I+\ep I\right)^{-2}\hSigma_i=\left(\hSigma_i+\rho I\right)^{-2}\hSigma_i-2\ep\left(\hSigma_i+\rho I\right)^{-3}\hSigma_i+o(\ep).$$
By putting these together, we have
\begin{align*}
&\psi(\rho+\ep)-\psi(\rho)=\\
&\ep\cdot\frac{2k}{n}\tr\left[\Delta^{-1}\sum_{i=1}^k\left(\hSigma_i+\rho I\right)^{-2}\hSigma_i\cdot\Delta^{-2}\sum_{i=1}^k\left(\hSigma_i+\rho I\right)^{-2}\hSigma_i-\Delta^{-2}\sum_{i=1}^k\left(\hSigma_i+\rho I\right)^{-3}\hSigma_i\right]\\
&+o(\ep),
\end{align*}
which means $\psi(\rho)$ is differentiable and the derivative at $\rho$ is 
$$\psi'(\rho)=\frac{2k}{n}\tr\left[\Delta^{-1}\sum_{i=1}^k\left(\hSigma_i+\rho I\right)^{-2}\hSigma_i\cdot\Delta^{-2}\sum_{i=1}^k\left(\hSigma_i+\rho I\right)^{-2}\hSigma_i-\Delta^{-2}\sum_{i=1}^k\left(\hSigma_i+\rho I\right)^{-3}\hSigma_i\right].$$
Actually, we can show that $\psi'(\rho) \leq 0$ holds for all $\rho\in[0, +\infty)$. In order to do that, we need to introduce the so-called Schur complement and its applications on positive semi-definite matrices.
\begin{lemma}[Schur complement condition for positive semi-definiteness]
For any symmetric matrix $M$ of the form
$$M=
\begin{bmatrix}
A         & B\\
B^\top & C
\end{bmatrix},
$$
if $C$ is invertible then the following properties hold:
\begin{enumerate}
\item $M\succ 0$ iff $C\succ 0$ and $A-BC^{-1}B^\top\succ 0$.
\item If $C\succ 0$, then $M\succeq 0$ iff $A-BC^{-1}B^\top\succeq 0$.
\end{enumerate}
\label{schur}
\end{lemma}
In order to show $\psi'(\rho)\leq 0$, it is sufficient to show
\begin{align*}
&\tr\left[\Delta^{-1}\sum_{i=1}^k\left(\hSigma_i+\rho I\right)^{-2}\hSigma_i\cdot\Delta^{-2}\sum_{i=1}^k\left(\hSigma_i+\rho I\right)^{-2}\hSigma_i-\Delta^{-2}\sum_{i=1}^k\left(\hSigma_i+\rho I\right)^{-3}\hSigma_i\right]\\
&=\tr\left[\Delta^{-2}\left(\sum_{i=1}^k\left(\hSigma_i+\rho I\right)^{-2}\hSigma_i\cdot\Delta^{-1}\sum_{i=1}^k\left(\hSigma_i+\rho I\right)^{-2}\hSigma_i-\sum_{i=1}^k\left(\hSigma_i+\rho I\right)^{-3}\hSigma_i\right)\right]\leq 0
\end{align*}
Let $A=\sum_{i=1}^k\left(\hSigma_i+\rho I\right)^{-3}\hSigma_i,$ $B=\sum_{i=1}^k\left(\hSigma_i+\rho I\right)^{-2}\hSigma_i$, then we only need to show
$$\tr\left[\Delta^{-2}\left(A-B^\top\Delta^{-1}B\right)\right]\geq 0.$$
Since $\Delta^{-2}$ is a positive semi-definite matrix, and for two positive semi-definite matrices $X = P^\top P$, $Y=Q^\top Q$ we have $\tr(XY)=\tr(P^\top PQ^\top Q)=\tr(PQ^\top QP^\top)\geq 0$, the remaining work is to show $A-B^\top\Delta^{-1}B\succeq 0$. Now, by using Lemma \ref{schur}, it enough to show 
$$M=
\begin{bmatrix}
A         & B\\
B^\top & \Delta
\end{bmatrix}\succeq 0.
$$
We can write $M$ as sum of matrices $M_i$, where
$$M_i=
\begin{bmatrix}
\left(\hSigma_i+\rho I\right)^{-3}\hSigma_i    & \left(\hSigma_i+\rho I\right)^{-2}\hSigma_i\\
\left(\hSigma_i+\rho I\right)^{-2}\hSigma_i & \left(\hSigma_i+\rho I\right)^{-1}\hSigma_i
\end{bmatrix}.
$$
By using Lemma \ref{schur} again, $M_i\succeq 0$ since
$$\left(\hSigma_i+\rho I\right)^{-3}\hSigma_i-\left(\hSigma_i+\rho I\right)^{-2}\hSigma_i\cdot\left(\hSigma_i+\rho I\right)\hSigma_i^{-1}\cdot\left(\hSigma_i+\rho I\right)^{-2}\hSigma_i=0.$$
Finally we have $M=\sum_{i=1}^kM_i\succeq 0$, i.e. $\psi'(\rho)\leq 0$. The special case when $\rho=0$ is of special interest. In this case, $\Delta=kI$ and we can simplify the derivative
$$\psi'(0)=\frac{2}{nk^2}\cdot\tr\left[\left(\sum_{i=1}^k\hSigma_i^{-1}\right)^2-k\sum_{i=1}^k\hSigma_i^{-2}\right].$$
By using the Cauchy-Schwarz inequality for positive semidefinite matrices
$$\tr(AB)\leq\sqrt{\tr(A^2)\tr(B^2)}\leq\frac{\tr(A^2)+\tr(B^2)}{2},$$
we can easily verify that $\psi'(0)\leq 0$, and equality holds if and only if $X_1^\top X_1=X_2^\top X_2=\cdots=X_k^\top X_k$.

\section{Numerical simulations}
\label{numsa}
\subsection{Relative efficiency for regression}

\begin{figure}
\begin{subfigure}{.4\textwidth}
  \centering
\includegraphics[scale=0.3]{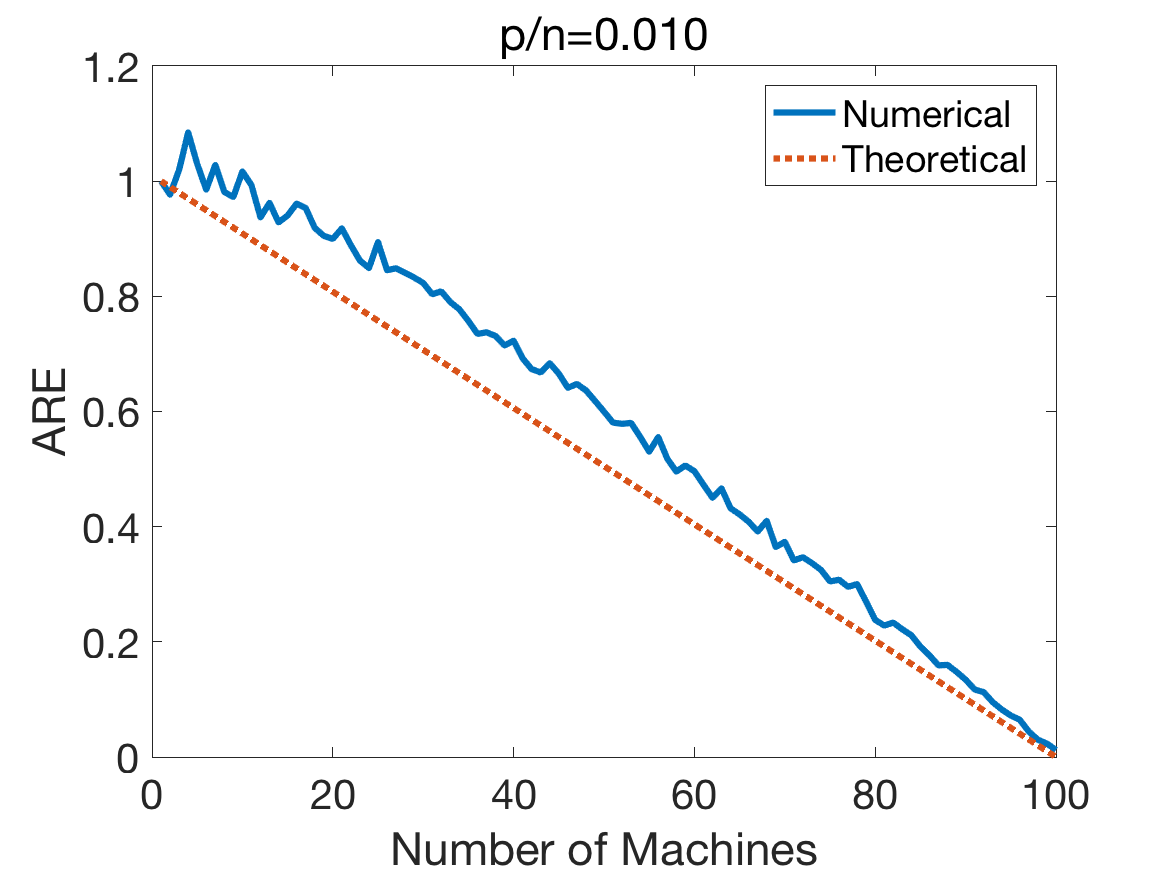}
\end{subfigure}
\begin{subfigure}{.4\textwidth}
  \centering
\includegraphics[scale=0.3]{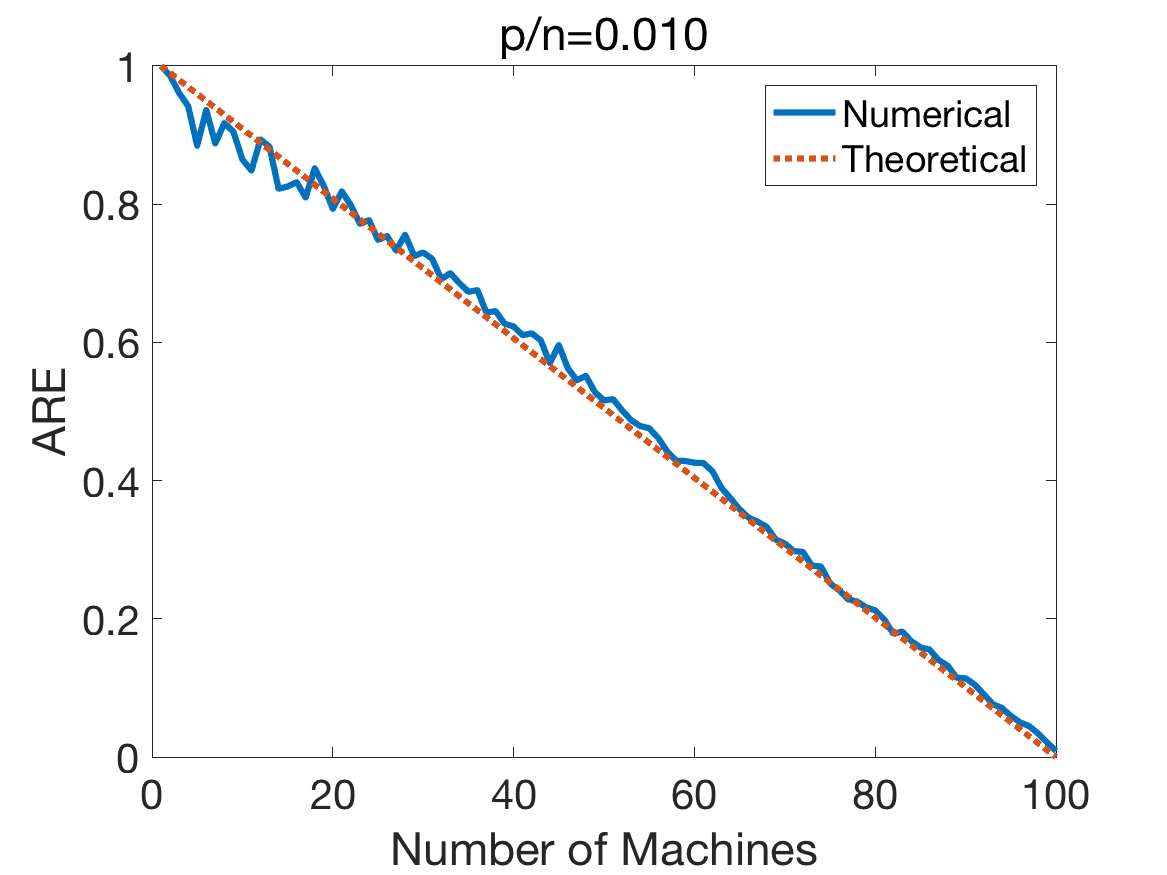}
\end{subfigure}
\caption{Relative efficiency in regression.}
\label{fig:reg}
\end{figure}

Figure \ref{fig:reg} shows a comparison of our theoretical formulas for ARE and realized relative efficiency in a regression simulation. Here we consider regression problems with $Y = X\beta+\ep$, where $X$ is $n\times p$ with iid standard Gaussian entries, $\beta=0$, and $\ep$ has iid standard Gaussian entries. We choose $n>p$, and for each value of $k$ such that $k<n/p$, we split the data into $k$ equal groups. We then show the results of the expression for the realized relative efficiency $\|\hbeta-\beta\|^2/\|\hbeta_{dist}-\beta\|^2$ compared to the theoretical ARE. We take $n=10,000$ and $p=100$.

We observe that the two agree closely.  However, there is more sampling variation than in the previously reported simulations, where we only compared the expected values of the relative efficiency to its asymptotic limit. In particular, the realized relative efficiency can be greater than unity. This is not a contradiction as our theoretical results only concern the expected values. However, we find that the simulations still match the theoretical results quite well.

% \section{Numerical simulations}
% \label{nums}

\subsection{Relative efficiencies for the elliptical model}

\begin{figure}
\begin{subfigure}{.45\textwidth}
  \centering
\includegraphics[scale=0.37]{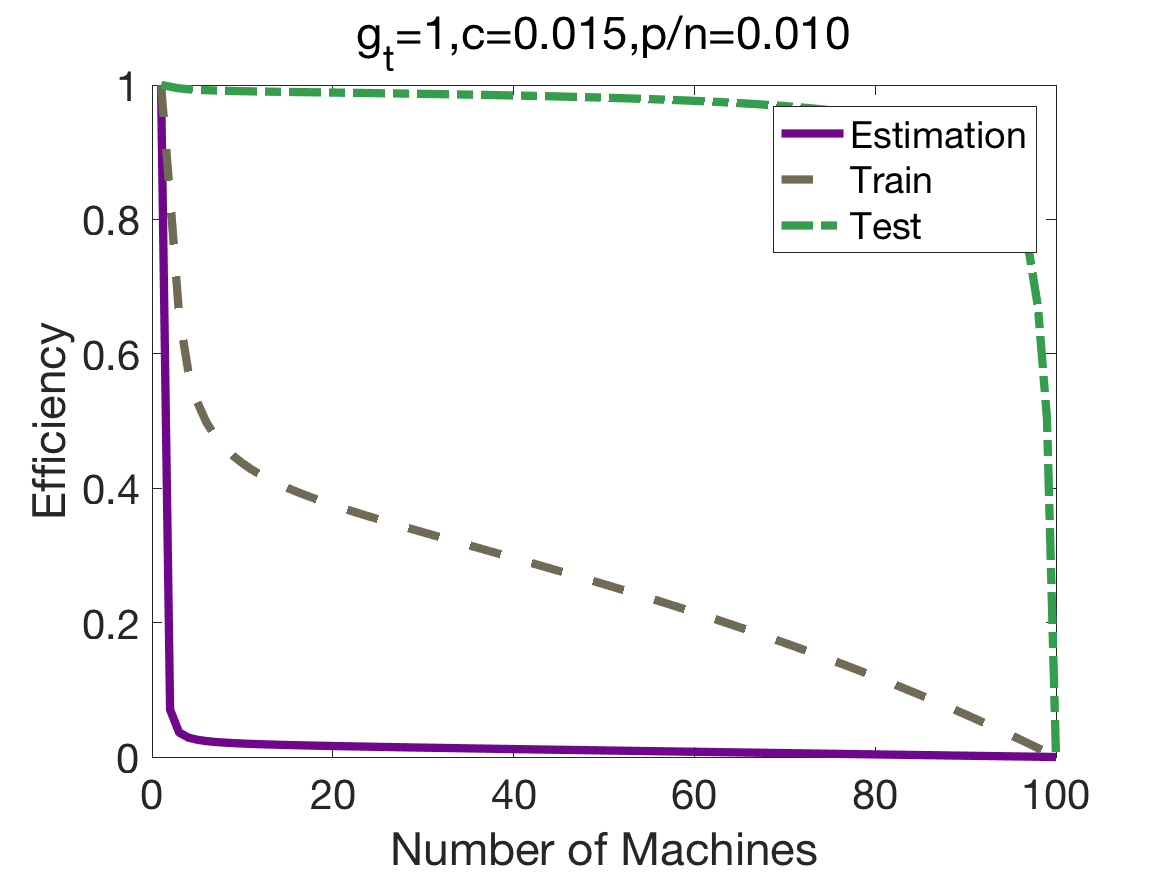}
\end{subfigure}
\begin{subfigure}{.45\textwidth}
  \centering
\includegraphics[scale=0.37]{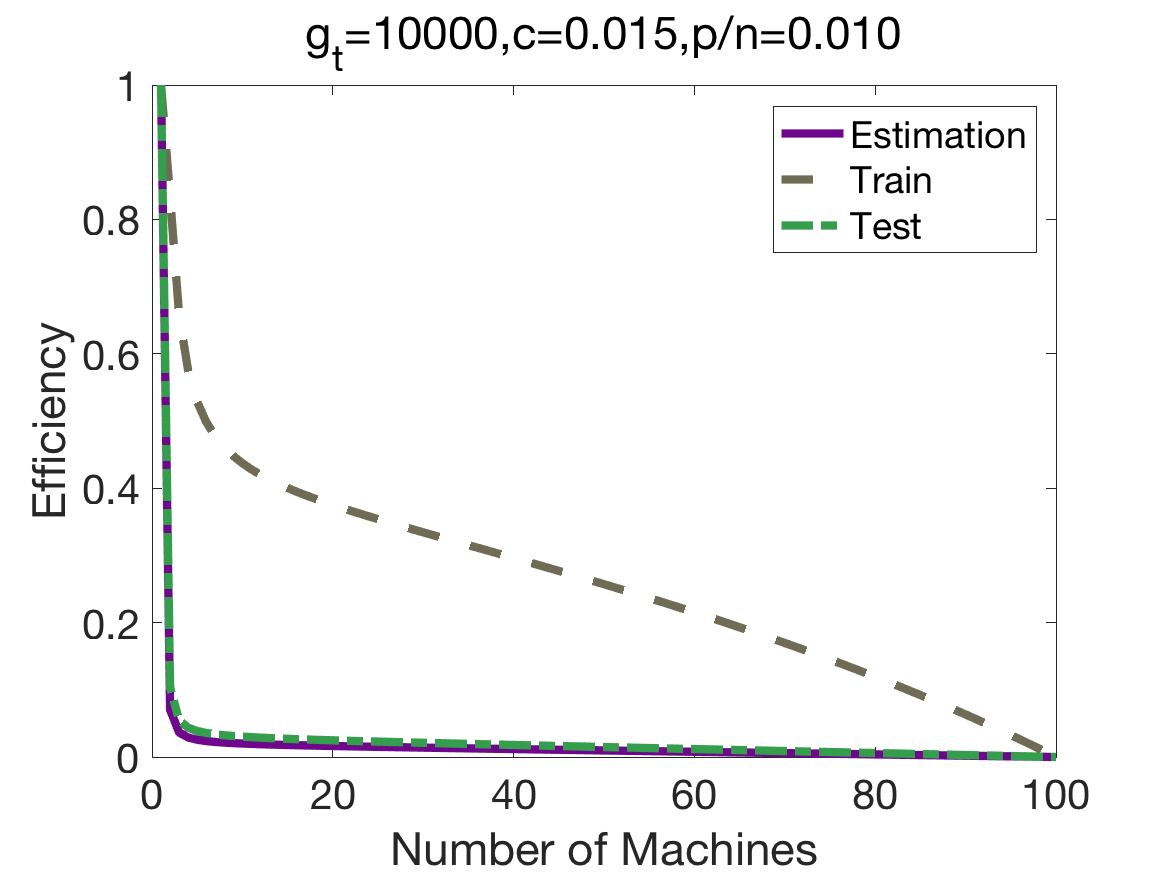}
\end{subfigure}
\caption{Relative efficiency for elliptical model.}
\label{fig:elliptic}
\end{figure}

For the elliptical model, we can also study the relation between different asymptotic efficiencies and show a plot similar to Figure \ref{ef-mp} for the Marchenko-Pastur model. Intuitively, we cannot expect a universal phenomenon in this situation since all efficiencies depend on the distribution $G$. Let us consider the worst-case example from Theorem \ref{ex2}. Figure \ref{fig:elliptic} shows that the asymptotic relative efficiency for out-of-sample prediction could be either very good or as bad as the ARE. 

In the first plot, we take $p=100,n=10000$, while $\alpha=10000$, and $c=0.015$. The test datapoint has magnitude $g_t=1$. In the second plot, we choose the same parameters but change the magnitude of the test datapoint to $g_t=10000$. Intuitively, when $g_t$ is large, the irreducible noise is negligible. Otherwise, the irreducible noise will make the problem easier. This is precisely what we observe in our figure, where in the first case, test error increases only a little (the efficiency is nearly unity), while in the second case test error increases a lot. We also expect this from our formula in Theorem \ref{OE_mp_thm}.

\section{Empirical data analysis}
\label{emp}

\begin{figure}
\begin{subfigure}{.45\textwidth}
  \centering
\includegraphics[scale=0.26]{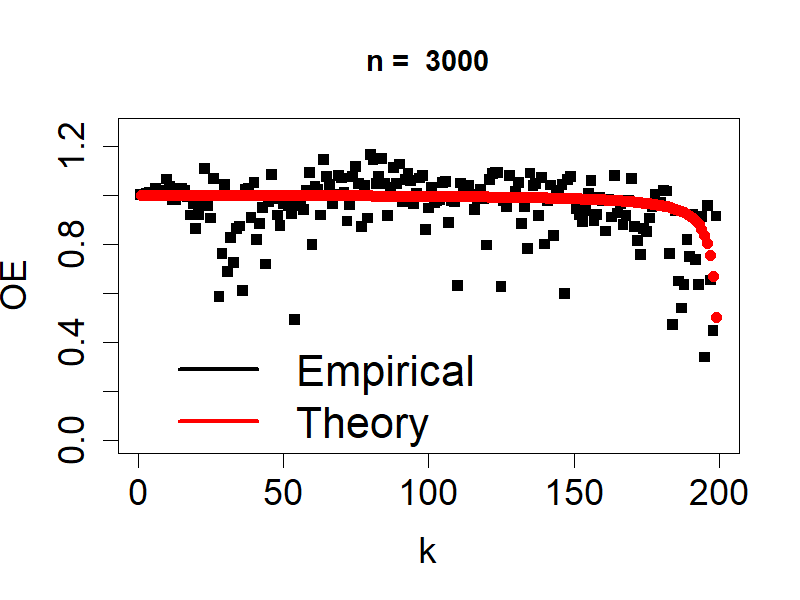}
\end{subfigure}
\begin{subfigure}{.45\textwidth}
  \centering
\includegraphics[scale=0.26]{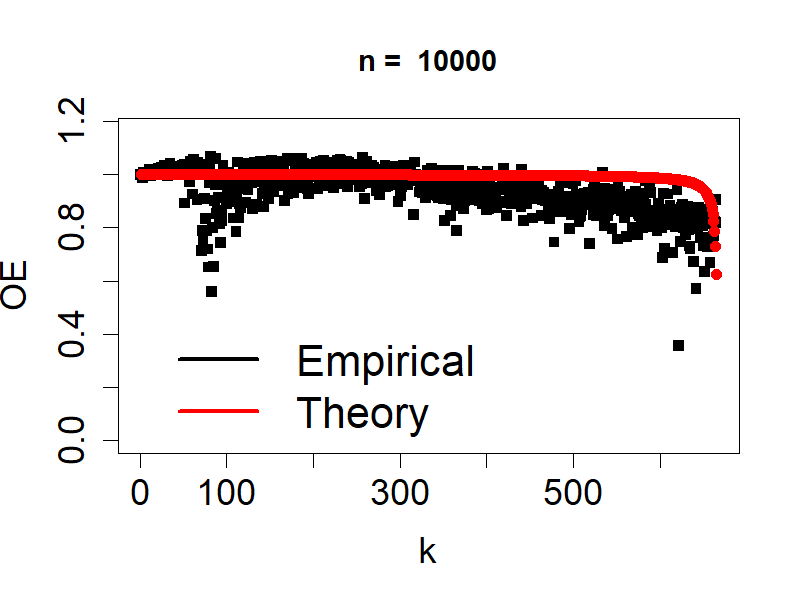}
\end{subfigure}
\caption{NYC flights data.}
\label{fig:ny}
\end{figure}

In this section we present an empirical data example to assess the accuracy of our theoretical results. Specifically, figure \ref{fig:ny} shows a comparison of our theoretical formulas for OE and actual out-of-sample prediction error (test error) on the NYC flights dataset \citep{nycflights13}. We observe a quite good match.

Specifically, we performed the following steps in our data analysis. We downloaded the flights data as included in the nycflights13 R package \citep{nycflights13}. We joined the separate datasets (weather, planes, and airlines). We omitted data points with missing entries. We removed one out of each pair of variables with absolute correlation higher than 0.8. This left a total of $N=60,448$ samples and $p=17$ variables. For $n=3000,10000$, we randomly sampled a training set of size $n$, and a non-overlapping test set of size also equal to $n$. The test set size does not have equal the training set size, and we only followed this protocol for simplicity.

We then fit linear regression estimators to this data in a global and distributed way. For the distributed version, we split the train data as equally as possible into $k$ subsets, for each $k\le n/p$. We then fit a linear regression to each subset, and took a weighted average with the optimal weights. We computed the test error of both the global and the distributed estimators over the test sample, and defined their ratio to be the empirical OE. We compared this to our theoretical formula for the OE, see figure \ref{fig:ny}. 

We observe a quite good match between the theoretical and empirical results. However, the empirical estimate of OE can be larger than unity. This is because of sampling noise. Our results show that $OE\le 1$, but only for the theoretical quantity where we have taken expectations. To get estimators with reduced variance, one could average over multiple Monte Carlo trials. However, those are beyond the our scope.

\end{document}